\documentclass[a4paper,10pt,reqno]{amsart}
\usepackage{amscd,amssymb,graphicx,color,epigraph}

\usepackage[shortalphabetic,initials]{amsrefs}

\usepackage{hyperref}

\def\Ext{\operatorname{Ext}}
\def\ev{\operatorname{ev}}
\def\oM{\overline{M}}

\def\oPic{\overline{\Pic}}

\def\bP{\Bbb P}
\def\bp{\bold p}

\def\bE{\Bbb E}
\def\dlog{\operatorname{dlog}}

\def\Mat{\operatorname{Mat}}

\def\gr{\operatorname{gr}}

\def\Sym{\operatorname{Sym}}

\def\Tr{\operatorname{Tr}}

\def\Re{\operatorname{Re}}

\def\U{\operatorname{U}}
\def\el{\operatorname{el}}

\def\Bun{\operatorname{Bun}}
\def\cBun{\operatorname{\it Bun}}

\def\cO{\Cal O}
\def\cm{\mathfrak m}

\def\codim{\operatorname{codim}}

\def\bZ{\Bbb Z}
\def\bC{\Bbb C}
\def\bQ{\Bbb Q}

\def\bR{\Bbb R}
\def\bA{\Bbb A}

\def\cD{\Cal D}
\def\cS{\Cal S}

\def\tS{\tilde S}

\def\oS{\overline{S}}

\def\tS{\tilde{S}}
\def\phi{\varphi}

\def\oN{\overline{N}}

\def\cD{\Cal D}

\def\cM{\Cal M}

\def\PGL{\operatorname{PGL}}
\def\GL{\operatorname{GL}}

\def\dlog{\operatorname{dlog}}
\def\div{\operatorname{div}}

\def\Res{\operatorname{Res}}

\def\cO{\Cal O}

\def\Sing{\operatorname{Sing}}

\def\Pic{\operatorname{Pic}}
\def\MHV{\operatorname{MHV}}
\def\cPic{\operatorname{\Cal P\mathit{ic}}}
\def\ocPic{\overline{\cPic}}

\def\Hom{\operatorname{Hom}}
\def\Cal{\mathcal}

\def\eps{\varepsilon}

\def\bLa{{\bold\Lambda}}
\def\bbLa{\overline{\bold\Lambda}}

\def\arrow{\mathop{\longrightarrow}\limits}

\def\Efi#1#2#3#4#5{\displaystyle
#1\!\!-\!\!#2
\!\!-\!\!#3
\!\!-\!\!#4
\hskip-24.2pt\lower4.5pt\hbox{${\scriptstyle|}
\hskip-3.35pt\lower6pt\hbox{$#5$}$}}

\def\Evia#1#2#3#4#5{\displaystyle
#1\!\!-\!\!#2
\!\!-\!\!#3
\hskip-24.2pt\lower4.5pt\hbox{${\scriptstyle|}
\hskip-3.35pt\lower6pt\hbox{$#4\!\!-\!\!\!-\!\!\!-\!\!$}$\hskip2.3pt${\scriptstyle|}
\hskip-3.35pt\lower6pt\hbox{$#5$}$}}

\def\Ezia#1#2#3#4{\displaystyle
#1\!\!-\!\!#2
\hskip-14.8pt\lower4.5pt\hbox{${\scriptstyle|}
\hskip-3.35pt\lower6pt\hbox{$#3\!\!-\!\!$}$\hskip2.3pt${\scriptstyle|}
\hskip-3.35pt\lower6pt\hbox{$#4$}$}}

\def\Efia#1#2#3#4#5#6{\displaystyle
#1\!\!-\!\!#2
\!\!-\!\!#3
\!\!-\!\!#4
\hskip-24.2pt\lower4.5pt\hbox{${\scriptstyle|}
\hskip-3.35pt\lower6pt\hbox{$#5$}$\hskip5.7pt${\scriptstyle|}
\hskip-3.35pt\lower6pt\hbox{$#6$}$}}

\def\Esi#1#2#3#4#5#6{\displaystyle
#1\!\!-\!\!#2
\!\!-\!\!#3
\!\!-\!\!#4\!\!-\!\!#5
\hskip-24.2pt\lower4.5pt\hbox{${\scriptstyle|}
\hskip-3.35pt\lower6pt\hbox{$#6$
\lower3pt\hbox{\ }}$}}

\def\Esia#1#2#3#4#5#6#7{\displaystyle
#1\!\!-\!\!#2
\!\!-\!\!#3
\!\!-\!\!#4\!\!-\!\!#5
\hskip-24.2pt\lower4.5pt\hbox{${\scriptstyle|}
\hskip-3.35pt\lower6pt\hbox{$#6$\hskip-3.8pt\lower4.5pt\hbox{${\scriptstyle|}
\hskip-3.35pt\lower6pt\hbox{$#7$}$}}
\lower3pt\hbox{\ }$}}

\def\Ese#1#2#3#4#5#6#7{\displaystyle
#1\!\!-\!\!#2
\!\!-\!\!#3
\!\!-\!\!#4\!\!-\!\!#5\!\!-\!\!#6
\hskip-33.6pt\lower4.5pt\hbox{${\scriptstyle|}
\hskip-3.35pt\lower6pt\hbox{$#7$
\lower3pt\hbox{\ }
}$}}

\def\Esea#1#2#3#4#5#6#7#8{\displaystyle
#1\!\!-\!\!#2
\!\!-\!\!#3
\!\!-\!\!#4\!\!-\!\!#5\!\!-\!\!#6\!\!-\!\!#7
\hskip-33.6pt\lower4.5pt\hbox{${\scriptstyle|}
\hskip-3.35pt\lower6pt\hbox{$#8$
\lower3pt\hbox{\ }
}$}}

\def\Eei#1#2#3#4#5#6#7#8{\displaystyle
#1\!\!-\!\!#2
\!\!-\!\!#3
\!\!-\!\!#4\!\!-\!\!#5\!\!-\!\!#6\!\!-\!\!#7
\hskip-43.2pt\lower4.5pt\hbox{${\scriptstyle|}
\hskip-3.35pt\lower6pt\hbox{$#8$
\lower3pt\hbox{\ }
}$}}

\def\Eeia#1#2#3#4#5#6#7#8#9{{\displaystyle
#1\!\!-\!\!#2
\!\!-\!\!#3
\!\!-\!\!#4\!\!-\!\!#5\!\!-\!\!#6\!\!-\!\!#7\!\!-\!\!#8
\hskip-52.2pt\lower4.5pt\hbox{${\scriptstyle|}
\hskip-3.35pt\lower6pt\hbox{$#9$
\lower3pt\hbox{\ }
}$}}}

\def\arrow{\mathop{\longrightarrow}\limits}

\def\oM{\overline{M}}
\def\os{\overline{S}}
\def\pt{\operatorname{pt}}
\def\ocM{\overline{\Cal M}}

\def\oQ{\overline{Q}}

\def\tS{\tilde{S}}
\def\ts{\tS}
\def\ts7{\tilde{S}_7}

\def\cN{\Cal N}
\def\cQ{\Cal Q}

\def\bP{\Bbb P}
\def\bL{\Bbb L}

\def\cC{\Cal C}

\def\bE{\Bbb E}

\def\Bl{\operatorname{Bl}}
\def\dlog{operatorname{dlog}}

\def\RHom{\operatorname{RHom}}
\def\Sym{\operatorname{Sym}}

\def\Hom{\operatorname{Hom}}

\def\cO{\Cal O}
\def\cS{\Cal S}
\def\cD{\Cal D}

\def\Pic{\operatorname{Pic}}

\def\codim{\operatorname{codim}}

\def\bZ{\Bbb Z}
\def\bC{\Bbb C}
\def\bQ{\Bbb Q}

\def\bR{\Bbb R}
\def\bA{\Bbb A}

\def\cD{\Cal D}
\def\cS{\Cal S}

\def\tS{\tilde S}

\def\oS{\overline{S}}

\def\tS{\tilde{S}}

\def\oN{\overline{N}}

\def\cM{\Cal M}

\def\PGL{\operatorname{PGL}}

\def\dlog{\operatorname{dlog}}

\def\Res{\operatorname{Res}}
\def\Ker{\operatorname{Ker}}

\def\Sing{\operatorname{Sing}}
\def\Cal{\mathcal}

\def\Pic{\operatorname{Pic}}
\def\Ku{\operatorname{Kum}}
\def\Kt{\operatorname{K3}}

\def\oN{\overline{N}}
\def\os7p{\oS_7'}
\def\os7{\oS_7}
\def\on6{\oN_6}
\def\n6{\oN_6}

\def\cL{\Cal L}
\def\cG{\Cal G}

\def\dP{\bold{dP}_4}
\def\dPf{\bold{dP}_5}

\newtheoremstyle{mystyle}{}{}{\itshape}{}{\scshape}{.}{ }{}
\theoremstyle{mystyle}

\swapnumbers
\newtheorem{Theorem}{Theorem}[section]

\newtheorem{Proposition}[Theorem]{Proposition}
\newtheorem{Lemma}[Theorem]{Lemma}
\newtheorem{Corollary}[Theorem]{Corollary}
\newtheorem{Claim}[Theorem]{Claim}

\newtheoremstyle{myreview}{}{}{}{}{\scshape}{.}{ }{}
\theoremstyle{myreview}

\newtheorem{Definition}[Theorem]{Definition}

\newtheorem{Example}[Theorem]{Example}

\newtheorem{Remark}[Theorem]{Remark}

\newtheorem{Notation}[Theorem]{Notation}

\newtheorem{Assumption}[Theorem]{Assumption}
\newtheorem{Amplification}[Theorem]{Amplification}
\newtheorem{Summary}[Theorem]{Summary}
\newtheorem{Review}[Theorem]{}

\newcounter{et}[Theorem]
\def\cooltag{\tag{\arabic{section}.\arabic{Theorem}.\arabic{et}}\addtocounter{et}{1}}

\begin{document}
\title{Scattering amplitudes of stable curves}  
\author{Jenia Tevelev}

\address{Department of Mathematics and Statistics, University of Massachusetts Amherst, 710 North Pleasant Street, Amherst, MA 01003, USA}
\email{tevelev@umass.edu}



\begin{abstract}
Hypertree divisors on the 
moduli space 
of stable rational curves
 were introduced by Castravet and Tevelev in \cite{CT_Crelle}.
Their equations appear
as numerators of scattering amplitude forms 
for  $n$ 
particles in  $N=4$ 
Yang--Mills theory in the work of Arkani-Hamed, Bourjaily, Cachazo, Postnikov and Trnka \cite{MHV}. 
Rather than being a  coincidence, this is just the tip of the iceberg of an exciting relation between algebraic geometry and high energy physics.
We interpret leading singularities of 
scattering amplitudes of massless particles as  {\em probabilistic Brill--Noether theory}:
the study of statistics
of images of $n$ marked points under a random meromorphic function
uniformly distributed with respect to the translation-invariant volume form of the Jacobian.
We focus on the maximum helicity violating case, which 
leads to a beautiful physics-inspired geometry for various classes of complex algebraic curves: smooth, stable, hyperelliptic, real algebraic, etc.
\end{abstract}

\maketitle

\section{Introduction}

\begin{Review}
In physics, momenta 
of $n$ particles satisfy the momentum conservation law.
$\bold p_1+\ldots+\bold p_n=0$.
In~mathematics,  meromorphic forms $\omega$ 
with simple poles $p_1,\ldots,p_n$ (log forms) on 
a compact Riemann surface
satisfy the residue theorem:
\vskip-2pt
\begin{figure}[htbp]
\makebox[0.32\textwidth]{\vbox{$$\bold p_1+\ldots+\bold p_n=0$$\\ \ \\ \ }}
\makebox[0.25\textwidth]{\vbox{\includegraphics[height=0.8in]{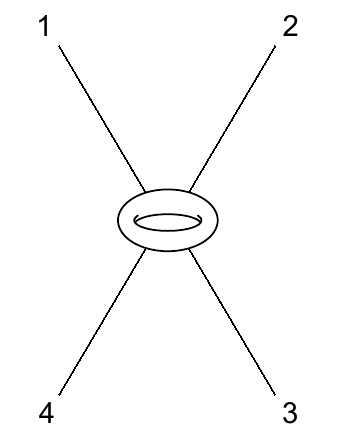}}}
\ \ \ \makebox[0.32\textwidth]{\vbox{$$\Res_{p_1}\omega+\ldots+\Res_{p_n}\omega=0$$\\ \ \\ \ }}
\end{figure}
\vskip-2pt
\noindent
Let $C$ be a stable  projective algebraic  curve with marked points $p_1,\ldots,p_n$. The exact sequence 
 $0\to H^0(C,\omega_C)\to H^0(C,\omega_C(p_1+\ldots+p_n))\arrow^{\Res}\bC^n \arrow^{\Sigma}\bC\to0$
 (valid for $n\ge1$)
 shows that we can always view momenta of $n$  one-dimensional particles satisfying the momentum conservation law
as residues of a form $\omega\in H^0(C,\omega_C(p_1+\ldots+p_n))$. For a smooth curve $C$, 
ambiguities in the choice of $\omega$ can be eliminated by fixing its integrals along periods.
For a nodal curve $C$, a local section of $\omega_C(p_1+\ldots+p_n)$ at each node $q_1,\ldots,q_r\in C$ can be identified with a log form on
the normalization~$C^\nu$ which has opposite residues at the points $q_i^-$, $q_i^+$ mapping to the node $q_i$. 
These residues can be viewed as momenta of ``internal'' or "on-shell" particles satisfying a momentum conservation law 
for each of the $s$ irreducible components of $C$, as follows from an exact sequence
$$0\to H^0(C^\nu,\omega_{C^\nu})\to H^0(C,\omega_C(p_1+\ldots+p_n))\arrow^{\Res}\bC^{n+r} \arrow^{\Sigma_1,\ldots,\Sigma_s}\bC^s\to0.$$
When all components of $C$ are rational, $H^0(C^\nu,\omega_{C^\nu})=0$ and the datum of a log-form on $C$ is equivalent to the datum of external and internal momenta
satisfying momentum conservation laws (one law for each irreducible component of~$C$). Equivalently, 
this is the datum of one ``edge variable'' for each internal (corresponding to a node) and external
(corresponding to a marked point) 
edge of the dual graph of $C$ with a conservation law for each vertex of the dual graph. 
\end{Review}

\begin{Review}
Imagine a random process that gives 
a vector of log forms  on $(C; p_1,\ldots,p_n)$.
What will be the probability  distribution of 
residues? We will study the following 
 simple random process.
The~first step  is a random tensor product factorization
\begin{equation}\label{afqrgqrb}
\omega_C(p_1+\ldots+p_n)=L\otimes\tilde L,
\cooltag\end{equation}
into two line bundles, $L$ of degree $d$ and $\tilde L$ of degree $\tilde d$ 
(and $d+\tilde d=2g-2+n$).
The~choice of $L$ is a choice of a  point in $\Pic^dC$
and, when $C$ is smooth,  every connected component of the Picard group 
has a uniform probability measure,  the volume form invariant under translations
by the subgroup $\Pic^0C$ of topologically trivial line bundles.
If~$C$ is not smooth  then
$\Pic^0C$ contains non-compact factors~$\bC^*$ and probabilistic interpretation is less straightforward.
In fact the case most related to physics literature is when all irreducible components of $C$ are rational curves
and  $\Pic^0C\simeq(\bC^*)^g$.
Non-compactness will not be an issue for us.
\end{Review}

\begin{Review}
The second step is a choice of sections $s_\alpha\in H^0(C,L)$, $\tilde s_{\tilde\alpha}\in H^0(C,\tilde L)$
for some indices 
$\alpha\in I$ and 
$\tilde\alpha\in\tilde I$.
Tensoring these sections gives a matrix of log forms
$$\omega_{\alpha\tilde\alpha}=s_\alpha\otimes \tilde s_{\tilde\alpha}\in H^0(C,\omega_C(p_1+\ldots+p_n)).$$
We focus on the case $|I|=|\tilde I|=2$, which is  related to the $\cN=4$ Yang--Mills theory.
Let~$\bL\subset \Mat_{2,2}$ be the subvariety of matrices of rank at most $1$.
The projectivization  $\bP(\bL)$ is a quadric in $\bP^3$ isomorphic to $\bP^1\times\bP^1$.
Let
$$Y=\{\bp_1,\ldots,\bp_n\in\bL\quad|\quad \bp_1+\ldots+\bp_n=0\}\subset \bL^n.$$
The space $\Mat_{2,2}$ can be viewed as the complexified $4$-dimensional Minkowski space 
and $\bL$ as the complexified light cone.
Thus $Y$ parametrizes complexified momenta of $n$ massless particles in $4$ dimensions
satisfying the momentum conservation law.
We observe that  
$\left(\Res_{p_1}\omega_{\alpha\tilde\alpha},\ldots,\Res_{p_n}\omega_{\alpha\tilde\alpha}\right)\in Y$, since
each matrix 
$\Res_{p_i}\omega_{\alpha\tilde\alpha}$ is proportional  to the matrix $(s_\alpha(p_i)\tilde s_{\tilde\alpha}(p_i))$
after choosing trivializations.
So this matrix has rank at most one and  the residues of $\omega_{\alpha\tilde\alpha}$
(at marked points and nodes) can be viewed as momenta of massless (external and internal) particles.

This construction is reversible, at least in the most physically relevant case.
Namely, suppose all components of $C$ are rational and the dual graph is $3$-valent.
Given massless momenta  $\bold p_1,\ldots,\bold p_n$ (external) and  $\bold q_1,\ldots,\bold q_r$ (internal)
satisfying momentum conservation laws (one for each vertex of the dual graph), suppose
 that no two momenta adjacent to the same vertex are proportional.
Then there exists a factorization \eqref{afqrgqrb} and
sections $s_\alpha\in H^0(C,L)$, $\tilde s_{\tilde\alpha}\in H^0(C,\tilde L)$ for $\alpha,\tilde\alpha=1,2$
such that the momenta are  the residues of the $2\times 2$ matrix of log forms $(s_\alpha\otimes \tilde s_{\tilde\alpha})$.\footnote{Indeed, let $(\omega_{\alpha\tilde\alpha})$ be  the $2\times 2$ matrix of log forms that corresponds to the momenta.
Since the restriction of $\omega_C(p_1+\ldots+p_n)$ to every irreducible component of $C$ has degree $1$,
the sections $\omega_{\alpha\tilde\alpha}$ globally generate the line bundle $\omega_C(p_1+\ldots+p_n)$ and therefore
give a morphism $C\to\bP^3$ such that every irreducible component of $C$ maps to a line. By assumption, images of marked points and nodes belong to the quadric
$\bP(\bL)$. Since a line intersecting a quadric in $3$ different points belongs to it, the image of $C$ lies on the quadric, i.e.~the matrix of log forms has rank~$1$.
The line bundles $L$ and $\tilde L$ are pull-backs of $\cO_{\bP^1\times\bP^1}(1,0)$ 
and $\cO_{\bP^1\times\bP^1}(0,1)$, where $\bP(\bL)\cong \bP^1\times\bP^1$.}
\end{Review}

\begin{Review}
We would like to focus on situations where there are no constraints on external and internal momenta,
i.e.~when a dense subset of points of $Y$ can be obtained as residues up to a finite ambiguity.
This requires imposing the following condition:\footnote{
By the Riemann--Roch theorem,
$H^0(C, L)$ and $H^0(C,\tilde L)$ have expected dimensions
$r=d+1-g$ and $\tilde r=\tilde d+1-g$.
Rescaling all $s_\alpha$ by $z\in\bC^*$ and $\tilde s_{\tilde \alpha}$ by $z^{-1}$
gives the same forms $\omega_{\alpha\tilde\alpha}$ and so the same residues.
Thus the variety of choices $(L,\tilde L,s_\alpha,s_{\tilde \alpha})$ modulo $\bC^*$ 
has expected dimension
$g+2r+2\tilde r-1=g+2n-1$. On the other hand, $Y$ is a variety of dimension $3n-4$
(assuming $n\ge 4$).
Thus the condition on dimensions is $g+2n-1=3n-4$, which is equivalent to \eqref{mnabdvfb}.
}
\begin{equation}\label{mnabdvfb}
n=g+3.
\cooltag\end{equation}
\end{Review}

\begin{Review}
We would like to study distribution of  residues of the matrix $\omega_{\alpha\tilde\alpha}$ 
but  we run into a  problem: there is no
natural probability measure 
on the space of sections $s_\alpha$, $\tilde s_{\tilde\alpha}$ unless we choose some extra data
such as a hermitian metric on $L$.
Instead of making a choice of a specific 
probability measure, we can design an experiment with outcomes 
that don't depend on sections
by introducing a rational  (i.e.~defined on an open subset) map
$\Lambda:\,Y\dashrightarrow M$
to some algebraic variety $M$.
If $\Lambda$ is independent of sections $s_\alpha$ and $\tilde s_{\tilde\alpha}$
then it descends to a rational map 
$\bLa:\,\Pic^dC\dashrightarrow M$
and describing probability measure on $M$ becomes a well-posed problem.
Concretely, projecting the quadric $\bP(\bL)\cong\bP^1\times\bP^1$ to the first or the second factor gives rational maps 
$\bL^n\dashrightarrow(\bP^1)^n$ known as spinor variables in physics.
Taking a quotient by the $\PGL_2$ action gives a rational map $(\bP^1)^n\dashrightarrow M_{0,n}$ to
the moduli space of $n$ distinct marked points on $\bP^1$.
To summarize, we have two rational maps
$$\Lambda,  \tilde\Lambda:\, Y\hookrightarrow\bL^n\dashrightarrow(\bP^1)^n\dashrightarrow M_{0,n}.$$
The image of a point 
$\left(\Res_{p_1}\omega_{\alpha\tilde\alpha},\ldots,\Res_{p_n}\omega_{\alpha\tilde\alpha}\right)$
under the map $\Lambda$ (resp.,  $\tilde\Lambda$)
is given by an $n$-tuple $([s_1(p_i):s_2(p_i)])_{i=1\ldots,n}$ 
(resp.,~$([\tilde s_1(p_i):\tilde s_2(p_i)])_{i=1\ldots,n}$) of points in $\bP^1$ modulo $\PGL_2$.
One can further  compactify $M_{0,n}$ by the Grothendieck--Knudsen 
moduli space $\oM_{0,n}$ of stable rational curves or choose a different copactification.
\end{Review}

Requiring that $\Lambda$ descends to $\bLa:\,\Pic^dC\dashrightarrow M_{0,n}$ 
obviously means that the expected dimension of $H^0(C,L)$, which in the physical context is known as helicity, 
should be equal to $2$. By Riemann--Roch, this is equivalent to the requirement
\begin{equation}\label{ffffffasdh}
d=g+1.
\cooltag\end{equation}

\begin{Definition}\label{KSLJDfksjdg} 
Let $(C;p_1,\ldots,p_n)$ be a  stable curve of genus~$g$ with $n=g+3$ marked points.
Fix a multidegree vector  $\vec d=(d_i)$ (one degree $d_i$ for each irreducible component of $C$)
such that $d=\sum d_i$ satisfies \eqref{ffffffasdh}. This gives a
connected component $\Pic^{\vec d}C\subset\Pic^dC$.
We say that a pair $(C,\vec d)$ is {\em a maximum helicity violating (MHV) curve} 
if, for a generic line bundle $L\in \Pic^{\vec d}C$, we have
\begin{enumerate}
\item 
$L$ is not special, i.e.~has exactly two linearly independent global sections: 
$$H^0(C,L)=\bC^2,\quad H^1(C,L)=0.$$
\item 
The evaluation map $\alpha:\,H^0(C,L)\otimes\cO_C\to L$ is surjective, equivalently
$$\phi_L:\,C\dashrightarrow\bP^1$$ 
is a morphism (which is automatic if $C$ is smooth) and $L\simeq\phi_L^*\cO_{\bP^1}(1)$.
\item The rational {\em scattering amplitude map} 
$$
\bLa:\,\Pic^{\vec d}C\dashrightarrow M_{0,n},\qquad
L\mapsto \left(\phi_L(p_1),\ldots,\phi_L(p_n)\right)$$
is dominant at $L$, or equivalently generically finite.
\end{enumerate}
\end{Definition}

\begin{Definition}
The {\em scattering amplitude form} of an MHV curve 
is a unique (up to a constant multiple) non-zero $\Pic^{\vec 0}C$-invariant holomorphic $g$-form 
$$A\in H^0(\Pic^{\vec d}C,\Omega^g)$$
viewed as a multi-valued meromorphic form on~$M_{0,n}$. 
\end{Definition}

\begin{Summary}
We start with a stable  curve $C$ of genus $g$ with $n=g+3$ marked points.
The~MHV experiment is a choice of a random line bundle $L$ of degree $d=g+1$ or,
equivalently, a random meromorphic function $\phi_L:\,C\to\bP^1$ of degree~$d$.
When $C$ is smooth, 
the probability distribution of $L$ is uniform with respect to the translation-invariant volume form on the Jacobian.
We study probability measure on $M_{0,n}$
that gives statistics of points $\phi_L(p_1),\ldots,\phi_L(p_n)\in\bP^1$.
This gives
a family of (complexified) probability measures that depend on $4g$ complex parameters
describing the input curve $(C, p_1, ..., p_{n})\in\oM_{g,n}$.
\end{Summary}

\begin{Remark}
It~is clearly superfluous to keep the notation $g$, $d$, $n$ for quantities  related by the equations 
\eqref{mnabdvfb}, \eqref{ffffffasdh}
throughout this paper.
Our excuse is that they are associated with the main players, the curve $C$, its Picard group  and $M_{0,n}$.  
\end{Remark}

\begin{Example}\label{asgsrgr}
An MHV curve of genus $0$ is just $\bP^1$ with three marked points. In~this case $d=1$ and $\Pic^1C$ is a point, namely the line bundle $L=\cO_{\bP^1}(1)$.
The~map $\phi_L:\,\bP^1\to\bP^1$ is an isomorphism 
which maps $p_1,p_2,p_3$ to three distinct points, which can be moved to the points $0,1,\infty$ by applying the $\PGL_2$ action.
It~follows that the scattering amplitude map $\bLa$ in genus $0$ is simply
$$\pt=\Pic^1\bP^1\arrow^{\bLa} M_{0,3}=\pt.$$
\end{Example}

\begin{Notation}
A connected component  $\Pic^{\vec d}C$ of the Picard group is determined by multidegrees 
$d_s=\deg L|_{C_s}$, one for each  irreducible component $C_s\subset C$.
We ~have $d=\sum d_s$. We draw an MHV curve
as an {\em on-shell diagram}, a dual graph of~$C$ with marked points as exterior legs.
\begin{figure}[htbp]
\includegraphics[width=\textwidth]{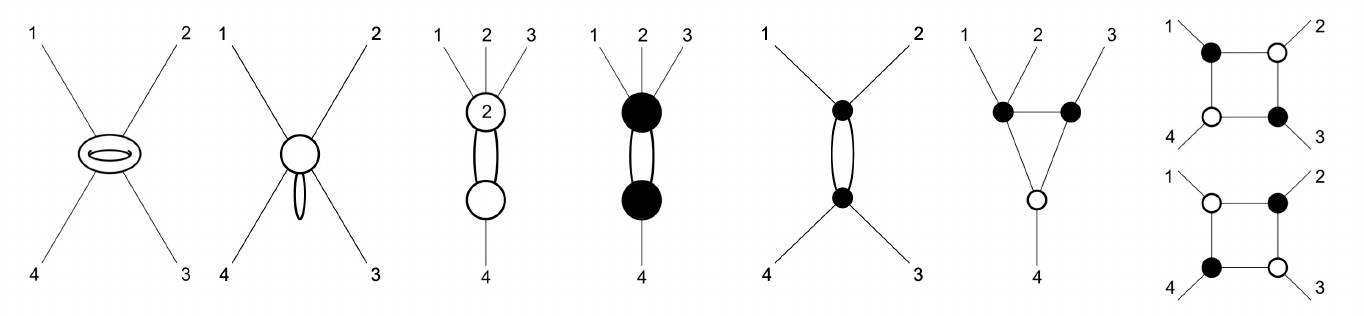}
\caption{On-shell diagrams of MHV curves of genus $1$}\label{SDgSg}
\end{figure}
Vertices of the diagram correspond to irreducible components of $C$ and interior
edges correspond to nodes.
Genus zero components are drawn as circles,
components of positive genus~$g_s$ are decorated with $g_s$ holes.
Components where the line bundle has degree $d_s=0$ are left blank,
components of degree $d_s=1$ are shaded black, and if $d_s>1$ then we  just write $d_s$ in or next to the component.
If the curve is irreducible then we don't indicate the degree at all since it's always equal to $g+1$.
The advantage of the on-shell diagram over the dual graph of the stable curve is that it records the multidegrees $\vec d$ of~$L$. In Theorem~\ref{wgearhstjsrjrj}, we will show that the locus of MHV curves is open.
In particular, the set of possible on-shell diagrams is closed under edge contractions.
\end{Notation}

\begin{Example}\label{rbqerhqer}
In genus $1$ there are many possibilities illustrated in Figure~\ref{SDgSg}. 
Here $d=2$ and $n=4$. So the map $\phi_L:\,C\to\bP^1$ is a double cover 
and~the scattering amplitude map $\bLa$ takes $L\in\Pic^2C$ to the cross-ratio of four points 
$$\phi_L(p_1),\ \phi_L(p_2),\ \phi_L(p_3),\ \phi_L(p_4)\in\bP^1.$$
It turns out that if $C$ is a smooth elliptic curve then $\bLa:\,\Pic^2C\to\bP^1$ is  a double cover
and the scattering amplitude form $A$ is an integrand  of an elliptic integral.
\end{Example}

The calculations in genus $0$ and $1$ motivate the following theorem\footnote{
In the time since the paper appeared on the arXiv for the first time, many further interpretations and generalizations of this 
``Tevelev degree'' have been investigated, see e.g.~\cite{Gen7,Gen6,Gen1,Gen5,Gen2,Gen3,Gen4}.}, which is
explored from different points of view in subsequent sections:

\begin{Theorem}\label{2^gTheorem}
The scattering amplitude map $\bLa$ has degree $2^g$ for a general curve $C$.
\end{Theorem}

\begin{Review}[Contents of~\ref{khgckhfckhfc}] 
One of the goals of the paper is to {\em localize}  $2^g$ line bundles
in the preimage of a general point of $M_{0,n}$,
i.e. to describe branches of the multi-valued inverse function~$\bLa^{-1}$.
In Theorem~\ref{efvwb} we  show that every smooth curve  is an MHV curve.
We use the theory of special divisors \cite{ACGH}, which was classically developed to 
study the images of Abel--Jacobi maps, in particular the focus was on the $d=g-1$ case.
The MHV regime $d=g+1$ is just as rich and exciting.
We~introduce various special divisors in $\Pic^{g+1}C$  as well as the {\em planar locus} $W$  that parametrizes presentations of the curve $C$
as a nodal plane curve of degree $g+1$. 
\end{Review}

\begin{Review}[Contents of~\ref{zdfbsdh}]
Given that all smooth curves are MHV curves, one expects a simple classification of stable MHV curves
but this is quite a delicate question, which we start to investigate in this section.
These results are not central to the paper but they are used throughout.
If~the curve $C$ has a node such that removing it separates $C$ into two connected components (i.e.~the on-shell diagram
is not $2$-connected) then $C$ is not an MHV curve -- the only MHV curves with compact Jacobians
are smooth curves.
This~is  known as a {\em one-channel factorization} in physics.
If the on-shell diagram is $2$-connected but not $3$-connected 
(more precisely, if $C$ has a two-channel factorization,
see Definition~\ref{FbfhFHF}) then
$C$ is separated into two components by removing two nodes, and
information about the scattering amplitude can be read from the components, see
Theorems~\ref{safbsfhsfn} and~\ref{ADvAeg}.
\end{Review}

\begin{Review}[Contents of~\ref{zdadthethfbsdh}] 
To study families of MHV curves and scattering amplitudes,
we 
introduce the 
universal
scattering amplitude map 
$$\bLa:\ \cPic^{MHV}\cC\dashrightarrow M_{0,n}$$
over the
locus of MHV curves $\ocM^{\MHV}_{g,n}\subset\ocM_{g,n}$.
We use two open substacks in the stack of quasi-maps: moduli of
stable quotients \cite{MOP} and moduli of presentations of slope-stable line bundles.
We find a convenient polarization for MHV curves
and compactify $\cPic^{MHV}\cC$ by a projective family 
of compactified Jacobians. 
\end{Review}

\begin{Review}[Example~\ref{rbqerhqer} --  continued]
While a smooth elliptic curve 
$E$  degenerates into a wheel $C$ of {\em four} projective lines,  $\Pic^2E$ 
degenerates into $\oPic^{MHV}C$, a wheel of {\em two} projective lines, the  MHV components $\Pic^{0,1,0,1}$ and $\Pic^{1,0,1,0}$ 
represented by stacked on-shell diagrams on the right side of Figure~\ref{SDgSg}.
Each of these components has the same  amplitude form,
a phenomenon called {\em ``square move''} by physicists.
For any of the curves in Figure ~\ref{SDgSg},
the fibre of the universal scattering amplitude map is 
a ramified double cover of $\bP^1$ by an irreducible curve of arithmetic genus $1$ 
except for the four-wheel, when it becomes a reducible $2:1$ cover 
$$\oPic^{MHV}C=\bP^1\cup\bP^1\arrow^{2:1}\bP^1=\oM_{0,4}.$$
\end{Review}

\begin{Review}[Contents of~\ref{sdvqefv}] 
Scattering amplitude maps and forms  of hyperelliptic curves give a new perspective on the theory of parabolic vector bundles of rank~$2$ on~$\bP^1$.
A~hyperelliptic curve $C$ is given by the equation $y^2=f(z)$, where $f$ is a polynomial of degree $2g+2$ or $2g+1$ without multiple roots.
This gives a double cover map 
$$\phi_h:\,C\to\bP^1,\quad  (z,y)\mapsto z.$$
Marked points $p_1,\ldots,p_n$ project to points $z_1,\ldots,z_n\in\bP^1$, which we assume are different.
In the study of pointed hyperelliptic curves it is often assumed that all marked points are Weierstrass points (the roots of $f(z)$)
but in our approach the marked points are  decoupled from the Weierstrass points.
We show in 
Lemma~\ref{SSRHS}
that the scattering amplitude map of hyperelliptic curves
factors as follows:
$$\bLa:\,\Pic^{g+1}C\arrow^{\bbLa} \Bun(\bP^1;z_1,\ldots,z_n)\mathop{\dashrightarrow}\limits^{\Xi} M_{0,n},$$
where $\Bun(\bP^1;z_1,\ldots,z_n)$ is the moduli stack of parabolic rank $2$ bundles with trivial determinant,
the map $\bbLa$ associates to a line bundle $L\in \Pic^{g+1}C$ its push-forward $(\phi_h)_*L$
with parabolic lines determined by marked points, and finally $\Xi$ 
is a  very basic birational map 
which can be described as follows. A projectivization of a generic parabolic vector bundle is a ruled surface $\bP^1\times\bP^1$ with points 
$(z_1,q_1),\ldots,(z_n,q_n)$ giving the parabolic structure and the map $\Xi$  assigns $(q_1,\ldots,q_n)$ to this bundle, see Figure~\ref{shasrhar}.
\begin{figure}[htbp]
\includegraphics[width=2.6in]{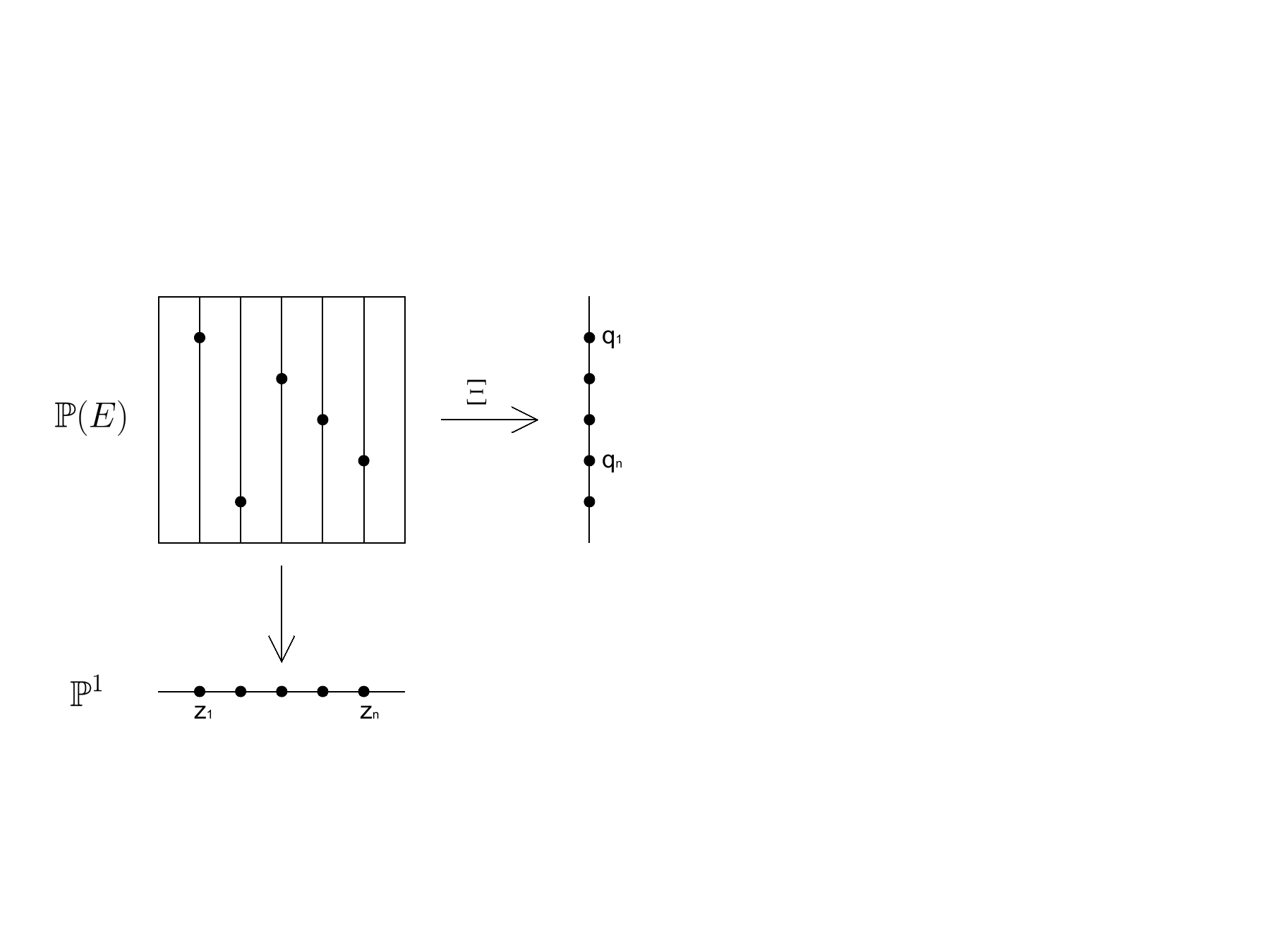}
\caption{$\Xi:\,\Bun(\bP^1;z_1,\ldots,z_n)\dashrightarrow M_{0,n}$, \quad $(E;V_1,\ldots,V_n)\mapsto (q_1,\ldots,q_n).$}\label{shasrhar}
\end{figure}
It~follows that in the hyperelliptic case the scattering amplitude map
combines effects of the birational morphism $\Xi$, which only depends on $z_1,\ldots,z_n$
but not on $C$, and the map~$\bbLa$, which is thus a finer invariant in the hyperelliptic case.
To study the scattering amplitude form as a form on the moduli space of parabolic bundles,
we have to choose its  projective model.
We study effects of a well-known wall-crossing between  projective models
and the action of the Weyl group $W(D_n)$ by elementary transformations.
\end{Review}

\begin{Review}[Contents of~\ref{kwJHEfg}] 
In Theorem~\ref{asfgbsfgasfh} we make another step towards proving Theorem~\ref{2^gTheorem}
- we show that the scattering amplitude map $\bLa$
has degree~$2^g$ for every hyperelliptic curve $C$.
We~use beautiful results of Jacobi \cite{Jacobi} (amplified by Moser and Mumford~\cite{Mumford}) that give an explicit
model for the Jacobian of a hyperelliptic curve (with a removed theta-divisor)
as the orbit space of conjugacy classes of $2\times 2$ polynomial matrices.
Moreover, translation-invariant vector 
fields on the Jacobian can be described in the form of the Lax differential equation
and this can be used 
to compute branches of the scattering amplitude form.
\end{Review}

\begin{Review}[Contents of~\ref{sfgasgsRH}] 
In genus $2$, where every curve is hyperelliptic, we combine methods of \ref{sdvqefv} 
with a classical observation of Halphen \cite{Halphen}: 
a sufficiently general divisor $P=p_1+\ldots+p_n$ of marked points  on a smooth MHV curve
embeds~it
$\phi_{P}:\,C\hookrightarrow \bP^3$
as a degree $n=g+3$ space curve. 
This~gives a way to study the scattering amplitude 
using geometry of $\bP^3$.
By the results of  \ref{sdvqefv},
the scattering amplitude map $\bLa:\,\Pic^3C\dashrightarrow \oM_{0,5}\simeq \dPf$
into the~quintic del Pezzo surface
factors through the moduli space of parabolic vector bundles of rank~$2$, which in this case is the quartic del Pezzo surface~$\dP$.
In~Theorem~\ref{vasdgsgG} we resolve this map by a finite morphism $\bbLa:\,\Bl_{16}\Pic^3C\to\dP$
of degree $4$ 
from the blow-up of $\Pic^3C$ in $16$ special points
to the quartic del Pezzo surface.
We~show in Theorem~\ref{sDGSG}  that whenever all marked points are Weierstrass points,
$\bbLa$~becomes a well-known map: 
it factors as a composition of two double covers, 
$$\Bl_{16}\Pic^3C\arrow^{2:1}  \hbox{\rm K3}\arrow^{2:1}\dP,$$
where the K3 surface is the minimal resolution of the Kummer surface.
For general marked points away from the Weierstrass points, the intermediate K3 surface disappears from the picture 
but the degree~$4$ scattering amplitude morphism $\bbLa$  as well as the
{\em double sixteen} configuration  on the abelian surface survive.
\end{Review}

\begin{Review}[Contents of~\ref{sdcqwv}] In this section we inject a measure of reality
into the study of scattering amplitude forms of smooth curves. In order to obtain real probability measures, we have
to turn to real algebraic geometry.
The answer is beautiful for
smooth real curves with the maximal number of real ovals, so called M-curves,
with marked points distributed as in Figure~\ref{wEGsg}: 
\begin{figure}[htbp]
\centerline{\includegraphics[width=\textwidth]{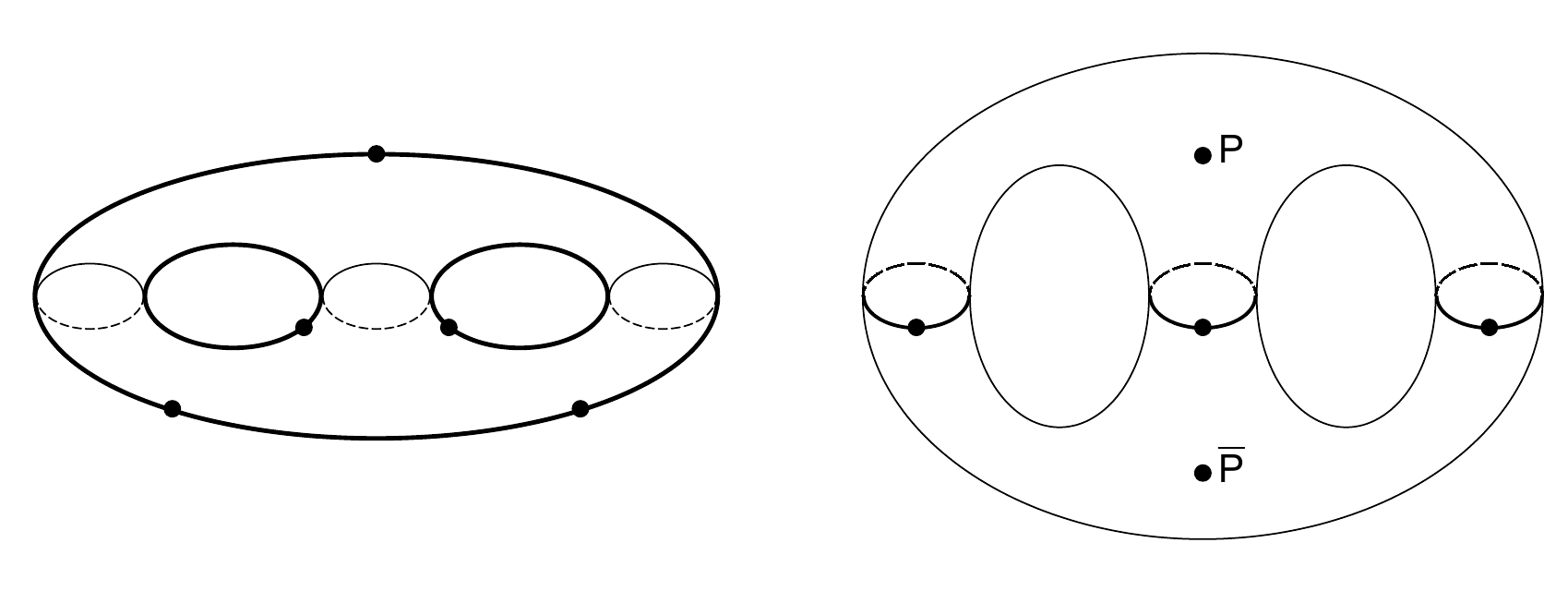}}
\caption{MHV M-curves of type A (left) and type B (right)}\label{wEGsg}
\end{figure}
either all marked points are real and all ovals contain one marked point except for one that contains three (type~A)
or all but two marked points are real, one for each real oval, and the other two are complex-conjugate
(type B). The main result of this section is Theorem~\ref{dfvwevev}:
line bundles in every fiber of the scattering amplitude map $\bLa$
over $M_{0,n}(\bR)$
``localize'' into different connected components (there are conveniently exactly $2^g$ of them)
$$\Pic^{g+1}_I(\bR)\subset \Pic^{g+1}(\bR).$$ 
We use this theorem of real algebraic geometry
to prove Theorem~\ref{2^gTheorem}, which is a theorem of complex algebraic geometry.
\begin{figure}[htbp]
\centerline{\includegraphics[width=0.55\textwidth]{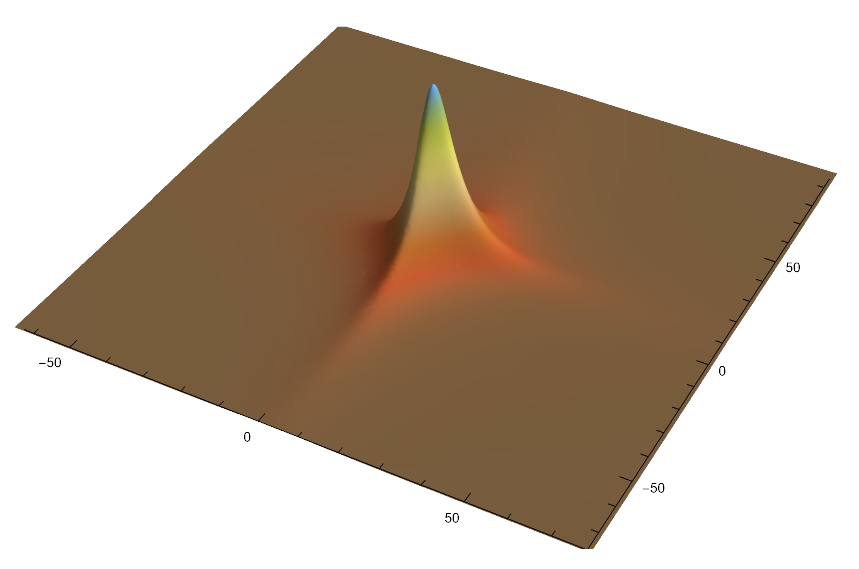}\qquad
\includegraphics[width=0.35\textwidth]{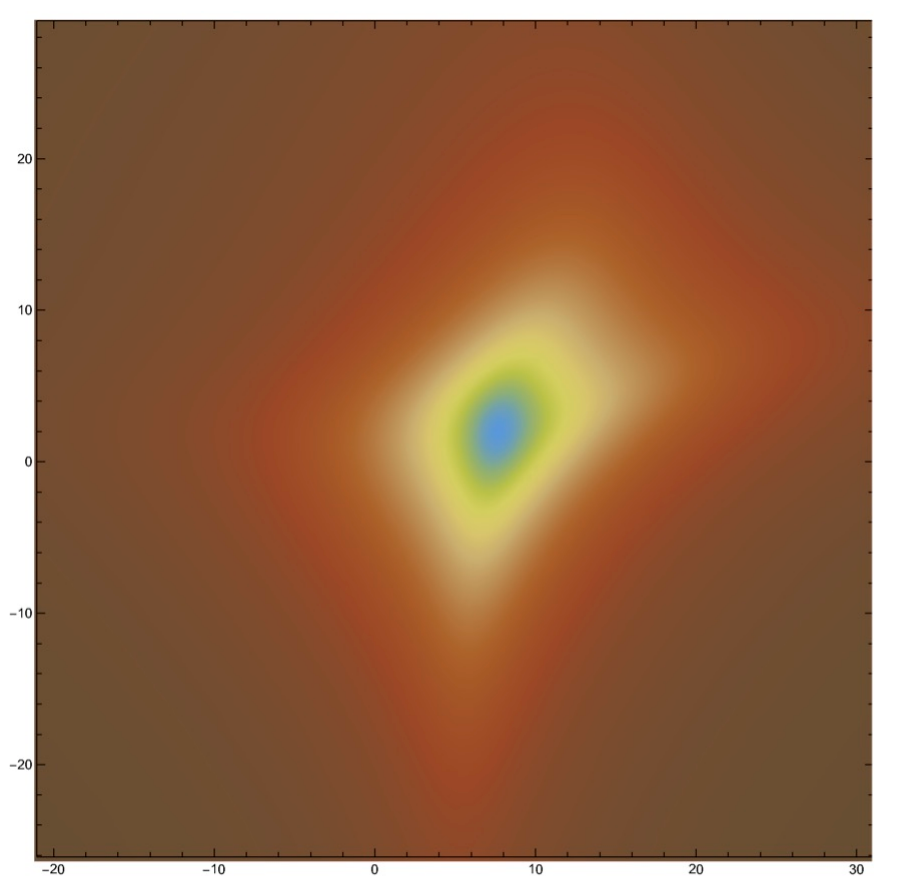}}
\caption{Scattering amplitude probability measures in genus $2$}\label{kgfkhgfh}
\end{figure}
Back to M-curves, in Theorem~\ref{DSvSDb} we show that the scattering amplitude map 
$\bLa$ gives real-analytic open immersions 
$$\bLa_I:\,\Pic_I^{g+1}(\bR)\hookrightarrow M_{0,n}(\bR)$$
from an open dense subset of every connected component of  $\Pic^{g+1}(\bR)$.
This gives real probability measures $|A_I|$ on $M_{0,n}(\bR)$ that depend on $4g$ real parameters.
For an especially nice connected component $\Pic_H^{g+1}(\bR)$,
which we call the {\em Huisman component}, 
the scattering amplitude map induces a real-analytic isomorphism
$$\bR^g/\bZ^g\simeq\Pic^{g+1}_H(\bR)\arrow^\bLa(\bR\bP^1)^g,$$
which gives a positive, smooth (in fact real-analytic) scattering amplitude probability measure on $(\bR\bP^1)^g$, see Theorem~\ref{sGARSGARHA}.
As an example, we study scattering amplitude probability measures in genus $2$,
which produces pretty probability density functions as in Figure~\ref{kgfkhgfh}.
\end{Review}

\begin{Review}[Contents of~\ref{SDgSfh}] 
In this section we describe {\em maximally degenerate} stable MHV curves,
i.e.~curves with trivalent on-shell diagrams. 
By combining results from \cite{CT_Crelle} and \cite{MHV},
we show that they are given by CT hypertrees: collections
$\Gamma=\{\Gamma_1,\ldots,\Gamma_d\}$
of triples in $\{1,\ldots,n\}$ that satisfy
\begin{equation}\label{CondS}
\bigl|\bigcup_{j\in S}\Gamma_j\bigr|\ge |S|+2
\quad\hbox{\rm for every $S\subset\{1,\ldots,d\}$}.
\tag{\ddag}\end{equation}
For example, as observed in \cite{CT_Crelle},  every checkerboard triangulation of a $2$-sphere gives a CT hypertree,
in fact two of them.
\begin{figure}[htbp]
\includegraphics[width=\textwidth]{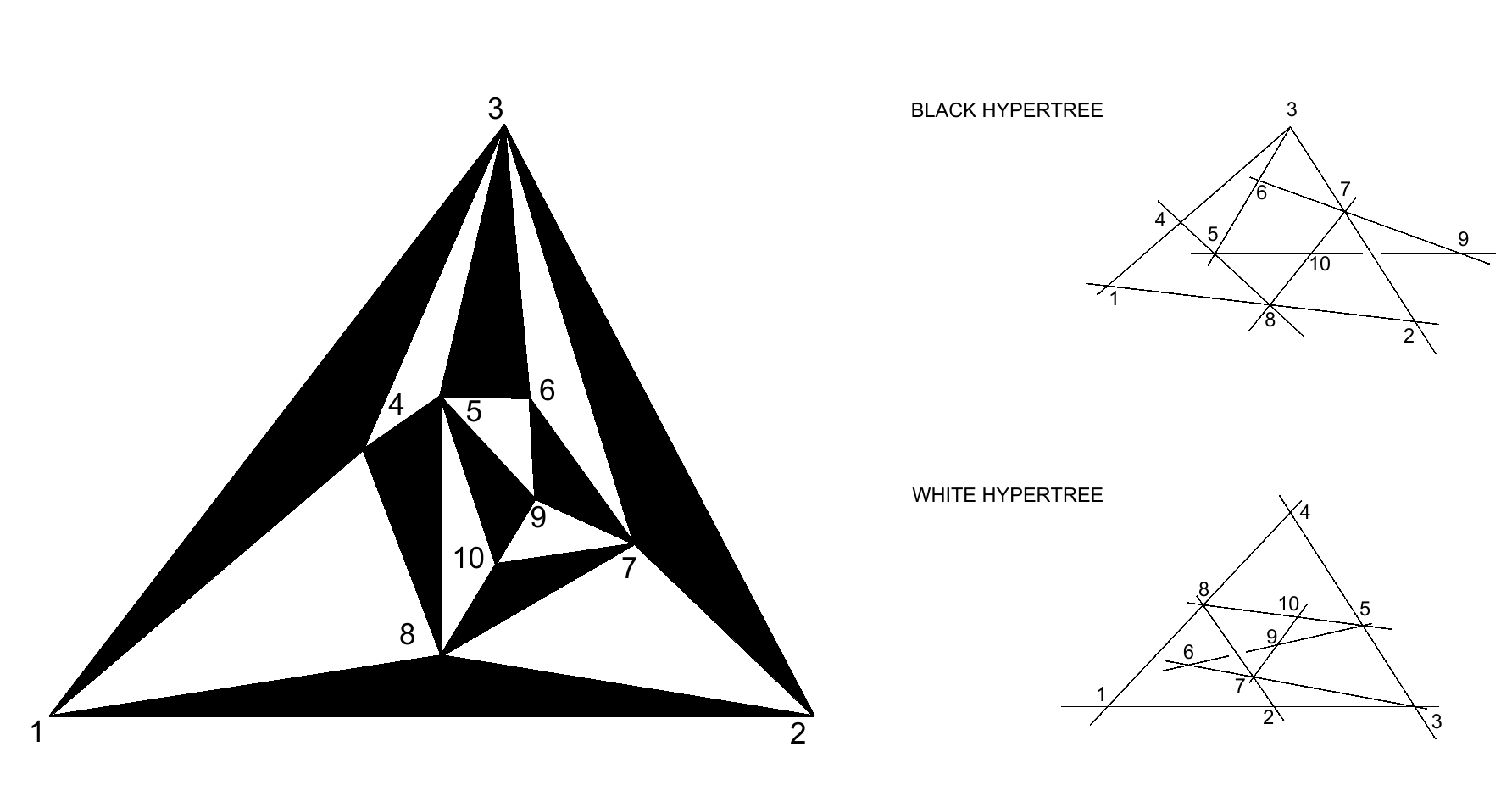}
\caption{\small Two spherical CT hypertrees from the same triangulation \cite{CT_Crelle}}\label{triangpic}
\end{figure} 
Vertices of the triangulation give the indexing set $\{1,\ldots,n\}$ 
and black triangles give triples, see Figure~\ref{triangpic}.
Another CT hypertree is given by  white triangles.
The~scattering amplitude perspective uncovers geometry of
spherical CT hypertrees: 
in Theorem~\ref{sVSGsrg}~we show that they  appear as stable degenerations of real MHV M-curves.
Beautiful results of Tutte on arborescences associated with triangulations
make an appearance in the proof.
\end{Review}

\begin{Review}[Contents of~\ref{lastsection}]
Here we compare our compactified Jacobian approach  to 
a  traditional Grassmannian approach of  \cite{Grass}, especially in the non-MHV case.
\end{Review}

{\bf Acknowledgements.}
I am grateful to Ana-Maria Castravet and Rahul Pandharipande for useful discussions
and to the  referee for  comments and corrections.
The paper started from a correspondence with Freddy Cachazo and Nick Early, who brought  \cite{MHV} to my attention. 
The project was  supported by the NSF grants DMS-1701704 and DMS-2101726 and the Simons Fellowship.

Graphics by \url{www.wolfram.com/mathematica} and \url{www.plainformstudio.com}

\tableofcontents

\section{Massless Brill--Noether theory}\label{khgckhfckhfc}

\begin{Example}
Let $C$ be a smooth genus $1$  curve with $4$ marked points $p_1,\ldots,p_4$.
Every line bundle $L\in\Pic^2C$ gives a double cover 
\begin{figure}[htbp]
\includegraphics[height=2in]{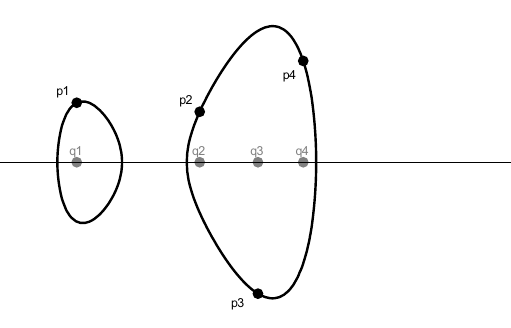}
\end{figure}
$\phi_L:\,C\arrow^{2:1}\bP^1$
presenting $C$ in the form $y^2=f(x)$, where $f$ is a polynomial of degree $3$ or~$4$.
Recall that 
$$M_{0,4}\simeq\bP^1\setminus\{0,1,\infty\}$$ via the cross-ratio function of four points.
We~compactify $M_{0,4}$ by $\oM_{0,4}\simeq\bP^1$.
If~we denote by $q_1,\ldots,q_4$ the images of marked points $\phi_L(p_1),\ldots,\phi_L(p_4)$
then the scattering amplitude map $\bLa$ takes $L\in\Pic^2C$ to their cross-ratio function
$${q_4-q_1\over q_2-q_1}\cdot {q_2-q_3\over q_4-q_3}.$$
Of course $\Pic^2 C\simeq C$ but not canonically. In~fact $\Pic^2C$ has
$6$ distinguished points 
$$p_{ij}=\cO(p_i+p_j),\quad i\ne j,$$
and  carries a natural degree $2$ line bundle
$$\cL=\cO(p_{12}+p_{34})\simeq\cO(p_{13}+p_{24})\simeq\cO(p_{14}+p_{23}).$$
The images of points $p_{ij}\in\Pic^2 C$ under the scattering amplitude map are
$$\bLa(p_{14})=\bLa(p_{23})=0,\quad \bLa(p_{13})=\bLa(p_{24})=1,\quad \bLa(p_{12})=\bLa(p_{34})=\infty$$
and   $\bLa$ is the  double cover 
$$\phi_{\cL}:\,\Pic^2C\arrow^{2:1}\bP^1.$$ 
In the complex torus model $\Pic^2C=\bC/\bZ+\bZ\tau$, 
the  inverse of the scattering amplitude map $\bLa$ is given by the elliptic integral
$$z=\int_{x_0}^x {dx \over \sqrt{g(x)}},$$
where $g$ is a degree $4$ or $3$ polynomial. To summarize:
\begin{enumerate}
\item The scattering amplitude form $A(x)={\displaystyle dx \over\displaystyle \sqrt{g(x)}}$ has
$4$ branch points. 
\item The curve $C$ is uniquely determined by the scattering amplitude form.
\item Marked points $p_1,\ldots,p_4$ are determined by $\bLa^{-1}(0)$, $\bLa^{-1}(1)$ and $\bLa^{-1}(\infty)$ 
but not uniquely: the action of the Klein $4$-group that permutes the marked points  $p_1,\ldots,p_4$ in pairs is not detected by the scattering amplitude.
\end{enumerate}
\end{Example}

\begin{Remark}
A special   case is when $p_1,\ldots,p_4$ are $2$-torsion points of an elliptic curve with  origin $p_1$.
We identify $\Pic^2C$ with $C=\Pic^1C$ by tensoring with $\cO(p_1)$.
Then $\cL$ is identified with $\cO(2p_1)$. The $6$ points $p_{ij}$ specialize to $3$ points $p_2,p_3,p_4$, 
each with multiplicity $2$. A special feature of this scattering amplitude form is that three of its branch points are at $0,1,\infty$.
The fourth one determines the $j$-invariant. More generally, 
we will see in Section~\ref{sdvqefv} that scattering amplitudes of hyperelliptic curves
can detect whether or not the marked points are at the Weierstrass points.
\end{Remark}


%
%
%
%

\begin{Theorem}\label{efvwb}
Every smooth pointed curve $(C;p_1,\ldots,p_n)$ of genus $g$ is an MHV curve.
\end{Theorem}

\begin{proof}
Recall that $n=g+3$ and $d=g+1$.
We are going to use repeatedly that a generic line bundle of degree less than $g$
is not effective \cite{ACGH}. Let $L$ be a generic line bundle of degree $g+1$. Then $\omega_C\otimes L^*$ is also a
generic line bundle.
Since its degree is $(2g-2)-(g+1)=g-3<g$,  we have $h^1(L)=h^0(\omega_C\otimes L^*)=0$, i.e. $L$ is not special.
This proves part (1) of Definition~\ref{KSLJDfksjdg}.

Next, we verify part (2).
If $p\in C$ then $h^1(L(-p))=h^0(\omega_C\otimes L^*(p))$.
Since the degree of $\omega_C\otimes L^*(p)$ is $g-2$,
it is effective only if can be written as $\cO(x_1+\ldots+x_{g-2})$ for some points on the curve.
It follows that $L\simeq\omega_C(p-x_1-\ldots-x_{g-2})$. But ~the locus of  these line bundles  is at most $(g-1)$-dimensional,
which contradicts the fact  that $L$ is generic.
Thus $L$ is globally generated.  

We  claim that $\phi_L(p_i)\ne\phi_L(p_j)$ when $i\ne j$, i.e.~$\bLa(L)\in M_{0,n}$ for a generic $L$.
Indeed,
$h^1(L(-p_i-p_j))=h^0(\omega_C\otimes L^*(p_i+p_j))=0$  since its degree is $g-1<g$. 
Thus $\phi_L$ sends $p_i$ and $p_j$ to different points of $\bP^1$.

To show that the scattering amplitude map $\bLa:\,\Pic^{g+1}C\dashrightarrow M_{0,n}$ is generically finite, we
compute its differential at a generic point.
Tensoring an exact sequence
$$0\to\cO(-P)\to\cO\to\bigoplus_{i=1}^{g+3}\cO_{p_i}\to0,$$
where $P=p_1+\ldots+p_{g+3}$,
with $\phi_L^*(T_{\bP^1})\simeq\cO(L^{\otimes2})$ gives an exact sequence
$$H^0(C,L^{\otimes2}(-P))\to H^0(C,\phi_L^*(T_{\bP^1}))\to\bigoplus_{i=1}^{g+3}H^0(p_i,\phi_L^*(T_{\bP^1}))\to
H^1(C,L^{\otimes2}(-P)).$$
Since $L^{\otimes2}(-P)$ has degree $2(g+1)-(g+3)=g-1$, both ends of the sequence vanish for generic $L$
and we have an isomorphism
$$H^0(C,\phi_L^*(T_{\bP^1}))\simeq\bigoplus_{i=1}^{g+3}H^0(p_i,\phi_L^*(T_{\bP^1}))\simeq\bC^{g+3}$$
In other words, every infinitesimal deformation of points $\phi_L(p_1),\ldots,\phi_L(p_{g+3})$ in~$\bP^1$
is induced by an infinitesimal deformation of the map $C\to\bP^1$ which since $H^1(C,L)=0$ 
is given by an infinitesimal deformation of $L$ and infinitesimal $\PGL_2$ action. 
It follows that the differential of $\bLa$ is an isomorphism at  $L$.
\end{proof}

Analyzing the special  loci from the proof of Theorem~\ref{efvwb} gives the following

\begin{Amplification}[planar locus $W$ and special divisors $E, E_{ij}, R\subset \Pic^{g+1}C$]\label{zxbzfn} $\ $
\begin{enumerate}
\item $L$ is not special away from the image $W$ of the map 
$$\Sym^{g-3}C\to\Pic^{g+1}C,\quad (x_1,\ldots,x_{g-3}) \mapsto \omega_C(-x_1-\ldots-x_{g-3}).$$
\item $L$ is base-point-free away from the image $E$ of the map 
$$C\times\Sym^{g-2}C\to\Pic^{g+1}C,\quad (p,x_1,\ldots,x_{g-2}) \mapsto \omega_C(p-x_1-\ldots-x_{g-2}).$$
\item $\bLa(L)\in M_{0,n}$ when $L$ is away from $E$ and  the images $E_{ij}$ for 
$1\le i<j\le n$ of the maps $\Sym^{g-1}C\to\Pic^{g+1}C$,
$$(x_1,\ldots,x_{g-1}) \mapsto \omega_C(-x_1-\ldots-x_{g-1}+p_i+p_j).$$
\item 
Let $\Theta\subset\Pic^{2g+2}$  be the image of the map $\Sym^{g-1}C\to\Pic^{2g+2}C$,
$$(x_1,\ldots,x_{g-1}) \mapsto \cO(x_1+\ldots+x_{g-1}+p_1+\ldots+p_{g+3}).$$
The scattering amplitude map $\bLa$ is unramified away from $E$, $E_{ij}$ and the locus $R=m^{-1}(\Theta)$,
where $m$ is the map 
$$m:\,\Pic^{g+1}C\to\Pic^{2g+2}C,\quad L\mapsto L^{\otimes 2}.$$
\item $W\subset E\cap\bigcap\limits_{i\ne j} E_{ij}$ has codimension $3$ in $\Pic^{g+1} C$. 
\item  
$E_{ij}$ and $\Theta$ are theta divisors.
\item $E$ is a theta divisor if $C$ is hyperelliptic, otherwise $E\equiv (g-1)\Theta$.
\item The divisor $R$ is algebraically equivalent to $4\Theta$.
\end{enumerate}
\end{Amplification}

\begin{proof}
Recall that theta divisors are defined up to translation by an element of $\Pic^0C$ 
and are not preserved by any non-zero translation.
The locus of effective divisors in $\Pic^{g-1}C$ is ``the'' theta-divisor. Thus
everything follows from the proof of Theorem~\ref{efvwb}, except for  the calculation 
of the class of the ramification divisor~$R$, which follows from the theorem of the cube,
and the class of the ``difference divisor''~$E$.
If $C$ is hyperelliptic then 
$\omega_C(p-x_1-\ldots-x_{g-2})\sim K_C-(g-2)h+p+x_1+\ldots+x_{g-2}$,
where $h$ is the hyperelliptic divisor, and so $E$ is the translate of the theta-divisor.
If $C$ is not hyperelliptic then the class of $E$ in the Neron--Severi group was computed in \cite[Prop.~3.7(b)]{FMP}.\footnote{As pointed out by the referee, \cite[Prop.~3.7]{FMP} studies the image of the difference map
$\psi_{b,a}:\,\Sym^bC\times\Sym^aC\to\Pic^{b-a}C,\quad (p_1,\ldots,p_b)\times (x_1,\ldots,x_a) \mapsto \cO_C(p_1+\ldots+p_b-x_{1}-\ldots-x_{a})$, only under the assumption that $1\le b\le a\le {g-1\over 2}$, which is stronger than ours ($b=1$, $a=g-2$.)
However, their argument works verbatim in our case.
We only need to show that $\psi_{1,g-2}$ is birational onto its image if $C$ is not hyperelliptic.
Arguing by contradiction, take two different points in a fiber, 
$p_1-x_1-\ldots-x_{g-2}\sim p'_1-x_1'-\ldots-x'_{g-2}$,
for a general $(g-1)$-tuple $p_1,x_1,\ldots,x_{g-2}$.
Then
$p_1+x'_1+\ldots+x'_{g-2}\sim p'_1+x_1+\ldots+x_{g-2}$.
If these divisors are different then, by Riemann singularity theorem and dimension count,
$\dim\Sing\Theta\ge g-3$, which contradicts Martens theorem since $C$ is not hyperelliptic.
Thus these divisors are the same, and so $p_1=x_i$ for some $i$, which contradicts generality.}
\end{proof}

\begin{Remark}
Divisors $E$, $E_{ij}$ and $R$ all provide natural polarizations of $\Pic^{g+1}C$.
Note especially that $E$ is independent of the marked points $p_1,\ldots,p_n$. 
\end{Remark}

\begin{Review}[Loci $E$ and $E_{ij}$]
After choosing a compactification of $M_{0,n}$, for example $\oM_{0,n}$, 
the scattering amplitude map (as any rational map) 
will extend generically along divisors $E$ and $E_{ij}$. Line bundles $L\in E$ have base locus but (away from $W$) 
still determine meromorphic functions $\phi_L:\,C\to\bP^1$. However, moving base points doesn't change $\phi_L$
but changes $L$ and as a result $\bLa$ contracts $E$ to a locus of smaller dimension, 
most dramatically to a point $o\in M_{0,n}$ in the hyperelliptic case 
(see~Corollary~\ref{adfgadrha}). This produces a singularity in the scattering amplitude probability measure (see \ref{qergarh}).
By contrast, divisors $E_{ij}$ are largely harmless: extension of $\bLa$ maps $E_{ij}$  to the locus in the compactification of $M_{0,n}$
where two marked points in $\bP^1$ come together, for example to the boundary divisor $\Delta_{ij}$ in the case of~$\oM_{0,n}$.
Here is an example of a different and very ergonomic compactification.
\end{Review}

\begin{Lemma}\label{adrhadrhsth}
Fix three indices, for example $g+1,g+2,g+3$.
The scattering amplitude map induces a generically finite rational map
\begin{equation}\label{bSheb}
\bLa:\, \Pic^{g+1}C\dashrightarrow M_{0,n}\xrightarrow{\pi_1,\ldots,\pi_g}(\oM_{0,4})^g\simeq(\bP^1)^g,
\cooltag\end{equation}
where 
$\pi_i:\,M_{0,n}\to \oM_{0,4}$ for $i=1,\ldots,g$
is the forgetful morphism given by the indices $i,g+1,g+2,g+3$.
A general line bundle $L\in\Pic^{g+1}C$ is mapped by $\pi_i$ to the cross-ratio
of points 
$\phi_L(p_i),\phi_L(p_{g+1}),\phi_L(p_{g+2}),\phi_L(p_{g+3})$.
Let $L\in \Pic^{g+1}C\setminus\{R\cup E\}$. If points $\phi_L(p_{g+1}),\phi_L(p_{g+2}),\phi_L(p_{g+3})$
are different then \eqref{bSheb} is regular and unramified at $L$.
\end{Lemma}

\begin{proof}
Identical to the last paragraph of the proof of Theorem~\ref{efvwb}.
\end{proof}

\begin{Review}[Locus W]\label{evwevwev}
We call $W$ the {\em planar locus}:
generically along $W$, $\phi_L$ is a morphism 
\begin{equation}\label{sgasgas}
\phi_L:\,C\to\bP^2.
\cooltag\end{equation}
If $C$ is a general curve, $\phi_L$ realizes it as a degree $g+1$ 
plane curve with ${g(g-3)\over 2}$ nodes away from the marked points $p_1,\ldots,p_n$.
Generically along $W$, the scattering amplitude map is resolved by the blow-up 
\begin{equation}\label{agargasr}
\bLa:\,G=\Bl_W\Pic^{g+1}C\dashrightarrow M_{0,n}.
\cooltag\end{equation}
In the language of 
Brill--Noether theory~\cite{ACGH}, $G$ parametrizes 
pencils of divisors on $C$ of degree $g+1$. Let 
$$\hat W\subset G$$
be an exceptional divisor over $W$. The map $\hat W\to W$ is a $\bP^2$-bundle (generically along $W$).
Generically along $\hat W$, the map $\bLa$ of \eqref{agargasr} can be described as follows: 
a~point $(L,p)\in\hat W$ gives both a ``planar realization''
\eqref{sgasgas} and a  point $p\in\bP^2$.
Projecting points 
$$\phi_L(p_1),\ldots,\phi_L(p_n)\in\bP^2$$ from $p$
gives points $q_1,\ldots,q_n\in \bP^1$.
The class of $(q_1,\ldots,q_n)$ in $M_{0,n}$
is the image of $(L,p)$ under~$\bLa$. 
By \cite[Th.~3.1]{CT_Cont}, $\bLa(D_W)\subset\oM_{0,n}$ is a divisor 
covered by surfaces 
\begin{equation}\label{aljskhfgsk}
S=\Bl_{\phi_L(p_1),\ldots,\phi_L(p_n)}\bP^2
\cooltag\end{equation}
for general $L\in W$
unless points $\phi_L(p_1),\ldots,\phi_L(p_n)$ lie on a conic, in which case the conic is contracted to a point.
\end{Review}

\begin{Lemma}
The multi-valued scattering amplitude form $A$ on $M_{0,n}$ has branches that vanish
along the divisor $\bLa(D_W)$ with multiplicity $2$.
\end{Lemma}

\begin{proof}
Equivalently, we claim that the pull-back  of the scattering amplitude form $A$ on $\Pic^{g+1}C$ to $G$
vanishes to the order $2$ along $\hat W$.
But  $W$ has codimension $3$ and for a blow-up $\pi:\,G=\Bl_WX\to X$
with exceptional divisor $\hat W$,
one has  $\pi^*K_X=K_{G}(-k\hat W)$, where $k=\codim_WX-1$.
\end{proof}

\begin{Example}
$W$ is empty in genus $1$ and $2$.
For a curve of genus $3$,
$$W=\{K\}\in\Pic^{4}C$$
and $\phi_K:\,C\to\bP^2$ is an embedding of $C$ as a quartic curve if $C$ is not hyperelliptic
and a $2:1$ map to a conic in $\bP^2$ if $C$ is hyperelliptic.
If $C$ is a general quartic (resp. ~ hyperelliptic) curve then 
$S$ as in \eqref{aljskhfgsk} is a general smooth (resp.~nodal) cubic surface.
This is related to the fact that $M_{0,6}$ has another projective model, 
the Segre cubic threefold in $\bP^4$, and $S$ is its hyperplane section.
At the moment little is known
about divisors $\bLa(\hat W)\subset M_{0,n}$ for $g>3$.
\end{Example}

\section{One and two channel factorization}\label{zdfbsdh}

\begin{Review}\label{asrsrh}
Recall that  the on-shell diagram is the dual graph of $C$ decorated with degrees of $L\in\Pic^{\vec d}C$ on its irreducible components.
If~$C$ is an MHV curve then none of these degrees are negative since $L$ is globally generated.
However, some of the degrees can be equal to $0$.
The corresponding components are 
contracted by the map $\phi_L:\,C\to\bP^1$.
The union of irreducible components where $L\in\Pic^{\vec d}C$ has degree $0$ can be written 
as a disjoint union of maximal connected components
$$C^{(0)}_1,\ldots, C^{(0)}_r.$$
The following lemma is essentially from \cite{MHV}:
\end{Review}

\begin{Lemma}\label{SFbSFh}
If $C$ is an MHV curve then each connected component $C^{(0)}_i$ is a curve of arithmetic genus $0$ (i.e.~a tree of $\bP^1$'s) with at most one marked point.
\end{Lemma}

\begin{proof}
The restriction of a generic line bundle $L$ to each $C^{(0)}_i$ is a globally generated, on the other hand generic, line bundle of degree $0$.
Therefore this restriction is a trivial line bundle and so the genus of each $C^{(0)}_i$ is zero.
Since $\phi_L$ contracts each $C^{(0)}_i$, it~can~contain at most one marked point,
otherwise $\bLa$ is not dominant.
\end{proof}

\begin{Remark}\label{AKHJDbfHSDB}
One can substitute each tree $C^{(0)}_i$ of rational components with any other fixed curve of arithmetic genus zero
with the same number of marked points (i.e.~zero or one) and the same number of ``outbound'' points
(where $C^{(0)}_i$ is connected to the rest of the curve $C$). This  doesn't change the scattering amplitude. 
In the language of on-shell diagrams, Lemma~\ref{SFbSFh} says that 
a~subgraph of degree~$0$ vertices is a disjoint union of $r$ connected trees
of white circles with at most one marked point on each  tree.
Each of these trees can be substituted with one white {\em megacircle}
without changing the scattering amplitude.
\end{Remark}

The following fact is very useful.

\begin{Lemma}\label{jhgkjgfk}
Let $C$ be an MHV curve. Choose two points $x, y\in C$ (not necessarily marked points).
Then either $x$ and $y$ belong to the same rational component $C^{(0)}_i$ of \ref{asrsrh} or 
$$\phi_L(x)\ne\phi_L(y)$$ 
for a general line bundle $L\in\Pic^{\vec d}C$.
\end{Lemma}

\begin{proof}
Let $L$ be a general line bundle in $\Pic^{\vec d}C$. It gives a morphism $\phi_L:\,C\to\bP^1$.
Suppose $x$ and $y$ are not in the same  component $C^{(0)}_i$ (clearly contracted by $\phi_L$)
but  $\phi_L(x)=\phi_L(y)=0\in\bP^1$. Conditions (1) and (2)
in the Definition~\ref{KSLJDfksjdg}  of an MHV curve are open in $\Pic^{\vec d}C$.
We claim that $L$ can be deformed to force $\phi_L(x)\ne\phi_L(y)$.
Let $g:\,C'\to C$ be a morphism and let $D,D'\subset C'$ be divisors defined as follows:
\begin{enumerate}
\item If $x$ is a smooth point of $C$ and its irreducible component $A$ containing $x$
is not contracted by $\phi_L$ then $C'=C$, $D=m[x]$, where $\phi_L^*(0)=m[x]+\ldots$ and $D'=m[x']$, where  $x'\in A$ is a general point.
\item If $x$ is a node of $C$ and irreducible components $A$ and $B$ passing through~$x$ are not contracted by $\phi_L$ 
(it could be that $A=B$) then $C'$ is a partial normalization of $C$ obtained by separating the node $x$,
$D$ is a divisor supported at $g^{-1}(x)$ such that $(\phi_L\circ g)^*(0)=D+\ldots$ and $D'$ is a general effective divisor on $C'$
with the same multidegree as $D$.
\item Finally, if $x\in C^{(0)}_i$ then let $C'$ be $\overline{C\setminus C^{(0)}_i}$, 
$D$ is a divisor supported at $g^{-1}(C^{(0)}_i)$ such that $(\phi_L\circ g)^*(0)=D+\ldots$ and $D'$ is a general effective divisor on $C'$
with the same multidegree as $D$.
\end{enumerate} 
Let $L_0\in\Pic^{\vec 0}C$ be a line bundle such that $g^*L_0\simeq\cO_{C'}(D'-D)$, which exists since $g^*: \Pic(C)\rightarrow\Pic(C')$ is surjective in all three cases.
We can assume that $\phi_{L\otimes L_0}(y)=0$.
We claim that $\phi_{L\otimes L_0}(x)\ne0$. Indeed, let $s$ be a global section of $L$ that vanishes at $x$ and $y$
and let $s_0$ be a rational section of $L_0$ such that $(g^*s_0)=D'-D$. Then $g^*(s_0s)$ has no poles and doesn't vanish along $D$, 
and therefore $s_0s$ is a global section of $L_0\otimes L$ that vanishes at $y$ but not at $x$.
\end{proof}

There are two more restrictions on MHV curves:

\begin{Lemma}\label{wefwetb}
Let $(C, \vec d)$ be an MHV curve with a connected subcurve $A$ of arithmetic genus~$p$, 
and $L\in\Pic^{\vec d}C$ a generic line bundle.
\begin{enumerate}
\item If $p>0$ then $\deg L|_A\ge p+1$.
\item If $\deg L|_A=p+1$ then $A$ contains at most $p+3$ marked points.
\end{enumerate}
\end{Lemma}

\begin{proof}
Indeed, $L|_A$  is both globally generated and generic,
thus $\deg L|_A\ge p+1$. This proves (1). 
Let $N_A\subset\{1,\ldots,n\}$ be the subset of marked points that belong to~$A$ and let $n_A$ be its cardinality.
 If $\deg L|_A=p+1$ then $\phi_{L|_A}$ is a map $A\to \bP^1$.
The~scattering amplitude map $\bLa$ is dominant, therefore the induced map $\Pic^{\vec d_A}A\to M_{0,n_A}$
 that sends $L|_A$ to the configuration of points $\phi_{L|_A}(p_i)$ for $i\in N_A$
 must be dominant as well. It follows that $n_A\le p+3$.
 This proves (2). 
\end{proof}

\begin{Corollary}
Figure~\ref{SDgSg} lists all MHV curves with $g=1$  (up to permuting markings).
\end{Corollary}

\begin{proof}
A dual graph of a nodal curve of arithmetic genus $1$ is a graph of genus $1$, i.e.~
a genus 1 vertex or a cycle (perhaps reduced to a loop) with trees attached.
By Lemma~\ref{wefwetb}, the degree of the cycle is $2$, so all trees have degree $0$. By Lemma~\ref{SFbSFh},
there are actually no trees, so the on-shell diagram is a cycle. This leaves diagrams from Figure~\ref{SDgSg}.
\end{proof}

\begin{Example}\label{sDGASG}
Let $C$ be an irreducible nodal curve of arithmetic genus $1$  with $4$ marked points $p_1,\ldots,p_4$.
We~view $C$ as $\bP^1$ with $0$ and $\infty$ identified. Marked points $p_1,\ldots,p_4\in\bC^*\subset \bP^1$.
For $x,y\in\bC^*$, the line bundle $L=\cO([x]+[y])\in\Pic^2C$ determines the map 
$$\phi_L:\,C\arrow^{2:1}\bP^1,\qquad p\mapsto {p\over (p-x)(p-y)}$$ 
(note that $\phi_L(0)=\phi_L(\infty)$).
It is easy to see that $\cO([x]+[y])\simeq \cO([x']+[y'])$ if and only if $xy=x'y'=z\in\bC^*$.
This gives an identification of $\Pic^2C\simeq\bC^*$. The map
$$\bLa:\,\Pic^2C\to\oM_{0,4}\simeq\bP^1,\qquad \bLa(L)=\lambda=
[{\phi_L(p_1):\phi_L(p_2);\phi_L(p_3):\phi_L(p_4)}]$$
is the cross-ratio of points varying with $L$. Using invariance of the cross-ratio,
$$\bLa(L)=\left[{p_1\over (p_1-x)(p_1-y)}:{p_2\over (p_2-x)(p_2-y)};{p_3\over (p_3-x)(p_3-y)}:{p_4\over (p_4-x)(p_4-y)}\right]$$
$$=\left[{ (p_1-x)(p_1-y)\over p_1}:{(p_2-x)(p_2-y)\over p_2};{(p_3-x)(p_3-y)\over p_3}:{(p_4-x)(p_4-y)\over p_1}\right]$$
$$=\left[{ p_1-(x+y)+{z\over p_1}}:{p_2-(x+y)+{z\over p_2}};{p_3-(x+y)+{z\over p_3}}:{p_4-(x+y)+{z\over p_4}}\right]$$
$$=\left[{ p_1+{z\over p_1}}:{p_2+{z\over p_2}};{p_3+{z\over p_3}}:{p_4+{z\over p_4}}\right]=\lambda$$
Thus $\lambda=\bLa(z)$ extends to a map $\bP^1\arrow^{2:1}\bP^1$ that sends $0, \infty$ to $[p_1:p_2;p_3:p_4]$.
The target $\bP^1$ is the $\oM_{0,4}$ and the source $\bP^1$ is a new feature of the stable curve case, the normalization of the
{\em compactified Jacobian} $\overline{\Pic^2C}$,  in this case  a nodal cubic.
The~ $1$-form ${dz\over z}$ of $\overline{\Pic^2C}$ is invariant under the group action of $\Bbb C^* 
\cong\Pic^{\vec0} C$ on $\overline{\Pic^2C}$. As~a~form on $\overline{\Pic^2C}$, the amplitude 
has log poles at its boundary point
but as a form of $\lambda$ the pole is located at $[p_1:p_2;p_3:p_4]\not\in\{0,1,\infty\}$.
To summarize,

\begin{enumerate}
\item The scattering amplitude form is $A(\lambda)={d\lambda\over (\lambda-\lambda_0)\sqrt{f_2(\lambda)}}$,
with the log pole 
$$\lambda_0=[p_1:p_2;p_3:p_4]\not\in\{0,1,\infty\},$$
where $f_2$ is a polynomial of degree $2$ with roots $\bLa\left(\pm\sqrt{p_1p_2p_3p_4}\right)$.
\item The marked points $p_1,\ldots,p_4$ are determined modulo 
the action of the Klein $4$-group that permutes them in pairs.
Specifically, 
$$\bLa^{-1}(0)=\{p_1p_4,p_2p_3\},\quad \bLa^{-1}(1)=\{p_1p_3,p_2p_4\},\quad \bLa^{-1}(\infty)=\{p_1p_2,p_3p_4\}.$$
\end{enumerate}
\end{Example}

\begin{Review}
Next we consider various situations when the curve $C$
can be separated into two connected components $A$ and $B$ by  nodes.
We will use the following notation:
\begin{enumerate}
\item $g_A$ and $g_B$ for the arithmetic genus of components $A$ and $B$.
\item $d_A$ and $d_B$ for the degree of $L_A=L|_A$ and $L_B=L|_B$. We have $d_A+d_B=d$.
\item $N_A$ and $N_B$ for the subsets of marked points on $A$ and $B$ and $n_A$~and~$n_B$
for their cardinalities.
We have $n_A+n_B=n$.
\end{enumerate}
\end{Review}

\begin{Review}[One-channel factorization]
Let $C$ be a nodal curve with a separating node.
In~other words, the on-shell diagram of $C$ is not $2$-connected. This is related to a one-channel factorization 
in physics where virtual particles are are interchanged through the separating node.
The following lemma is well-known in physics.
\end{Review}

\begin{Lemma}\label{qwefv2e2}
A curve with a separating node is never an MHV curve. In particular, the only MHV curves with compact Jacobians are smooth curves.
\end{Lemma}

\begin{proof}
We argue by contradiction.
Let the node separate $C$ into components $A$ and~$B$.
Note that $g_A+g_B=g$. 
If $g_A>0$ and $g_B>0$, then  $d_A\ge g_A+1$, $d_B\ge g_B+1$ by Lemma~\ref{wefwetb}, thus $g+1\ge g+2$, a contradiction.
If $g_A=g$ and $g_B=0$ (or vice versa) then $d_A=g+1$ and $d_B=0$  by Lemma~\ref{wefwetb}.
But $B$ contains at least two marked points, which contradicts Lemma~\ref{SFbSFh}.
\end{proof}

\begin{Definition}[two-channel factorization]\label{FbfhFHF}
Suppose that $C$ can be separated into $A$ and $B$ by two nodes but not by one node.
In other words, the on-shell diagram of $C$ is $2$-connected but not $3$-connected.
We say that $C$ has a {\em two-channel factorization} unless $A$ (or $B$) is a $\bP^1$ with a marked point 
and $L$ has degree $0$ on it.
\end{Definition}

\begin{Example}\label{SfbSfb}
Let $(C;p_1,p_2,p_3,p_4)$ be a curve obtained by gluing two copies of $\bP^1$ at $0$ and $\infty$,
with $p_1,p_2$ on the first component and $p_3,p_4$ on the second. 
This is an example of a two-channel factorization.
The MHV component of the Picard group is $\Pic^{1,1}C$.
Every line bundle $L\in \Pic^{1,1}C$ can be written as $\cO([x]+[x'])$, where $x$ is a point of the first component and $x'$ of the second 
(away from the nodes of $C$). Consider the map
$$\phi_L:\,C\arrow^{2:1}\bP^1,\quad p\mapsto\begin{cases} 
a+{b\over p-x} & \hbox{\rm on the first component,}\cr
a'+{b'\over p-x'} & \hbox{\rm on the second component,}\cr
\end{cases}$$
where $a=a'$ and $bx'=b'x$ so that $\phi_L$ is regular.
It is easy to see that $\cO([x]+[x'])\simeq\cO([y]+[y'])$ if and only if ${x\over x'}={y\over y'}=z$.
This gives an identification $\Pic^{1,1}C\simeq\bC^*$.

We are studying the scattering amplitude map
$\bLa:\,\Pic^{1,1}C\to\oM_{0,4}\simeq\bP^1$ given by 
the cross-ratio of points $\phi_L(p_1),\ldots,\phi_L(p_4)$ varying with $L$. Using invariance of the cross-ratio,
$$\lambda=\left[
a+{b\over p_1-x}:
a+{b\over p_2-x};
a+{b'\over p_3-x'}:
a+{b'\over p_4-x'}
\right]
$$
$$=\left[
{b\over p_1-x}:
{b\over p_2-x};
{b'\over p_3-x'}:
{b'\over p_4-x'}
\right]=\left[
{x\over p_1-x}:
{x\over p_2-x};
{x'\over p_3-x'}:
{x'\over p_4-x'}
\right]
$$
$$=\left[
{p_1\over x}:
{p_2\over x};
{p_3\over x'}:
{p_4\over x'}
\right]
=\left[
{p_1}:
{p_2};
{z p_3}:
{z p_4}
\right]
={z p_4-p_1\over p_2-p_1}\cdot {p_2-z p_3\over z p_4-z p_3}.
$$
As in Example~\ref{sDGASG}, we view the map $\bLa$ as the map from the compactified Jacobian:
$$\overline{\Pic^{1,1}C}=\bP^1_z/_{\{ 0=\infty\}}\arrow^{2:1} \bP^1_\lambda=\oM_{0,4}.$$
The difference from Example~\ref{sDGASG} is that 
$\bLa(0)=\bLa(\infty)=\infty$. In particular, the scattering amplitude form, which  
is the $\Pic^{0,0} C$- invariant form $dz\over z$ 
written as a multivalued form of ~$\lambda$, has a log pole at $\infty$.
Explicitly, from the above, we have
\begin{equation}\label{afbfhafhdf}
\lambda z(p_2-p_1)(p_4-p_3)=(zp_4-p_1)(p_2-zp_3).
\cooltag\end{equation}
Solving \eqref{afbfhafhdf} for $z$ using quadratic formula gives,
after  algebraic manipulations,
$${dz\over z}=(p_1-p_2)(p_4-p_3){d\lambda\over\sqrt{f_2(\lambda)}},$$
where $f_2(\lambda)$ is the discriminant of this quadratic equation.
\end{Example}

\begin{Lemma}\label{aEFaf}
Suppose $(C, \vec d)$ is an MHV curve with a two-channel factorization into  components $A$ and $B$ by nodes $x$ and $y$. 
Let $L\in\Pic^{\vec d}C$ be a generic line bundle. Then 
\begin{enumerate}
\item $d_A=g_A+1$, $d_B=g_B+1$.
\item $n_A=g_A+3$, $n_B=g_B+1$ (Case I) or $n_A=g_A+2$, $n_B=g_B+2$ (Case II).
\item $H^0(A,L_A)=H^0(B,L_B)=2$, \quad $H^1(A,L_A)=H^1(B,L_B)=0$.
\item $L_A$ and $L_B$ are globally generated.  
\item $\phi_L(x)\ne\phi_L(y)$. In particular, neither $A$ nor $B$ contains 
a degree $0$ rational component $C^{(0)}_i$ with both $x$ and $y$ as in Lemma~\ref{jhgkjgfk}.
\end{enumerate}
\end{Lemma}

\begin{proof}
Let $L\in\Pic^{\vec d}C $ be a generic line bundle.
Consider a partial normalization map 
$$\nu:\,A\coprod B\to C$$
and an exact sequence
$$0\to\cO_C\to\nu_*(\cO_A\oplus\cO_B)\arrow^\alpha\cO_x\oplus\cO_y\to 0,$$
where $\alpha(f,g)=(f(x)-g(x),f(y)-g(y))$. Since $H^1(C,L)=0$, tensoring with~$L$ gives
$H^1(A,L_A)=H^1(B,L_B)=0$. Since restrictions $L_A$ and $L_B$ are globally generated and neither $A$ nor $B$
is a $\bP^1$ with a marked point and degree $0$, we get 
\begin{equation}\label{asrgasr}
d_A=g_A+1,\quad d_B=g_B+1,
\cooltag\end{equation}
$$H^0(A,L_A)=H^0(B,L_B)=2.$$
Next we consider  the restriction map
$H^0(A,L_A)\to \bC^2$ at points $x$ and~$y$. We claim that it is an isomorphism (and the same for $B$)
and in particular  $\phi_L(x)\ne\phi_L(y)$.
If~not, let $s_A\in H^0(A,L_A)$ be a non-zero section that vanishes at $x$ and $y$.
Extension of this section by $0$ on $B$ gives a section $s_1$ in $H^0(C,L)$.
Let $s_2\in H^0(C,L)$ be a linearly independent section.
Then the morphism $\phi_L:\,C\arrow^{[s_1:s_2]}\bP^1$ contracts $B$ to a point.
Thus all components of $B$ have degree $0$, which contradicts \eqref{asrgasr}.

Finally, we study distribution of marked points among components $A$ and $B$. 
By Lemma~\ref{wefwetb}~(2),
$n_A\le g_A+3$,  $n_B\le g_B+3$. Up to symmetry,
there are only  two possible cases, namely (I) and (II).
\end{proof}

\begin{Amplification}\label{adgbarhR}
The proof of the lemma shows how to build the map $\phi_L:\,C\to\bP^1$.
We~start with the maps $\phi_{L_A}:\,A\to\bP^1$ and $\phi_{L_B}:\,B\to\bP^1$ that each map $x$ and $y$ to different points.
Without loss of generality, we  can assume that 
$$\phi_{L_A}(x)=\phi_{L_B}(x)=0\quad\hbox{\rm and}\quad \phi_{L_A}(y)=\phi_{L_B}(y)=\infty.$$
The~maps $\phi_{L_A}$ and $\phi_{L_B}$ then glue and give a map $\phi:\,C\to\bP^1$, which corresponds 
to {\em some} line bundle $L$ that restricts to $L_A$ and $L_B$. 
The restriction map $$\Pic^{\vec d}C\to  \Pic^{\vec d_A}A\times \Pic^{\vec d_B}B$$
has kernel $\bC^*$. Changing $L$ by an element $z\in \bC^*$
gives a map $\phi_{zL}$ such that 
$$\phi_{zL}|_A=\phi_{L}|_A\quad\hbox{\rm and}\quad \phi_{zL}|_B=z\phi_{L}|_B.$$
\end{Amplification}

 \begin{Review}\label{qfvefb}
Curves $A$ and $B$ do not have to be stable.
Suppose that $A$ is not stable\break 
(analysis for $B$ is the same).
Then $A$ contains a rational component that in $C$ is  attached to $x$ or $y$ (or both) which has
less than $3$ special points (marked points or nodes) once $x$ and $y$ are separated.
Since $C$ doesn't have a $1$-channel factorization, 
$A$ can't have a rational component with $1$ special point (in $A$) attached to both $x$ and~$y$.
Thus $A$ is at least semi-stable. We list the remaining possibilities in Figure~\ref{fwgwrg}:
\begin{figure}[htbp]
\includegraphics[width=4.5in]{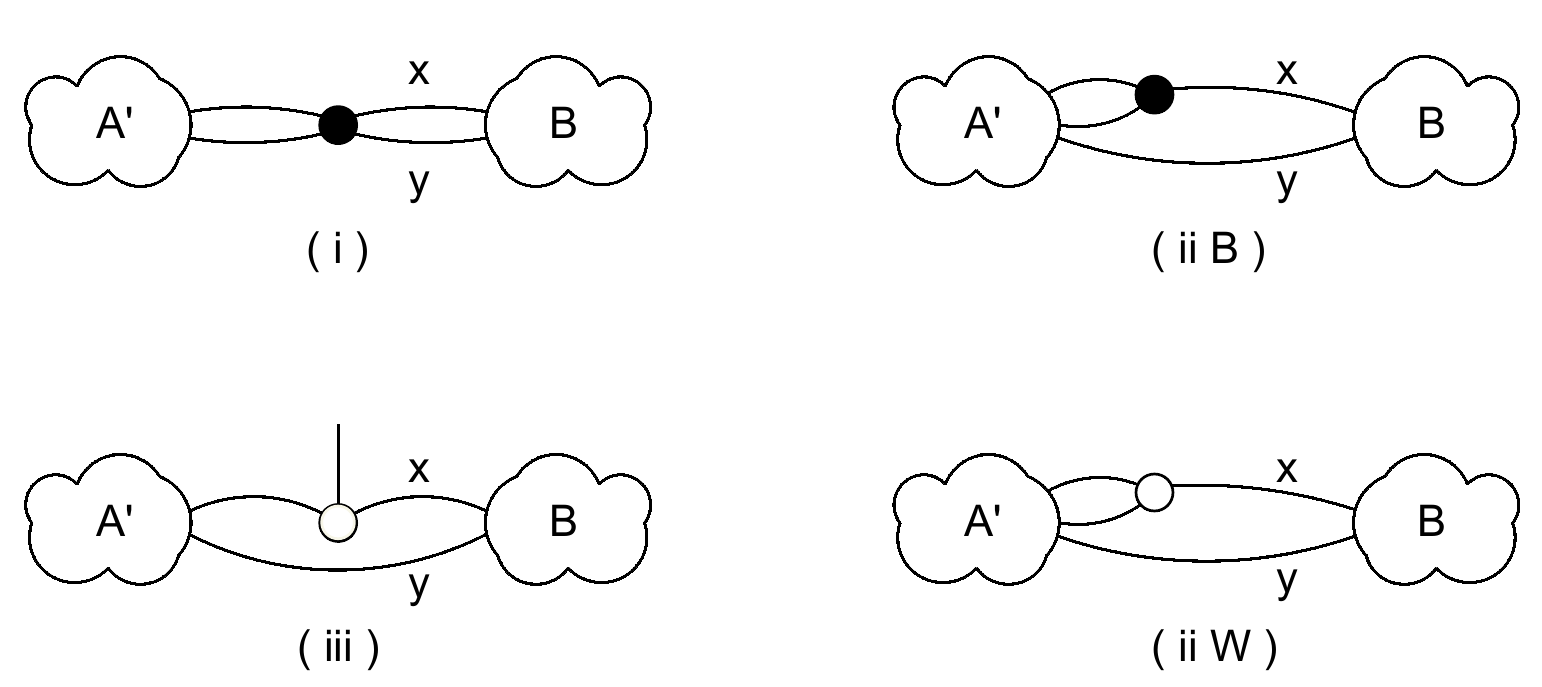}
\caption{}\label{fwgwrg}
\end{figure}
\end{Review}
\begin{enumerate}
\item[(i)] $A$ contains a $4$-pointed $\bP^1$ without marked points attached to $x$ and $y$
and to the rest of the subcurve $A'$ at $2$ points. 
Then  $d_{A'}\ge g_A$ by Lemma~\ref{wefwetb}~(1) if $g_A > 1$ or by Lemma~\ref{SFbSFh} and $n_A=2$ if $g_{A'}=0$. Then $d_{A'} + d_{\bP^1} = g_A+1$ and $d_{\bP^1} \ge 1$ since 
$\phi_L(x) \neq \phi_L(y)$. Thus $\bP^1$ has degree $1$.

\item[(iiW)] $A$ contains a $3$-pointed $\bP^1$ of degree $0$ without marked points attached to the rest of the subcurve $A'$
at $2$ points and also to $x$ (or $y$).
\item[(iiB)] $A$ contains a $3$-pointed $\bP^1$ of degree $1$ without marked points attached to the rest of the subcurve $A'$
at $2$ points and also to $x$ (or $y$).\footnote{Note that this $\bP^1$ can't have degree $>1$ by Lemma~\ref{wefwetb}~(1).}
\item[(iii)] $A$ contains a $3$-pointed $\bP^1$ with a marked point attached to the rest of the subcurve $A'$ at 
one point and also to $x$ (or $y$). This $\bP^1$ must have degree $0$ by Lemma~\ref{wefwetb}~(1).
\end{enumerate}

\begin{Theorem}\label{safbsfhsfn}
A curve $C$ with a two-channel factorization such that $n_A=g_A+3$ and $n_B=g_B+1$
is an MHV curve if and only if 
\begin{enumerate}
\item The stabilization $\tilde A$ of $A$ with its $n_A$ marked points is an MHV  curve that does not contain 
a degree $0$ rational component $C^{(0)}_i$ as in \ref{asrsrh} with both $x$ and $y$.
\item The curve $\tilde B$ obtained by adding to $B$ two extra marked points in addition to its $n_B$ marked points
(at the nodes separating $A$ from $B$) is an MHV curve.
\end{enumerate}
A (multivalued) inverse of the map $\bLa$ at 
a general point $(p_1,\ldots,p_n)\in M_{0,n}$ can be found as follows:
Take $L_A\in\bLa_{A}^{-1}\left((p_i)_{i\in N_A}\right)$. 
Then $p'=\phi_{L_A}(x)\ne \phi_{L_A}(y)=p''$.
Take~$L_B\in\bLa_B^{-1}\left(p', (p_i)_{i\in N_B}, p''\right)$.
Finally, we can find $L$ using Amplification~\ref{adgbarhR}. 
\end{Theorem}

\begin{proof}
If $A$ is stable then it is an MHV curve. Indeed,  
by Amplification~\ref{adgbarhR} we have a commutative diagram
\begin{equation}\label{asdvfv}
\begin{CD}
\Pic^{\vec d}C  @>\bLa>>  M_{0,n} \\
@V{\Res_A:\,L\mapsto L_A}VV                                                     @VV{\pi_A}V \\
\Pic^{\vec d_A}A                     @>\bLa_A>>         M_{0,n_A}
\end{CD}
\cooltag\end{equation}
where $\pi_A$ is the forgetful map,
and so the rational map $\bLa_A$ is dominant.
If $A$ is not stable then we claim that its stabilization $\tilde A$ is an MHV curve.
Indeed, cases (i) and (iiB) of \ref{qfvefb} are impossible
because the remaining part $A'$ would have genus $g_A-1$, degree $g_A$, but will contain $g_A+3$ points,
in contradiction with Lemma~\ref{wefwetb}~(2).
In~the remaining cases (iiW) and (iii)  of \ref{qfvefb}, $\tilde A$ is obtained by contracting a $\bP^1$ 
(or two~$\bP^1$, one attached to $x$ and another to $y$), both of degree $0$.
It follows that the map $\phi_L|_A$ factors through $\tilde A$ and we can argue as in the case when $A$ is stable.

\begin{Claim}
The curve $\tilde B$ obtained by  adding $x$ and $y$ to $B$ 
as extra marked points
is MHV.
\end{Claim}

\begin{proof}[Proof of the Claim]
Indeed, $\tilde B$ is clearly stable and we have already checked conditions (1) and (2)
in the Definition~\ref{KSLJDfksjdg}  of an MHV curve.\footnote{We will not use this but note that if $B$ is not stable then case (iii) of \ref{qfvefb} is impossible 
because attaching this $3$-pointed $\bP^1$ to $A$ instead of $B$ creates a subcurve of genus $g_A$, degree $g_A+1$ with $g_A+4$
marked points, which contradicts Lemma~\ref{wefwetb}~(2).}
We~need to check that 
$$\bLa_{\tilde B}:\,\Pic^{\vec d_B}B\to M_{0,n_B+2}$$ 
is dominant.
We use the fact that $\bLa$ is dominant and diagram \eqref{asdvfv}.
Fix a general line bundle $L_A$ on $A$ and fix $n_A$ points $\bLa_A(L_A)\subset\bP^1$.
Since these points include three different points in $\bP^1$, no automorphisms of $\PGL_2$ preserve them.
We adjust homogeneous coordinates on $\bP^1$ so that $\phi_{L_A}(x)=0$ and  $\phi_{L_A}(y)=\infty$. 
Thus $\bLa(L)$ for $L\in \Res_A^{-1}(L_A)$ is determined by points $\phi_L(p_i)$ for $i\in N_B$.
Since $\bLa|_{\Res_A^{-1}(L_A)}$ is a dominant map from a semi-abelian variety of dimension $n_B=g_B+1$,
we see that the following locus of points is dense: 
$$\{\phi_L(p_i)\}_{i\in N_B}\subset (\bP^1)^{n_B}.$$
The map $\bLa_{\tilde B}$ sends $L_B$ to an isomorphism class of the $(n_B+2)$-tuple 
of points $(0, \phi_L(p_i)_{i\in N_B},\infty)$ in $\bP^1$, equivalently to the class of 
$\{\phi_L(p_i)\}_{i\in N_B}$ modulo $\bC^*$. Therefore, $\bLa_{\tilde B}$ is dominant.
\end{proof}

It remains to show that if $A$ and $B$ are as in the theorem
then $C$ is an MHV curve and to construct a multivalued inverse of the scattering amplitude map~$\bLa$.
Start with a point $(p_1,\ldots,p_n)\in M_{0,n}$.
Since $\tilde A$ is an MHV curve, we can find a line bundle $L_A$ such that $\bLa_A(L_A)=(p_i)_{i\in N_A}$.
We lift $L_A$ to a line bundle on~$A$. By~Lemma~\ref{jhgkjgfk}, $p'=\phi_{L_A}(x)\ne \phi_{L_A}(y)=p''$.
Since $\tilde B$ is an MHV curve, we can find $L_B$ such that  $\bLa_B(L_B)=(p', (p_i)_{i\in N_A}, p'')$.
We finish using Amplification~\ref{adgbarhR}. 
\end{proof}

When $n_A=g_A+2$, $n_B=g_B+2$ (as in Example~\ref{SfbSfb}),
our result is less precise. 

\begin{Theorem}\label{ADvAeg}
Suppose that $C$ is an MHV curve with a two-channel factorization such that $n_A=g_A+2$ and $n_B=g_B+2$
which does not admit an alternative two-channel factorization as in Theorem~\ref{safbsfhsfn}.
Then  
\begin{enumerate}
\item The curve $\tilde A$ obtained by adding to $A$ an extra marked point 
(at one of the nodes separating $A$ from $B$) and stabilizing is an MHV curve that does not contain a
a degree $0$ rational component $C^{(0)}_i$ as in \ref{asrsrh} with both $x$ and $y$.
\item The same for $B$.
\end{enumerate}
\end{Theorem}

\begin{proof}
Suppose that $A$ is not stable. In cases (i) and (iii) of \ref{qfvefb},
the curve admits an alternative two-channel factorization
of type  considered in Theorem~\ref{safbsfhsfn}. Namely, in case (i) (resp.~(iii))
we move the $4$-pointed  (resp.~ the $3$-pointed) $\bP^1$ to~$B$.
This gives 
a two channel factorization such that $n_{A'}=g_{A'}+3$, $n_{B'}=g_{B'}+1$.
So~we~will assume that neither (i) nor (iii)  occurs for either $A$ or $B$.

\begin{Claim}
We can choose extra points $p$ and $q$ for $A$ and $B$ at one of the points $x$ or $y$ such that curves
$\tilde A$ and $\tilde B$ obtained by stabilization are MHV curves.
\end{Claim}

Indeed, if~$A$ or $B$ is not stable and case (iiB) occurs, we would have to add an extra point to that $\bP^1$.
Note that (iiB) can only occur at one point of $A$, $x$~or~$y$, and the same for $B$.
Indeed, otherwise $A'$ has genus $g_A-2$, degree $g_A-1$ and $g_A+2$ marked points, which contradicts Lemma~\ref{qwefv2e2}~(2).

If~(iiB) does not occur, we are going to decide where to add an extra point later in the proof of the claim.
We~stabilize the curves if (iiW) occurs but the extra point~$p$ is not attached to the corresponding 
$3$-pointed $\bP^1$.
The curves $\tilde A$ and $\tilde B$ are of course stable and we have already checked conditions (1) and (2)
in the definition of an MHV curve. We just need to check that $\bLa_{\tilde A}$ and $\bLa_{\tilde B}$ are dominant.
Since the situation is symmetric, we only consider $\tilde A$. Suppose first that (iiB) occurs.
From the commutative diagram \eqref{asdvfv}, the map $\bLa_A$ is dominant. Note that an extra point $p$ doesn't impose
any conditions and can be anywhere in $\bP^1$.  Thus $\bLa_{\tilde A}$ is dominant as well.
Next we suppose that (iiB) does not occur. For a general line bundle $L_A\in\Pic^{\vec d_A}A$,
the $n_A$ points $\phi_{L_A}(p_i)_{i\in N_A}$ are distinct and as we vary $L_A$  span a $(n_A-3)$-dimensional
open locus in $M_{0,n_A}$. It suffices to show that adding $\phi_{L_A}(x)$ (or $\phi_{L_A}(y)$) gives the locus of dimension $n_A-2=g_A$.
Equivalently, we claim that adding both $\phi_{L_A}(x)$ and $\phi_{L_A}(y)$ gives the locus of dimension $g_A$.
We argue by contradiction. If this is not the case then $\phi_{L_A}(x)$ and $\phi_{L_A}(y)$ are determined by positions of points 
$\phi_{L_A}(p_i)_{i\in N_A}$ up to a finite ambiguity. Positions of  $n_B$ points $\phi_{L_B}(p_i)_{i\in N_B}$  are then determined
by $g_B+1$ parameters as in Amplification~\ref{adgbarhR}.
Thus the locus of images of marked points has dimension at most $g_A+g_B<g$, which contradict the fact that $\Lambda$ is dominant.
\end{proof}

\section{Compactified Jacobians of MHV curves}\label{zdadthethfbsdh}

\begin{Review}
In this section we will globalize the scattering amplitude map 
by the diagram
$$\begin{CD}
\cPic^{MHV}_0\cC        @>\bLa>>         M_{0,n}\\
@VVV                                          \\     
\ocM^{\MHV}_{g,n}
\end{CD}$$
where 
$\ocM^{\MHV}_{g,n}\subset\ocM_{g,n}$ is an open
locus of MHV curves,
$\cPic^{MHV}\cC$ is a universal MHV Picard family over it,
$\cPic^{MHV}_0\cC\subset\cPic^{MHV}\cC$ is an open locus and 
$\bLa$
is 
the universal
scattering amplitude map. We also  compactify
$\cPic^{MHV}\cC$ by 
a~projective family $\ocPic^{MHV}\cC$ of compactified Jacobians.
Recall that $n=g+3$ and $d=g+1$.
\end{Review}

\begin{Review}
Consider the stack $\cQ_{g,n,d}$ of {\em quasi-maps} \cite{QM} of $C$ to $\bP^1$.
Its sections over a scheme~$S$ are triples 
$(\cC,\cL,\alpha)$, where $\cC\to S$ is a family of semi-stable curves of genus $g$ with $n$ marked points,
$\cL$ is a line bundle on $\cC$ of relative degree $d$ and $\alpha:\,\cO_\cC^{\oplus 2}\to\cL$
is a homomorphism not vanishing on each fiber. The morphism  
$$\nu:\,\cQ_{g,n,d}\to\ocM_{g,n}$$
sends $\cC$ to its stabilization. The stack $\cQ_{g,n,d}$ (and its higher rank analogues) contains
many useful open proper separated substacks.
\end{Review}

\begin{Definition}
Let $\ocM^{\MHV}_{g,n}\subset\ocM_{g,n}$ be the locus of families of stable curves such that all their  geometric fibers are MHV curves for some choice of $\vec d$.
Let $\cQ_{g,n,d}^{\MHV}\subset \cQ_{g,n,d}$ be the locus of families of quasi-maps such that, for any geometric fiber 
$(C,L,\alpha)$, the curve 
$C$ is stable and $(L,\alpha)$ satisfies Definition~\ref{KSLJDfksjdg} (1)--(3).
\end{Definition}

\begin{Theorem}\label{wgearhstjsrjrj}
$\ $
\begin{enumerate}
\item 
$\ocM^{\MHV}_{g,n}$ is an open substack of 
$\ocM_{g,n}$.
\item 
$\cQ_{g,n,d}^{\MHV}$ is an open substack of the proper and separated Deligne--Mumford stack
$\oQ_{g,n,d}$ of stable quotients as defined in \cite{MOP}.
\item The morphism $\nu:\,\cQ_{g,n,d}^{\MHV}\to\ocM^{\MHV}_{g,n}$ is smooth, of relative dimension $n$.
\item Consider the action of $\GL_2$ 
 on $\cQ_{g,n,d}^{\MHV}$ via its action on $\cO_\cC^{\oplus 2}$.
 The stabilizer of every point is the diagonal torus $\bC^*$.
The quotient stack $\cPic^{MHV}_0\cC:=[\cQ_{g,n,d}^{\MHV}/PGL_2]  $ is 
an open substack in the rigidified Picard stack of pairs $(C,L)$, where $L$ is a degree $d$ line bundle on~$C$.
\item Let $\Delta_{ij} \subset(\bP^1)^n$  be diagonals ($1\le i<j\le n$). We have a commutative diagram
\smallskip
$$\begin{CD}
\cQ_{g,n,d}^{\MHV} @>>>  (\bP^1)^n\setminus\bigcup_{i,j}\Delta_{ij} \\
@VV{/PGL_2}V                                                     @VV{/{PGL_2}}V\\
\cPic^{MHV}_0\cC        @>\bLa>>         M_{0,n}\\
@VVV                                          \\     
\ocM^{\MHV}_{g,n}
\end{CD}$$
\end{enumerate}
\end{Theorem}

\begin{proof}
Dualizing the homomorphism $\alpha:\,\cO_C^{\oplus 2}\to L$
gives an exact sequence
\begin{equation}\label{sajkhfbaskjhfb}
0\to L^*\to \cO_C^{\oplus 2}\arrow^q Q\to 0.
\cooltag\end{equation}
The moduli space of stable quotients $\oQ_{g,n,d}$ parametrizes data \eqref{sajkhfbaskjhfb} such that  
$Q$ is locally free at the nodes and at the marked points
and has positive degree along all strictly semistable components of $C$. 
All these conditions are clearly satisfied in our case, in fact $q=\alpha$.
From \cite{MOP} we conclude that 
\begin{enumerate}
\item There is a universal curve $\cC$ over $\oQ_{g,n,d}$ with $n$ sections
and a universal quotient $0\to \cS\to\cO_\cC^{\oplus 2}\to\cQ\to0$
over $\cC$ with an invertible sheaf $\cS$.
\item The  map $\nu:\,\oQ_{g,n,d}\to\ocM_{g,n}$ that sends $\cC$ to its stabilization is proper.
\item Evaluating sections $q(1,0)$ and $q(0,1)$ at the marked points $p_1,\ldots,p_n$
gives evaluation maps $\ev_i:\,\oQ_{g,n,d}\to\bP^1$, which we combine into one morphism
$$\oQ_{g,n,d}\to(\bP^1)^n.$$
\item All structures above are equivariant under $\PGL_2$.
\item $\oQ_{g,n,d}$ has a $2$-term obstruction theory relative to $\nu$ given by $\RHom(\cS,\cQ)$.
\end{enumerate}
In~particular, $\nu$ is smooth at $(C,q)$ if $\Ext^1(S,Q)=0$, which in our case follows from the short exact sequence \eqref{sajkhfbaskjhfb} 
and $H^1(C,L)=0$.
Since $\nu:\,\cQ_{g,n,d}^{\MHV}\to\ocM_{g,n}$ is smooth, its image $\ocM^{\MHV}_{g,n}$ is open.
\end{proof}

\begin{Review}
Next we show that $\cPic^{MHV}_0\cC$ is separated (note that 
the quotient by a free action of an algebraic group is often not separated).
We will introduce a natural polarization $A$  on MHV curves and study slope-stability of line bundles $L$
with respect to it. It is well-known that 
slope stability on reducible curves
boils down to  {\em Gieseker's basic inequality}:
for any proper subcurve $Y\subset C$, we need to show
\begin{equation}\label{Gieseker_Ineq}
\left|d_Y-a_Y{d_C\over a_C}\right|<{1\over2}\#Y,
\cooltag\end{equation}
where 
$$d_Y=\deg L|_Y,\quad a_Y=\deg A|_Y\quad\hbox{\rm and}\quad 
\#Y:=|Y\cap\overline{ X\setminus Y}|.$$
Classically, the most studied case was $d_C=g-1$, $A=\omega_C$, and no marked points.
This is a convenient choice because ${d_C\over a_C}={1\over 2}$.
In the MHV Brill-Noether theory, when $d_C=g+1$ and marked points are present,
an equally convenient choice is a fractional polarization 
$$A=\omega_C+{4\over n}(p_1+\ldots+p_n),$$
which also gives ${d_C\over a_C}={1\over 2}$.
The $\bQ$-line bundle $A$ is ample as long as  $C$ is stable curve without rational tails
(subcurves of arithmetic genus $0$ attached to the rest of the curve at one point),
which doesn't happen for  MHV curves by Lemma~\ref{qwefv2e2}. 
\end{Review}

\begin{Lemma}
A line bundle is $A$-stable if and only if 
\begin{equation}\label{dwcsv}
d_Y>g_Y-1+{2\over n}n_Y,
\cooltag\end{equation}
for every proper connected subcurve $Y\subset C$, 
where $g_Y$ is the arithmetic genus of $Y$ and $n_Y$ is the number of marked points on~it.
For semi-stability, the inequality is not strict . 
\end{Lemma}

\begin{proof}
For every proper subcurve $Y\subset C$, we have to check that 
$$
\left|d_Y-{1\over 2}a_Y\right|<{1\over2}\#Y.
$$
For the complementary subcurve $Y^c:=\overline{ X\setminus Y}$,
we have $d_Y-{1\over 2}a_Y=-d_{Y^c}+{1\over 2}a_{Y^c}$.
Thus we just have to show that 
${1\over 2}a_Y-d_Y<{1\over2}\#Y$
for every proper subcurve $Y$, in~fact for every proper connected subcurve,
which is equivalent to \eqref{dwcsv}.
\end{proof}

\begin{Lemma}
\item For an MHV curve $C$,
every line bundle $L\in\Pic^{MHV}C$ is $A$-stable.
\end{Lemma}

\begin{proof}
To verify \eqref{dwcsv},
we consider several cases.
\begin{enumerate}
\item $d_Y\ge g_Y+2$. The equation \eqref{dwcsv} is automatic.
\item $d_Y=g_Y+1$. The equation \eqref{dwcsv} holds unless $n_Y=n$, in which case
$g_Y=g$ by Lemma~\ref{wefwetb}
and therefore $Y=C$ by Lemma~\ref{qwefv2e2}, contradiction.
\item $d_Y\le g_Y$. Then $d_Y=g_Y=0$ by Lemma~\ref{wefwetb} and thus $n_Y\le 1$ by 
Lemma~\ref{SFbSFh}. Equation \eqref{dwcsv} follows.
\end{enumerate}
This finishes the proof.
\end{proof}

\begin{Corollary}
The family $\cPic^{MHV}_0\cC$ over $\ocM^{\MHV}_{g,n}$ is compactified by a projective 
(and hence separated) family
$\ocPic^{MHV}\cC$ of compactified Jacobians.  
Its geometric points represent $\gr$-equivalence
classes of $A$-semistable admissible sheaves on an MHV curve $C$. 
\end{Corollary}

\begin{Amplification}
If all irreducible components of $C$ are rational then 
$\ocPic^{MHV}\cC$ is a stable toric variety of an algebraic torus $\Pic^0C\simeq(\bC^*)^g$
and can be described as follows.
Choose one on-shell diagram for $C$, i.e.~endow each vertex $i$ of the dual graph
$\Gamma$  of $C$ with degree $d_i$ such that $\sum d_i=d$. 
For every vertex $i\in\Gamma$, 
\begin{enumerate}
\item 
$e_i$ is the number of edges (count each loop twice);
\item 
$l_i$ is the number of legs, i.e.~
the number of marked points on the corresponding irreducible 
component of~$C$;
\item
$\theta_i$ is the Oda--Seshadri number \cite{Oda}, in our case 
$$\theta_i=-d_i+{4l_i\over dn}+{n(e_i-2)\over 2d}.$$
\end{enumerate}
Let $C_0$ and $C_1$ be spaces of chains of the graph $\Gamma$ (with arbitrary coefficients).
Let $$\partial:\,C_1\to C_0$$ be the differential and let $H_1$ be the first homology group. 
We can view $\theta=(\theta_i)$ as a vector in $C_0(\Gamma,\bQ)$. 
Consider the tiling of the affine space $\partial^{-1}(\theta)\subset C_1(\Gamma,\bQ)$ by the hyperplanes $x_i={1\over2}+\bZ$, where $x_i$
are the natural coordinates on $C_1$ (one coordinate for each edge of the graph). 
This affine space is parallel to the homology group $H_1(\Gamma,\bQ)$ and the tiling has to be viewed modulo the action of $H_1(\Gamma,\bZ)$.
Irreducible components
of the compactified Jacobian are given by the
$g$-dimensional polytopes of the tiling and the rest of the tiling determines how they are glued.
\end{Amplification}

\begin{proof}
Follows from~\cite{Sim} and \cite{Oda}, see also \cite{Alex}.
\end{proof}

\begin{Example}\label{sdfvwfv}
Let 
$C=C_1\cup\ldots\cup C_4$
be the following curve of arithmetic genus~$1$: the wheel of four $\bP^1$'s with marked points $p_i\in C_i$ for $i=1,\ldots,4$.
The~Picard group has two MHV components:
$$\Pic^{MHV}C=\Pic^{1,0,1,0}C\bigcup \Pic^{0,1,0,1}C.$$
Consider the first component. 
Since $L$ has degree $0$ on components $C_2$ and $C_4$,
the map $\phi_L:\,C\to\bP^1$ contracts  them to points, say $\phi_L(C_2)=0$
and 
$\phi_L(C_4)=\infty$. The restriction of $\phi_L$ to $C_1$ and $C_3$ is an isomorphism. Thus $\phi_L$ is completely determined by points $x=\phi_L(p_1)$
and $y=\phi_L(p_3)$. We have an identification
$$\Pic^{1,0,1,0}C\simeq\bC^*,\quad z=y/x.$$ 
The scattering amplitude map
$\bLa:\,\Pic^{1,0,1,0}C\to\oM_{0,4}\simeq\bP^1$ is given by 
the cross-ratio of points $\phi_L(p_1),\ldots,\phi_L(p_4)$ varying with $L$.
This map is an isomorphism:
$$\lambda=[x:0;y:\infty]=
{\infty-x\over 0-x}\cdot {0-y\over \infty-y}=z.
$$
The scattering amplitude form is 
${dz\over z}={d\lambda\over\lambda}$
with log poles at $0$ and $\infty$. 
Next we compute the compactified Jacobian. There are $2$ types of stable line bundles
and $4$ types of semi-stable line bundles (in two $\gr$-equivalence classes),
see Figure~\ref{adfbadfh}.
\begin{figure}[htbp]
\includegraphics[width=\textwidth]{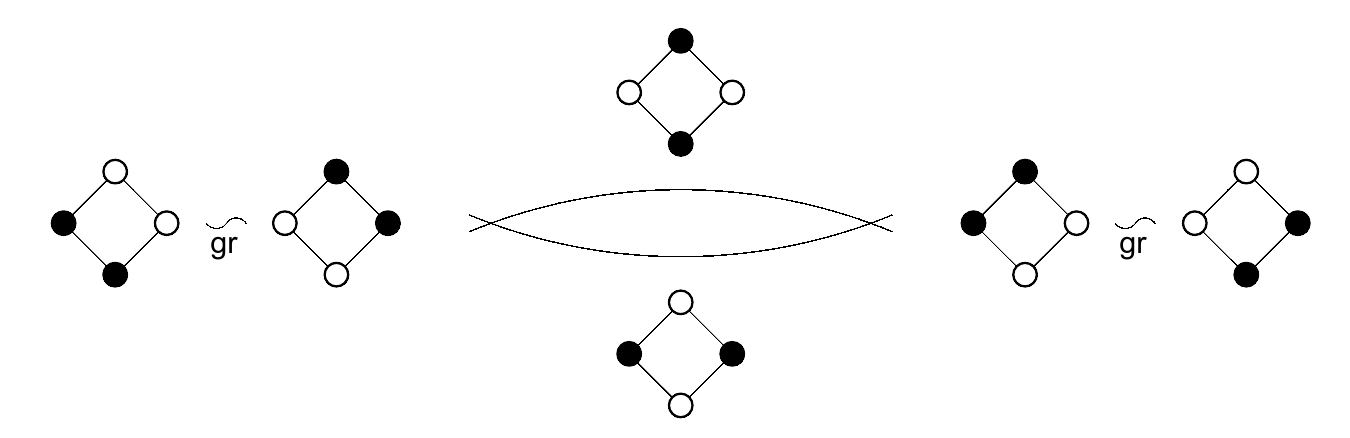}
\caption{MHV compactified Jacobian in genus $1$}
\label{adfbadfh}
\end{figure}
Stable components are $\bP^1$'s and $\gr$-equivalence classes of semi-stable line bundles correspond to their intersection 
points in the compactified Jacobian
$$\oPic^{MHV}C=\oPic^{1,0,1,0}C\bigcup \oPic^{0,1,0,1}C=\bP^1\cup\bP^1.$$
One can visualize topology on $\oPic^{MHV}C$ through ``chip-firing'' of black chips.
\end{Example}

\begin{Remark}
The MHV condition implies $A$-stability but the converse 
is not  true. For example, curves with $1$-channel factorization
can have $A$-stable  line bundles. 
Moreover, even the compactifiied Jacobian of an MHV curve
will typically contain non-MHV components
contracted by the scattering amplitude map $\bLa$ to loci of smaller dimension either in $M_{0,n}$
or in its boundary (for some compactification of~$M_{0,n}$). 
Nevertheless, we view $A$-stability and MHV conditions as close.
\end{Remark}

\begin{Example}\label{asfgasfgafh}
Consider a genus~$2$ curve with $5$ rational components,
each with a marked point, see the left of Figure~\ref{adfbadSbfsfh} for on-shell diagrams
of all possible stable line bundles.
\begin{figure}[htbp]
\includegraphics[height=2.45in]{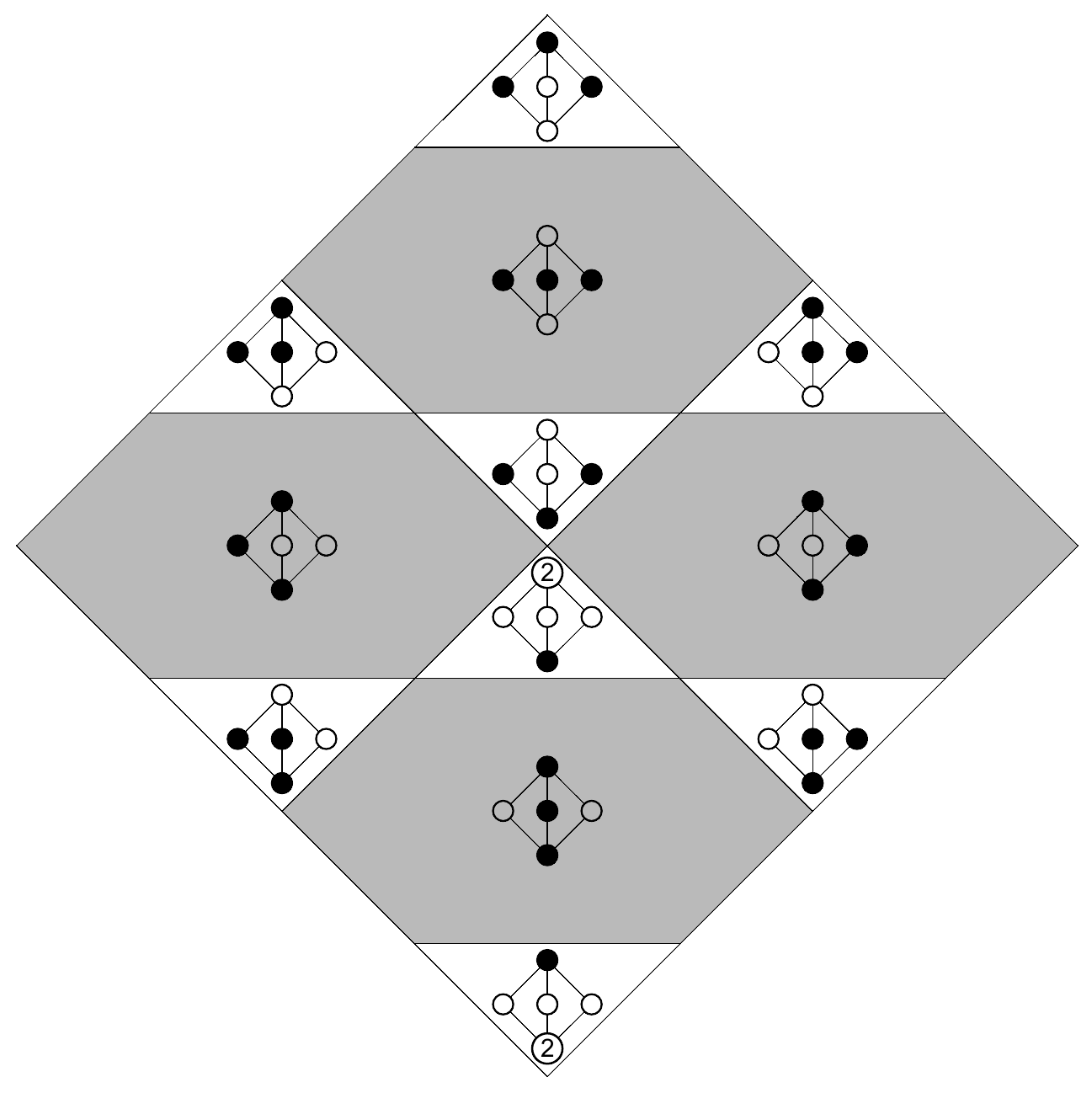}
\includegraphics[height=2.45in]{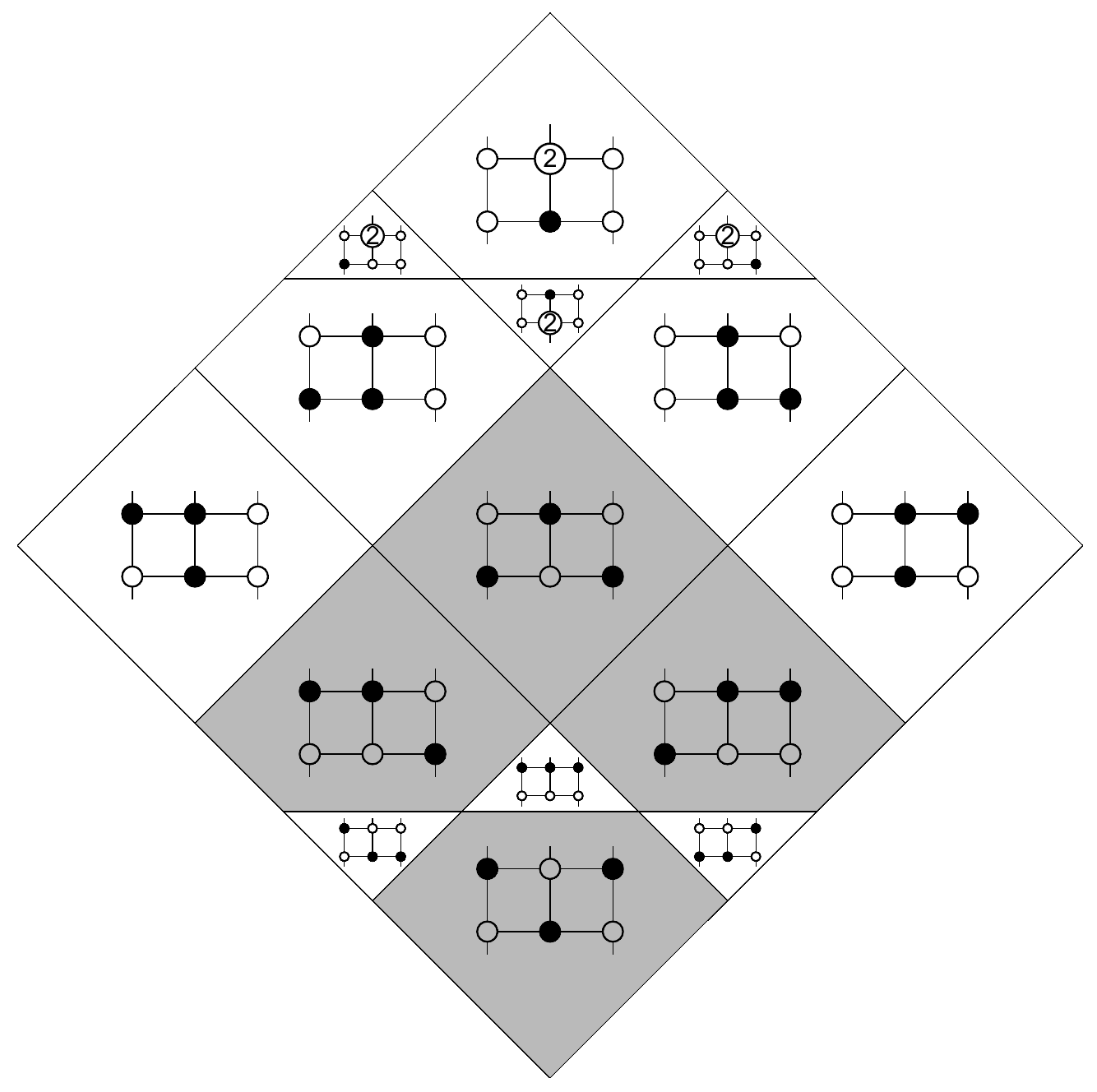}
\caption{MHV compactified Jacobians in genus $2$.}
\label{adfbadSbfsfh}
\end{figure}
There are no strictly semistable line bundles in this case. 
The compactified Jacobian contains $4$ MHV  (shaded gray)  and $8$ non-MHV irreducible components.
They are glued in an alternating fashion, MHV to non-MHV, according to Figure~\ref{adfbadSbfsfh}, which is a fundamental parallelogram of a tiling of a $2$-torus: 
components adjacent to a side of the parallelogram intersect components adjacent to the opposite side.
The compactified Jacobian is a stable toric surface of the algebraic torus $\Pic^0C\simeq(\bC^*)^2$
with polygons representing its irreducible components, which are toric surfaces.
MHV components are isomorphic to $\Bl_3\bP^2$ (hexagons) 
and non-MHV components are isomorphic to~$\bP^2$ (triangles).
Crossing each of the six walls corresponds to firing a black chip across the corresponding edge of the on-shell diagram. 
There are two types of non-MHV components, with or without an irreducible component where $L$ 
has degree $2$.
Under the scattering amplitude map, surfaces of the first type contract to curves intersecting the interior of $M_{0,5}$
and surfaces of the second type contract to curves in the boundary.
\end{Example}

\begin{Example}
Another genus~$2$ example is on the right of Figure~\ref{adfbadSbfsfh}.
There are $4$ MHV components (shaded gray) and $11$ non-MHV components, which are all
mapped
to the boundary of $M_{0,5}$ by the scattering amplitude map.
Three of the MHV components are isomorphic to $\Bl_2\bP^2$ (pentagons) and one to $\bP^1\times\bP^1$.
\end{Example}

\section{Scattering via moduli of parabolic bundles}\label{sdvqefv}

\begin{Review}
Let $C$ be a smooth  hyperelliptic  curve of genus $g\ge2$ with a hyperelliptic involution 
$p\mapsto \tau(p)$ (whenever it exists, the hyperelliptic involution is unique).
The quotient by $\tau$ is the double cover $\phi_h:\,C\arrow^{2:1}\bP^1$
associated with the line bundle
$$h=\cO(p+\tau(p))\in\Pic^2C,\quad p\in C.$$
 Let $p_1,\ldots,p_n\in C$ be distinct marked points, where $n=g+3$. 
In~our approach the marked points are decoupled from $2g+2$ Weierstrass points (fixed points
of the hyperelliptic involution). 
To simplify the analysis, we  make an assumption:
\end{Review}

\begin{Assumption}\label{Sfbzfdhadfh}
Points $\phi_h(p_1),\ldots,\phi_h(p_n)$ are different, i.e. 
 $p_i\ne\tau(p_j)$ for $i\ne j$.
Thus the special feature of the hyperelliptic case is existence of a well-defined point
\begin{equation}\label{mjhvjhv}
o(p_1,\ldots,p_n)=(z_1,\ldots,z_n)=(\phi_h(p_1),\ldots,\phi_h(p_n))\in M_{0,n}
\cooltag\end{equation}
that shouldn't be confused with  the scattering amplitude map 
$$\bLa:\,\Pic^{g+1}C\dashrightarrow M_{0,n},\quad
L\mapsto (\phi_L(p_1),\ldots,\phi_L(p_n)).$$
\end{Assumption}

\begin{Review}
Let $\Bun(\bP^1;z_1,\ldots,z_n)$ be the smooth algebraic stack of 
{\em quasi-parabolic vector bundles} $F$ on $\bP^1$ of rank $2$ with trivial determinant \cite{Pauly}.
Recall that a quasi-parabolic structure on a vector bundle is  a choice of a line $V_i\subset F|_{z_i}$ over each marked point. This data determines a ruled surface $\bP(F)\to\bP^1$ of even degree and points
$$q_i=\bP(V_i)\subset\bP(F|_{z_i})\quad\hbox{\rm for $i=1,\ldots,n$}.$$
\end{Review}

\begin{Definition}
There is a standard morphism of stacks 
$$\bbLa:\,\Pic^{g+1}C\to \Bun(\bP^1;z_1,\ldots,z_n),\quad
L\mapsto F=(\phi_h)_*L.$$
To see that $(\phi_h)_*L$  gives an object of  $\Bun(\bP^1;z_1,\ldots,z_n)$, 
we check that its determinant vanishes.
It suffices to check
that $(\phi_h)_*\cO$ has determinant of degree $-(g+1)$. As with any double cover, $(\phi_h)_*\cO\simeq\cO\oplus \cM^{-1}$, where 
$\cM^{\otimes 2}\simeq\cO(B)$, where $B$ is a branch divisor. But the number of branch points is $2g+2$
(c.f.~\cite[Ex.~IV.2.6]{Hartshorne}.)

To define parabolic lines, note that 
$F|_{z_i}=L|_{p_i}\oplus L|_{\tau(p_i)}$ 
(or $L_{p_i}/\cm^2_{p_i}$ if $p_i$ is a Weierstrass point).
The line $V_i\subset F|_{z_i}$ is the kernel of the surjection $F|_{z_i}\to L|_{p_i}$.
\end{Definition}

\begin{Review}
A generic parabolic bundle $F$ is a trivial bundle $\bP^1\times\bC^2$ with lines $\{z_i\}\times V_i$ for  $V_i\subset\bC^2$\footnote{Indeed, the trivial vector bundle on $\bP^1$ does not deform, therefore
these parabolic bundles are parametrized by a non-empty open substack of $\Bun(\bP^1;z_1,\ldots,z_n)$. On the other hand,
all parabolic bundles of rank $2$ and degree $0$ can be deformed to a parabolic bundle with trivial $F$.} .
The corresponding ruled surface is the product $\bP^1\times\bP^1$ with points $(z_1,q_1),\ldots,(z_n,q_n)\in\bP^1\times\bP^1$,
where $q_i=\bP(V_i)$.
Thus we have a birational map $\Xi:\,\Bun(\bP^1;z_1,\ldots,z_n)\dashrightarrow M_{0,n}$
defined as in Figure~\ref{shasrhar}.
\end{Review}

\begin{Lemma}\label{SSRHS}
The scattering amplitude map of a hyperelliptic curve $C$ is a composition 
$$\bLa:\,\Pic^{g+1}C\mathop{\arrow}^{\bbLa} \Bun(\bP^1;z_1,\ldots,z_n)\mathop{\dashrightarrow}^\Xi M_{0,n}.$$
The bundle $\bbLa(L)$ has splitting type $\cO\oplus\cO$ away from the theta-divisor $E$ of Corollary~\ref{zxbzfn}.
\end{Lemma}

\begin{proof}
By Grothendieck's theorem, $(\phi_h)_*L\simeq\cO(-s)\oplus\cO(s)$ for some~$s$. We~require that $s=0$,
which is equivalent to the following conditions (i) and (ii):
\begin{enumerate}
\item[(i)]$H^0(\bP^1, F)=H^0(C,L)=2$, i.e.~$L$ is not special. This happens outside of the locus $W$ of
Corollary~\ref{zxbzfn}.
\item[(ii)] $\RHom(F,\cO(-1))=0$. By Grothendieck--Serre duality, this is equivalent to vanishing of $\RHom(L,\cO(-h+K+2h))=R\Gamma(L^*(K+h))$.
This means that $L$ has to be away from the  theta-divisor
$$K+h-x_1-\ldots-x_{g-1}=h+\tau(x_1)+\ldots+\tau(x_{g-1})\in\Pic^{g-1}C,$$
where we use that $K\sim(g-1)h$. In fact this theta-divisor is 
the divisor $E$ of Corollary~\ref{zxbzfn}, which contains $W$.
\end{enumerate}
Choosing a basis $s_1,s_2$ of $H^0(C,L)$ is equivalent to choosing a splitting $F\simeq\cO\oplus\cO$.
The scattering amplitude map $\bLa$ is given by $([s_1(p_1):s_2(p_1)],\ldots,[s_1(p_n):s_2(p_n)])$,
which is equal to the slopes of $V_1,\ldots,V_n$ inside $\bC^2$. Thus $\bLa=\Xi\circ\bbLa$.
\end{proof}

\begin{Review}[stability]
In view of Lemma~\ref{SSRHS}, in the hyperelliptic case the scattering amplitude map
combines effects of the birational 
morphism $\Xi$, which as we will see only depends on the point $o=(z_1,\ldots,z_n)\in M_{0,n}$
and the map~$\bbLa$, which is a finer invariant in the hyperelliptic case.
Moreover,  the divisor $E$ of $\Pic^dC$ is mapped to a divisor by $\bbLa$ but is collapsed 
to the point $o$ by $\bLa$,
which creates a singularity for the probability measure (see discussion of the real case in Section~\ref{sdcqwv}).
However, at the moment the target of $\bbLa$ is the stack of quasi-parabolic bundles.
To make $\bbLa$ more concrete, we will choose a projective model for the stack
(at the price of making $\bbLa$ a rational~map).
There is a notion of slope-stability that depends on a choice of a parabolic weight.
Concretely, the weight $\vec\alpha$ is a sequence $\alpha_1,\ldots,\alpha_n$ of real numbers with $0<\alpha_i\le {1\over 2}$.
A quasi-parabolic bundle $(F;V_1,\ldots,V_n)$ is $\vec\alpha$-stable (resp.~semistable) parabolic bundle
if and only if every line sub-bundle $L\subset F$ satisfies the  {\em slope inequality}
\begin{equation}\label{mb,hjb,hb}
k+\sum_{i\in I}\alpha_i<{1\over 2}\sum\limits_{i=1}^n\alpha_i\quad (\hbox{\rm resp.}\quad \le),
\cooltag\end{equation}
where $k=\deg L$ and $I\subset\{1,\ldots,n\}$ is a subset of indices such that $L|_{z_i}=V_i$.
We denote the corresponding moduli space by $\Bun_{\vec\alpha}(\bP^1;z_1,\ldots,z_n)$ or simply
by  $\Bun_\alpha(\bP^1;z_1,\ldots,z_n)$ if $\alpha_1=\ldots=\alpha_n=\alpha$.
Here is a summary of  wall-crossing with chambers given by inequalities \eqref{mb,hjb,hb}. 
\end{Review}

\begin{Review}
First we consider the {\em standard case}
$\sum\limits_{i=1}^n\alpha_i>2$ as in 
\cite{Bauer, Mukai,Casagrande,Kumar,Araujo}. 
If $\vec\alpha$ is strictly semistable
then
$$\Pic\Bun_{\vec\alpha}(\bP^1;z_1,\ldots,z_n)=\bZ^{n+1}.$$
The Fano model (i.e.~ the model with $-K$ ample) is 
\begin{equation}\label{asdvqfv}
\Bun_{1\over2}(\bP^1;z_1,\ldots,z_n),
\cooltag\end{equation}
which is a smooth variety if $n$ is odd 
and has isolated singularities
if $n$ is even. 
As~$\alpha$ increases from $2\over n$ to $1\over 2$, $\Bun_{\alpha}(\bP^1;z_1,\ldots,z_n)$
undergoes  the anticanonical minimal model program, i.e.~
the sequence of birational transformations that makes $-K$ more and more positive at every step. It proceeds as follows~\cite{Bauer}:
\begin{enumerate}
\item $\Bun_{2\over n}=\bP^{n-3}$ with $(\bP^1;z_1,\ldots,z_n)$ embedded into $\bP^{n-3}$ as a rational normal curve of degree $n-3$.
The MMP starts with
$\Bun_{{2\over n}+\eps}=\Bl_{z_1,\ldots,z_n}\bP^{n-3}$. 
\item The first birational transformation ``antiflips'' several $\bP^1$'s, namely lines connecting points $z_1,\ldots,z_n$ pairwise and the rational normal curve.
Each of these $\bP^1$'s is  blown-up and the exceptional divisor contracted onto~$\bP^{n-5}$.
\item On the following steps, we antiflip certain $\bP^k$'s analogously.
\item MMP stops when $-K$ becomes big and nef on $\Bun_{{1\over 2}-\eps}$, in fact even ample when $n$~is odd.
If~$n$~is even, the anticanonical model
$\Bun_{1\over 2}$ is obtained by contracting certain $K$-trivial $\bP^{n-4\over 2}$'s in $\Bun_{{1\over 2}-\eps}$ to  
singular points.
\end{enumerate}
\end{Review}

\begin{Review}
The {\em special case} $\sum\limits_{i=1}^n\alpha_i<2$ was studied in \cite{Moon}, in which case
$$\Pic\Bun_{\vec\alpha}(\bP^1;z_1,\ldots,z_n)=\bZ^n$$ 
if $\vec\alpha$ is strictly semistable.
The Fano model  is the symmetric GIT quotient 
$$(\bP^1)^n//\PGL_2,$$
which is smooth if $n$ is odd and with isolated singularities
if $n$ is even. More generally, in the special case
$\Bun_{\vec\alpha}(\bP^1;z_1,\ldots,z_n)$ is the GIT quotient $(\bP^1)^n//_{\vec\alpha}\PGL_2$
with respect to fractional polarization $\vec\alpha$.
The map   
$$\Xi^{-1}:\,M_{0,n}\dashrightarrow \Bun(\bP^1;z_1,\ldots,z_n)$$
is equivalent to the natural embedding 
$\Xi^{-1}:\,M_{0,n}\hookrightarrow (\bP^1)^n//_{\vec\alpha}\PGL_2$.
\end{Review}

\begin{Review}
The transition from the special (2)  to the standard (1) cases goes as follows.
\begin{enumerate}
\item[(2)] All stable parabolic bundles have splitting type $\cO\oplus\cO$. 
The product $\bP^1\times\bP^1$ with points $(z_1,z_1),\ldots,(z_n,z_n)$ is stable
giving the point $o\in M_{0,n}$. 
\item[(1)] The product $\bP^1\times\bP^1$ above is 
destabilized by $\cO(-1)$ embedded by the 
sequence $0\to \cO(-1)\to\cO\oplus\cO\to\cO(1)\to0$. The~point $o$ is blown-up.
The 
exceptional divisor parametrizes
parabolic bundles of   type $\cO(-1)\oplus\cO(1)$.
\end{enumerate}
We summarize the discussion by the following diagram:
\end{Review}

\bigskip
$$
\begin{matrix}
{1\over 2}   & \ldots\ldots\ldots\ldots  & {2\over n}+\varepsilon &  & {2\over n}    & & {2\over n}-\varepsilon\\[0.3cm]
\framebox{\rm\Small Fano}&  \framebox{\rm\Small  MMP} &\framebox{\rm\Small  log Fano}&  \framebox{\rm\Small  Big Bang} &\framebox{\rm\Small  Fano}&  \framebox{\rm\Small  Bigger Bang}&\framebox{\rm\Small  Fano}\\[0.3cm]
\Bun_F               &   \dashrightarrow\ldots\dashrightarrow         &     \Bl_n\bP^{n-3}                    &         \arrow   & \bP^{n-3}   &    \dashrightarrow      & (\bP^1)^n\!/\!/\!\PGL_2\\[0.3cm]
& 
&
&&&&\rotatebox{90}{$\hookrightarrow$}\\[0.3cm]
&
\Bun(\bP^1;z_1,\ldots,z_n)
&&&\mathop{\dashrightarrow}\limits^{\Xi}&&M_{0,n}\cr
\end{matrix}
$$
\medskip

\begin{Example} Suppose $g=2$, $n=5$. Apart from $\bP^2$ there are only two models:
$$\Bun_{{2\over5}+}=\Bun_{1\over2}\simeq\Bl_{z_1,\ldots,z_5}\bP^2=\dP\qquad\hbox{\rm (standard)}$$
and 
$$\Bun_{{2\over5}-}\simeq\Bl_4\bP^2=\dPf\qquad\hbox{\rm (special),}$$
the quartic and quintic del Pezzo surfaces.
The morphism 
$$\Xi:\,\dP\to\dPf\simeq\oM_{0,5}$$ 
contracts the conic through $z_1,\ldots,z_5$ to the point $o=(z_1,\ldots,z_5)\in M_{0,5}$.
This morphism can also be described as projection of 
$z_1,\ldots,z_5\in\bP^2$ to $\bP^1$ from a varying point of $\bP^2$ (cf.~\cite{CT_Cont}).
We will further study this example in Section~\ref{sfgasgsRH}.
\end{Example}

Enhancing $\bLa$ by $\bbLa$ corresponds to the transition
from the special (2)  to the standard (1) case of projective moduli of parabolic bundles.
In the case of  the Fano model, we have the following description of the indeterminancy locus of~$\bbLa$.

\begin{Theorem}\label{wefvev}
Consider the induced  map $\bbLa:\,\Pic^{g+1}C\dashrightarrow \Bun_{1\over2}(\bP^1;z_1,\ldots,z_n)$.
\begin{enumerate}
\item $\bbLa$ is regular away from   loci $U(k,I)$ (of codimension at least~$2$) of line bundles
$$\cO(kh+\sum_{i\in I} p_i+x_1+\ldots+x_j)\subset\Pic^{g+1}C\quad\hbox{\rm for}\quad x_1,\ldots,x_j\in C,$$
$$g+1=2k+|I|+j,\qquad j<{g-1\over 2}.$$
Loci $U(k,I)$ with equality $j = {g-1\over 2}$ give strictly semistable parabolic bundles.
\item $\bbLa^*(-K)\equiv 4\Theta$.
\end{enumerate}
\end{Theorem}

\begin{proof}
First we check (1) using an argument from \cite{Kumar}. 
Suppose $F=\bbLa(L)$ is unstable (resp.~strictly semi-stable) and take its
destabilizing line sub-bundle on $\bP^1$ of degree $k$ that contains  $V_i$ for 
$i\in I$. Pulling this line bundle back to $C$ and using adjointness gives an injection
$$f:\,\cO_C(kh)\mathop{\hookrightarrow} L$$
 such that $f(\cO_C(kh))$ vanishes at each point $p_i$ for $i\in I$. Thus we can write
$$L=\cO\left(kh+\sum_{i\in I}p_i+x_1+\ldots+x_j\right)$$
for some points $x_1,\ldots,x_j\in C$. Equating degrees gives
$$g+1=2k+|I|+j.$$
The slope inequality gives $j<{g-1\over 2}$ (equal for strictly semistable).
This gives loci listed in the theorem.
To show (2), notice that the ramification locus of $\bbLa$ and $\bLa$ is the same by (1), namely the divisor $R$.
By Riemann--Hurwitz and Corollary~\ref{zxbzfn}, 
$$\bbLa^*(-K)\sim -K_{\Pic^{g+1}C}+R\equiv 4\Theta.$$
This finishes the proof\footnote{
The papers \cite{Kumar,Araujo} contain the formula equivalent to $\bbLa^*(-K)\equiv 4^g\Theta$,
which is different from our formula.
But there is a mistake in the last line of the proof of Lemma~3.1 in \cite{Kumar}. Formulas (3.19) and (3.20) there are 
both correct but the conclusion  is wrong, 
in fact the correct conclusion is exactly that $\bbLa^*(-K)\equiv 4\Theta$.
This mistake does not affect any of the main results in these papers.}.
\end{proof}

\begin{Corollary}\label{adfgadrha}
The image of the  divisor $E\subset\Pic^{g+1}C$ of Corollary~\ref{zxbzfn}
under $\bar\Lambda$ is a divisor 
that parametrizes parabolic bundles of splitting type $\cO(-1)\oplus\cO(1)$.
This divisor is contracted by $\Xi$ to the point $o\in M_{0,n}$.
\end{Corollary}

\begin{proof}
The first statement is clear from the proof of Lemma~\ref{SSRHS}. For the second,
notice that $L\in F$
if and only if $L=h\otimes\cO(x_1+\ldots+x_{g-1})$ for some points on~$C$. Generically
along $E$, $\{x_1,\ldots,x_{g-1}\}$ is the base locus of $L$ and therefore $\phi_L=\phi_h$.
Thus $\Lambda(E)=(\phi_h(p_1),\ldots,\phi_h(p_n))=o$.
\end{proof}

Next we study the dependence of scattering amplitude on marked points.

\begin{Notation}
Consider the Weil group $W(D_n)=S_n\ltimes(\bZ_2)^{n-1}$, where we identify  $(\bZ_2)^{n-1}$ with the set 
of subsets of $\{1,\ldots,n\}$ of even cardinality. The Weil group acts on the stack $\Bun(\bP^1,z_1,\ldots,z_n)$ as follows:
the symmetric group acts by permuting marked points and the group $(\bZ_2)^{n-1}$ acts by {\em elementary transformations}:
for every subset $I\subset\{1,\ldots,n\}$ of even cardinality, consider an exact sequence
\begin{equation}\label{SdvSDbd}
0\to F'\arrow^\alpha F\to \bigoplus_{i\in I} (F|_{z_i})/V_i\to 0.
\cooltag\end{equation}
$F'$ is a rank $2$ vector bundle of degree $-|I|$ with parabolic lines defined as follows: 
$$
V'_i=\begin{cases}
\alpha^{-1}V_i & \hbox{\rm if}\quad i\not\in I\cr
\Ker\alpha|_{z_i} & \hbox{\rm if}\quad i\in I.
\end{cases}$$
To force degree to be $0$, the elementary transformation is defined as
$$\el_I(F)=F'\left({|I|\over 2}\right).$$
In the language of ruled surfaces, elementary transformations are given by blowing up an even number of parabolic points in fibers
and blowing down proper transforms of fibers. 
$W(D_n)$ acts on the Fano model $\Bun_F(\bP^1;z_1,\ldots,z_n)$ by elementary transformations, in fact it is its full automorphism group \cite{Araujo}.
\end{Notation}

\begin{Proposition}\label{sadvfb}
Let $C^n_0\subset C^n$ be the configuration space of points 
$p_1,\ldots,p_n\in C$ such that $p_i\ne p_j$ and $p_i\ne \tau(p_j)$ for $i\ne j$. Consider a commutative diagram
$$\begin{CD}
\Pic^{g+1}C\times C^n_0   @>\bbLa>>  \cBun_F(\bP^1)  @>\Xi>> M_{0,n} \\
@VVV                                                     @VV{\zeta}V \\
C^n_0                                 @>o>>         M_{0,n}
\end{CD}$$
where top arrows are rational maps and where 
$$\cBun_F(\bP^1)\arrow^\zeta M_{0,n}$$ 
is the universal moduli space of parabolic vector bundles:
$$\zeta^{-1}(z_1,\ldots,z_n)=\Bun_F(\bP^1,z_1,\ldots,z_n).$$
The map $o$ is defined in \eqref{mjhvjhv}.
The action of $W(D_n)$ has the following compatibility:
\begin{enumerate}
\item
The map
$\Xi:\,\Bun_F(\bP^1;z_1,\ldots,z_n)\dashrightarrow M_{0,n}$ 
is $S_n$ equvariant. Elementary transformations $\el_I$ give birational involutions of $M_{0,n}$ that depend on $z_1,\ldots,z_n$.
\item
The square of the diagram is $W(D_n)$ equivariant. The action is defined as follows: 
\begin{enumerate}
\item $S_n$ acts everywhere by permuting marked points.
\item $(\bZ_2)^{n-1}$ acts trivially\footnote{Warning: this $M_{0,n}$ is different from $M_{0,n}$ in the top right corner!} on $M_{0,n}$. 
\item Every $I\in (\bZ_2)^{n-1}$ acts on $(D;p_1,\ldots,p_n)\in \Pic^{g+1}C\times C^n_0$ is as follows:
\begin{equation}\label{sfvsfevwefvw}
p_i\mapsto\begin{cases} 
\tau(p_i)& \hbox{\rm if}\  i\in I\cr
p_i & \hbox{\rm if}\ i\not\in I\cr
\end{cases}\qquad\hbox{\rm and}\quad D\mapsto D-\sum_{i\in I}p_i+{|I|\over 2}h.
\cooltag\end{equation}
\end{enumerate} 
\end{enumerate}
\end{Proposition}

\begin{proof}
For (1), everything follows from definitions except for the remark about birational involutions on $M_{0,n}$.
Note that $\bP(F)=\bP^1_z\times\bP_q^1$ and $\el(I)$ amounts to an elementary transformation of this $\bP^1$ bundle (over $\bP^1_z$)
in points $(z_i,q_i)$ for $i\in I$.
It suffices to consider the  case when $I=\{1,n\}$.  We change coordinates so that $z_1=q_1=0$, $z_n=q_n=\infty$.
In these coordinates, proper transforms of lines $q=\lambda z$ after the elementary transformation become
horizontal rulings. 
\begin{figure}[htbp]
\includegraphics[width=\textwidth]{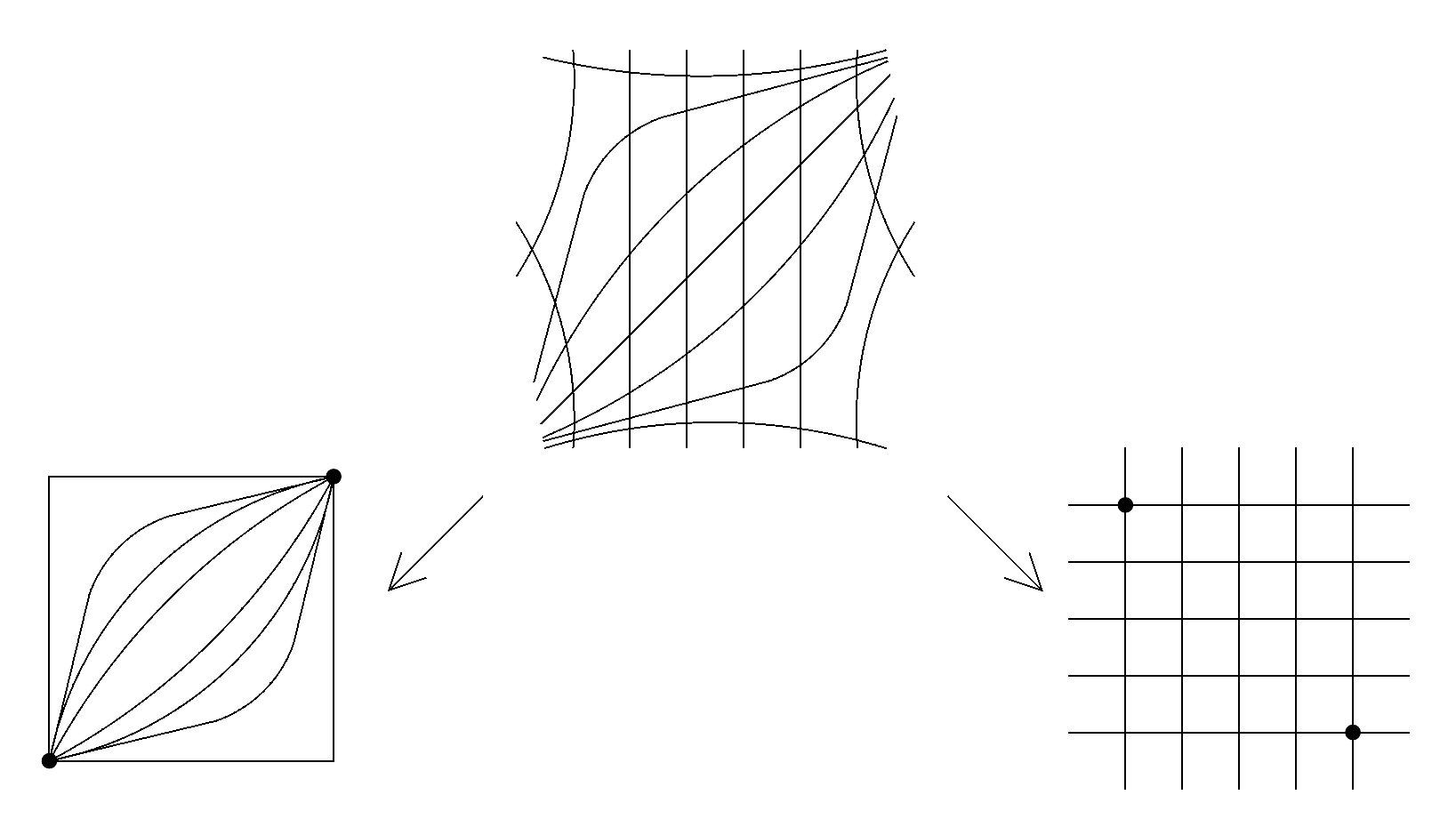}
\end{figure}
Thus 
$$(0,q_2,\ldots,q_{n-1},\infty)\mapsto \left(\infty ,{q_2\over z_2},\ldots,{q_{n-1}\over z_{n-1}},0\right)\sim
\left(0 ,{z_2\over q_2},\ldots,{z_{n-1}\over q_{n-1}},\infty\right).$$
This is a Cremona transformation with center that depends on $z_1,\ldots,z_n$.

For (2), notice that the sequence \eqref{SdvSDbd} is the push-forward of the sequence
$$0\to \cO\left(D-\sum_{i\in I}p_i\right)\to \cO(D)\to \bigoplus_{i\in I} \cO_{p_i}\to 0,$$
which gives the formula for $F'$ after tensoring with a multiple of $\cO(h)$.
\end{proof}

\section{Scattering using the matrix model}\label{kwJHEfg}

We continue to use notation of Section~\ref{sdvqefv}.
In particular, $C$ is a smooth  hyperelliptic  curve of genus $g\ge2$ that satisfies Assumption~\ref{Sfbzfdhadfh},
i.e.~ points $p_1,\ldots,p_n\in C$ are different and  no pair is in hyperelliptic involution.
Let $z_1,\ldots,z_n\in\bP^1$ be their projections
and $\bbLa:\,\Pic^{g+1}C\to \Bun(\bP^1;z_1,\ldots,z_n)$ is the standard morphism
to the stack of quasi-parabolic vector bundles.
We first review the Jacobi's description of the Jacobian of a hyperelliptic curve as the space 
of conjugacy classes of $2\times 2$ matrices following \cite{Mumford} and \cite{Beauville}.

\begin{Review}
The curve $C$ has equation $y^2=f(z)$, where $f\in\cO_{\bP^1}(2g+2)$ is a polynomial without multiple roots.
The map $\pi:=\phi_h:\,C\to\bP^1$ is the projection $(z,y)\mapsto z$.
The points $(z_0, 0)$ with $z_0$ a root of $f(z)$ are the Weierstrass points.
As the sheaf of $\cO_{\bP^1}$-modules, we have 
$$\pi_*\cO_C=\cO_{\bP^1}(-g-1)\oplus\cO_{\bP^1}.$$
The sheaf of algebras structure of $\pi_*\cO_C$ is completely determined  by the
map $\Sym^2\cO_{\bP^1}(-g-1)\to\cO_{\bP^1}$, which is simply
multiplication by $f\in\cO_{\bP^1}(2g+2)$.
By Lemma~\ref{SSRHS},
when $L \in \Pic^{g+1}(C)$ is away from the theta-divisor~$E$, we have
$$\pi_* L=\cO_{\bP^1}\oplus\cO_{\bP^1}.$$
The action of $\pi_*\cO_C$ on $\pi_* L$ is determined by a $2\times 2$ matrix 
\begin{equation}\label{sGSGsrg}
M=\left[\begin{matrix} V&U\cr W& V'\cr\end{matrix}\right]\in \Hom(\cO_{\bP^1}(-g-1)\otimes 
(\cO_{\bP^1}\oplus\cO_{\bP^1}), \cO_{\bP^1}\oplus\cO_{\bP^1}).
\cooltag\end{equation}
of polynomials $U,V,V',W$ in $z$ of degree at most $g+1$. Applying $M$ twice should be a multiplication by the polynomial $f(z)$, which gives
equations on polynomials
$$V'=-V$$
and
\begin{equation}\label{dfbzdfhz}
-\det(M)=V^2+UW=f(z).
\cooltag\end{equation}
Let $S(f)\subset \bA^{3(g+2)}$ be the affine subvariety given by equations \eqref{dfbzdfhz},
where $\bA^{3(g+2)}$ parametrizes $U,V,W$.
Then $S(f)$ is smooth  for any  
polynomial $f(z)$ without multiple roots. Moreover, the group
$\PGL_2$ acts on $S(f)$ (with elements viewed as $2\times2$ matrices) by conjugation freely and the quotient 
is, by \cite[Theoreme 1.4]{Beauville},
\begin{equation}\label{SFhDFh}
S(f)/\PGL_2\simeq \Pic^{g+1}C\setminus E.
\cooltag\end{equation}
\end{Review}

\begin{Theorem}\label{asfgbsfgasfh}
The map $\bbLa$ (and thus  the scattering amplitude map $\bLa$) has degree~$2^g$.
\end{Theorem}

\begin{proof}
To~determine the degree of $\bbLa$, we  fix a general point in $\Bun_F(\bP^1,z_1,\ldots,z_n)$
and count the number of preimages of this point in $\Pic^{g+1}C$. 
A general parabolic bundle has splitting type $\cO\oplus\cO$ and the 
parabolic structure is the set of points $(z_1,q_1),\ldots,(z_n,q_n)$, so essentially we
choose general points $q_1,\ldots,q_n\in\bP^1$.
The corresponding bundle $L\in\Pic^{g+1}C$ is away from the theta-divisor $E$, so we can locate it using~\eqref{SFhDFh}.

\begin{Review}\label{sfgqwrgqwrg}
Suppose  that neither of the points $p_1,\ldots,p_n$ is a Weierstrass point.
Then $(\pi_*L)|_{z_i}=L|_{p_i}\oplus L|_{\tau(p_i)}$ and thus the $2\times 2$ matrix $M(z_i)$ has two distinct eigenspaces, one corresponds to the parabolic structure given by $p_i$
and another by $\tau(p_i)$. The point $q_i\in\bP^1$ gives the slope of this eigenspace. 
Thus we need to do the following: 
\begin{enumerate}
\item Count the number of solutions
in $S(f)$ such that each matrix $M(z_i)$ has an eigenspace with a fixed general 
slope $q_i$.
\item Divide this number by $2^n$. Indeed, 
using $p_i$ or $\tau(p_i)$
gives the scattering amplitude of the same degree by Proposition~\ref{sadvfb}.
\end{enumerate}
Note that taking conjugacy classes of matrices by the $\PGL_2$ action is not necessary: 
fixing three  different slopes of eigenspaces eliminates the conjugacy action.
\end{Review}

\begin{Review}\label{srfqwrg}
The smooth solution set $S(f)\subset\bC^{3g+6}$ of dimension $g+3$ 
is the intersection of $2g+3$ affine quadrics, one for each coefficient of the degree $2g+2$ polynomial~$f(z)$.
Solutions at infinity $\bP^{3g+5}$ are given by the homogeneous equation $V^2=UW$, which has expected dimension $g+2$.
Indeed, if $V=0$ then either $U=0$ or $W=0$, which gives a union of two projective subspaces of dimension $g+1$ each.
If $V\ne0$ then the solution is determined by $V$ up to reordering of terms in the polynomial factorization and rescaling $U$ by $\lambda$
and $W$ by $\lambda^{-1}$.
Imposing the slope $q_i$ at $z_i$ is a linear equation on $U,V,W$. Since $\det M(z_i)\ne0$ but $\Tr M(z_i)=0$,
no matrix $M(z_i)$ has more than two eigenspaces, and so these linear equations have no base locus on $S(f)$.
By Bertini Theorem, the intersection is transversal and has expected dimension.
We claim that there are no solutions at infinity. Indeed, the solution set of the homogeneous system of equations
$V^2=UW$ can be thought of as $2\times 2$ matrices and is covered by $n$ open charts where $M(z_i)\ne0$.
Since $\det M(z_i)=\Tr M(z_i)=0$, $M(z_i)$ is a non-zero nilpotent matrix and thus have only one eigenspace.
Thus the base locus of the solution set is empty and so it has an expected dimension by Bertini Theorem, 
which means it is empty.
\end{Review}

\begin{Review}
The space of matrices of polynomials with fixed slopes $q_i$  is a linear space of dimension $3(g+2)-n=2g+3$ and we are counting
intersection points of $2g+3$ quadrics under the assumption that intersection is transersal and does not run away to infinity.
Thus we have $2^{2g+3}$ intersection points
and therefore 
$$\deg\bbLa={2^{2g+3}\over 2^n}=2^g.$$
\end{Review}

\begin{Review}
If some of the points $z_i$ are Weierstrass points, i.e.~$p_i=\tau(p_i)$, the argument goes as before, except for two issues:
\begin{enumerate}
\item We don't have to divide the number of solutions by $2$ as in \ref{sfgqwrgqwrg} (2).
\item By Claim~\ref{wef	wef}, $M(z_i)$ is a non-zero nilpotent matrix. 
The subvariety of nilpotent $2\times 2$ matrices is a quadric cone
and fixing the slope $q_i$ gives a ruling of this cone (intersection with the tangent plane of multiplicity~$2$).
Thus fixing the slope is not equivalent to pulling back a general
hyperplane by a  morphism to projective space as in \ref{srfqwrg}. 
The correct application of Bertini Theorem is to project this cone onto a conic (isomorphic to $\bP^1$)
and count this solution only once.
\end{enumerate}
Thus issues (1) and (2) cancel each other and we get the same count.
\end{Review}

\begin{Claim}\label{wef	wef}
$M(z_i)$ is a non-zero matrix.
\end{Claim}

Indeed, if $M(z_i)=0$ then $U$, $V$ and $W$ have a root at $z_i$.
But then $f$ has a double root at $z_i$, contradiction.
\end{proof}

\begin{Review}[\cite{Mumford}]\label{asrgarh}
In the model \eqref{SFhDFh}, one can eliminate 
the $\PGL_2$-action by making
\begin{itemize}
\item $f$ a monic polynomial of degree $2g+1$ (one of the roots is moved to $\infty$),
\item $U$ a monic polynomial of degree $g$,
\item $V$ a polynomial of degree $g-1$,
\item $W$ a monic polynomial of degree $g+1$.
\end{itemize}
Under these conditions, the solution set $M$ of equations \eqref{dfbzdfhz} in $\bA^{3g+1}$
is isomorphic to $\Pic^{g+1}C\setminus E$. Solutions look as follows.
Suppose $U(z)=(z-t_1)\ldots(z-t_g)$ has no multiple roots. Choose $s_i=\pm\sqrt{f(t_i)}$ for $i=1,\ldots,g$
and use Lagrange interpolation to find $V(z)$ such that $V(t_i)=s_i$  for $i=1,\ldots,g$.
Then $U(z)$ divides $f(z)-V(z)^2$ and we can define $W(z)={\displaystyle f(z)-V(z)^2\over \displaystyle U(z)}$.

For any $c\in\bP^1$, the Lax pair differential equation
gives a translation-invariant vector field $\dot F=(\dot U(z),\dot V(z),\dot W(z))$ on $M$ with components
$$\dot U(z)={V(c)U(z)-U(c)V(z)\over z-c};$$
$$\dot V(z)={1\over 2}{U(c)W(z)-W(c)U(z)\over z-c}-U(c)U(z);$$
$$\dot W(z)={W(c)V(z)-V(c)W(z)\over z-c}+U(c)V(z).$$
General points $c_1,\ldots,c_g$ gives linearly independent translation-invariant vector fields
$\dot F_1,\ldots,\dot F_g$ and thus a translation-invariant polyvector field
$A^\vee=\dot F_1\wedge\ldots\wedge\dot F_g$, which is  dual to the translation-invariant volume form $A$ on $\Pic^{g+1}C$.
Applying $d\bLa$ to $A^\vee$ and dualizing gives the value of the branch of the scattering amplitude form 
that corresponds to the point $(U,V,W)$. More concretely, we can factor the scattering amplitude map
$\bLa:\,M\to M_{0,n}$ into maps
$$E:\, M\dashrightarrow (\bP^1)^n\quad\hbox{\rm and}\quad Q:\,(\bP^1)^n\dashrightarrow M_{0,n},$$
where $Q$ is the quotient by the $\PGL_2$ action and $E$ is the map that assigns to a matrix of polynomials \eqref{sGSGsrg}
the slopes of its eigenspaces 
$$
{y_1-V(z_1)\over U(z_1)},\ldots, {y_n-V(z_n)\over U(z_n)}
$$
at marked points $p_i=(z_i,y_i)$ for $i=1,\ldots,n$.
\end{Review}


\section{Bypassing the Kummer surface}\label{sfgasgsRH}

We will
make results of the previous section more explicit for genus~$2$ curves. 

\begin{Notation} Fix a smooth pointed genus~$2$ curve $(C;p_1,\ldots,p_5)$.
\begin{enumerate}
\item Let $P:=p_1+\ldots+p_5$.
\item We view $C$ as a degree $5$ curve in $\bP^3$ using the embedding
$\phi_P:\,C\hookrightarrow\bP^3$.
\item $K$ is a canonical divisor, $\phi_K:\,C\arrow^{2:1}\bP^1$ is a hyperelliptic double cover.
\item We introduce $16$ points in $\Pic^3C$, 
$$\delta=P-K,\quad \delta_i=K+p_i, \quad \delta_{ij}=P-p_i-p_j,$$ 
where the indices are $1 \le  i \le 5$ and $1 \le i < j \le 5$, respectively.
\item Special loci in $\Pic^3C$ defined in Corollary~\ref{zxbzfn} are as follows: 
$$W=\emptyset, \quad R=\{D\ |\ 2D\in P+C\}\equiv 4\Theta,$$
$$E=K+C,\quad E_{ij}=C+p_i+p_j.$$
\item We introduce another useful theta divisor on $\Pic^3C$ for $i=1,\ldots,5$:
$$E_i=P-p_i-C.$$
\end{enumerate}
Here and elsewhere we don't  distinguish between line bundles and linear equivalence classes of divisors,
for example $E=K+C$ denotes the locus in $\Pic^3C$ of line bundles of the form $\cO(K+p)$ for $p\in C$.
Hopefully this won't cause confusion.
\end{Notation}

\begin{Lemma}
In genus~$2$, Assumption~\ref{Sfbzfdhadfh} is equivalent to any of the following:
\begin{enumerate}
\item No two points of $p_1,\ldots,p_5$ are related by the hyperelliptic involution.
\item No three points of $p_1,\ldots,p_5\in\bP^3$ are collinear. 
\item $16$ divisors $E$, $E_i$, $E_{ij}$ are pairwise different.
\item $16$ points $\delta$, $\delta_i$, $\delta_{ij}$ are pairwise different.
\item $\delta\not\in E_{ij}$ for any $i\ne j$.
\end{enumerate}
\end{Lemma}

\begin{proof}
Left as a fun exercise for the reader. The hint is  $C=K-C\subset\Pic^1 C$.
\end{proof}

A special feature of the genus $2$ case is that divisors of degree $2$ are effective.
This can be used to study divisors $D\in\Pic^3C$ and their linear systems by
associating to $D$ a residual divisor $P-D$ of degree $2$. This gives the following proposition.

\begin{Proposition}\label{sfgFhRh}
Let $\bP^2\subset\bP^3$ be the plane passing  through $p_1,\ldots,p_5$.
Let 
$$\dP=\Bl_{p_1,\ldots,p_5}\bP^2,$$ 
be the del Pezzo surface of degree $4$.
Recall that  $\oM_{0,5}$ is the del Pezzo surface of degree $5$.
The~scattering amplitude map can be extended to 
a commutative diagram 
$$\begin{CD}
\Sym^2C      @>\bbLa>>  \dP\\
@V{a}VV         @VV{\Xi}V \\
\Pic^3C   @>\bLa>>  \oM_{0,5}
\end{CD}$$
where horizontal arrows are rational maps.
The maps can be described as follows:
\begin{enumerate}
\item For any $(x,y)\in\Sym^2C$, consider the secant line $\ell_{xy}\subset\bP^3$ 
connecting $x$ and $y$ (or the tangent line to $C$ at $x$  if $x=y$). Then $\bbLa(x,y)=\ell_{xy}\cap \bP^2$.
\item  The map $\Xi$ is given by projecting $p_1,\ldots,p_5$ from a varying point of $\bP^2$ (cf.~\cite{CT_Cont}).
It  blows down
 the conic passing through $p_1,\ldots,p_5$ to the point  
$$o:=\left(\phi_K(p_1),\ldots,\phi_K(p_5)\right)\in M_{0,5}.$$
\item $a(x,y)=\cO(P-x-y)$. The map $a$ gives $\Sym^2C$ as the blow-up of $\Pic^3C$ at $\delta$.
\end{enumerate}
\end{Proposition}

\begin{proof}
Indeed, for $x,y\in C$, $|P-x-y|$ is the pencil of planes in $\bP^3$ through $\ell_{xy}$.
Projecting $p_1,\ldots,p_5$ from $\ell$ is equivalent to projecting them from $\ell\cap\bP^2$.
\end{proof}

\begin{Corollary}
The scattering amplitude map $\bLa:\,\Pic^3C\dashrightarrow M_{0,5}$ has degree $4$.
\end{Corollary}

If course this follows from Theorem~\ref{asfgbsfgasfh}, but here is an independent proof.

\begin{proof}
It suffices to check this for $\bbLa$.
We have to count how many secant lines $\ell_{xy}$ pass through a general point of $\bP^2$.
By dimension count, projection of $C$ from a general point of $\bP^2$ is a nodal curve of degree $5$
and therefore arithmetic genus $6$. Thus there will be $6-2=4$ nodes, each produced by a secant line. 
\end{proof}

\begin{Lemma}\label{SFgSFhF} 
We have divisors $\cD,\cD_i,\cD_{ij}\subset C^5\setminus\bigcup\limits_{i<j}\Delta_{ij}$ that
parametrize configurations of  marked points $p_1,\ldots,p_5\in C$ satisfying any of the following equivalent conditions:
\begin{enumerate}
\item $\delta\in E$ or $\delta_i\in E_i$ or $\delta_{ij}\in E_{ij}$, respectively.
\item $\delta\in R$ or $\delta_i\in R$ or $\delta_{ij}\in R$, respectively.
\item $\delta-K$ or $P-K-2p_i$ or $P-2p_i-2p_j$, respectively, is an effective divisor.
\end{enumerate}
Divisors $\cD,\cD_i,\cD_{ij}$ can also be characterized using geometric conditions:
\begin{enumerate}
\item[($\cD$)]  the unique quadric surface in $\bP^3$ containing $C$ is singular.
\item[($\cD_i$)] the tangent line at $p_i$ to $C\subset \bP^3$ intersects $C$ at another point.
\item[($\cD_{ij})$] tangent lines at $p_i$ and $p_j$ intersect.
\end{enumerate}
\end{Lemma}

\begin{proof}
Left as a  fun exercise for the reader.
\end{proof}

\begin{Theorem}\label{vasdgsgG}
Let $X$ be the blow-up of $\Pic^3C$ in $16$ points $\delta$, $\delta_i$ and $\delta_{ij}$
with exceptional divisors $\Delta$, $\Delta_i$, $\Delta_{ij}$.
Let $\bE$, $\bE_i$, $\bE_{ij}$ be the proper transforms of divisors $E$, $E_i$, $E_{ij}$.

The scattering amplitude rational map $\bLa$ induces a finite $4:1$ morphism $X\arrow^{\bbLa}\dP$.
The~preimages of sixteen $(-1)$-curves of $\dP$ are the following pairs:
$$\bbLa^{-1}(\hbox{\rm conic through $p_1,\ldots,p_5$})=\Delta\cup\bE;$$
$$\bbLa^{-1}(\hbox{\rm exceptional divisor over $p_i$})=\Delta_{i}\cup\bE_{i};$$
$$\bbLa^{-1}(\hbox{\rm line through $p_i$, $p_j$})=\Delta_{ij}\cup\bE_{ij}.$$
The restriction of $\bbLa$ to each  ``$\Delta\cup E$'' pair depends on on whether
$(p_1,\ldots,p_5)$ is contained in the corresponding configuration divisor ``$\cD$''  in $C^5$:
\begin{enumerate}
\item[(yes)]       $\bbLa$ is ramified (of order $2$) along $``\Delta''$ and has degree $2$ along ``$\bE$''.
\item[(no)] $\bbLa$ is an isomorphism along ``$\Delta$'' and has degree $3$ along ``$\bE$''.
\end{enumerate}
\end{Theorem}

\begin{proof}
We draw special curves on $X$ and how they intersect, see Figure~\ref{SFgFhFD}.
\begin{figure}[htbp]
\includegraphics[width=\textwidth]{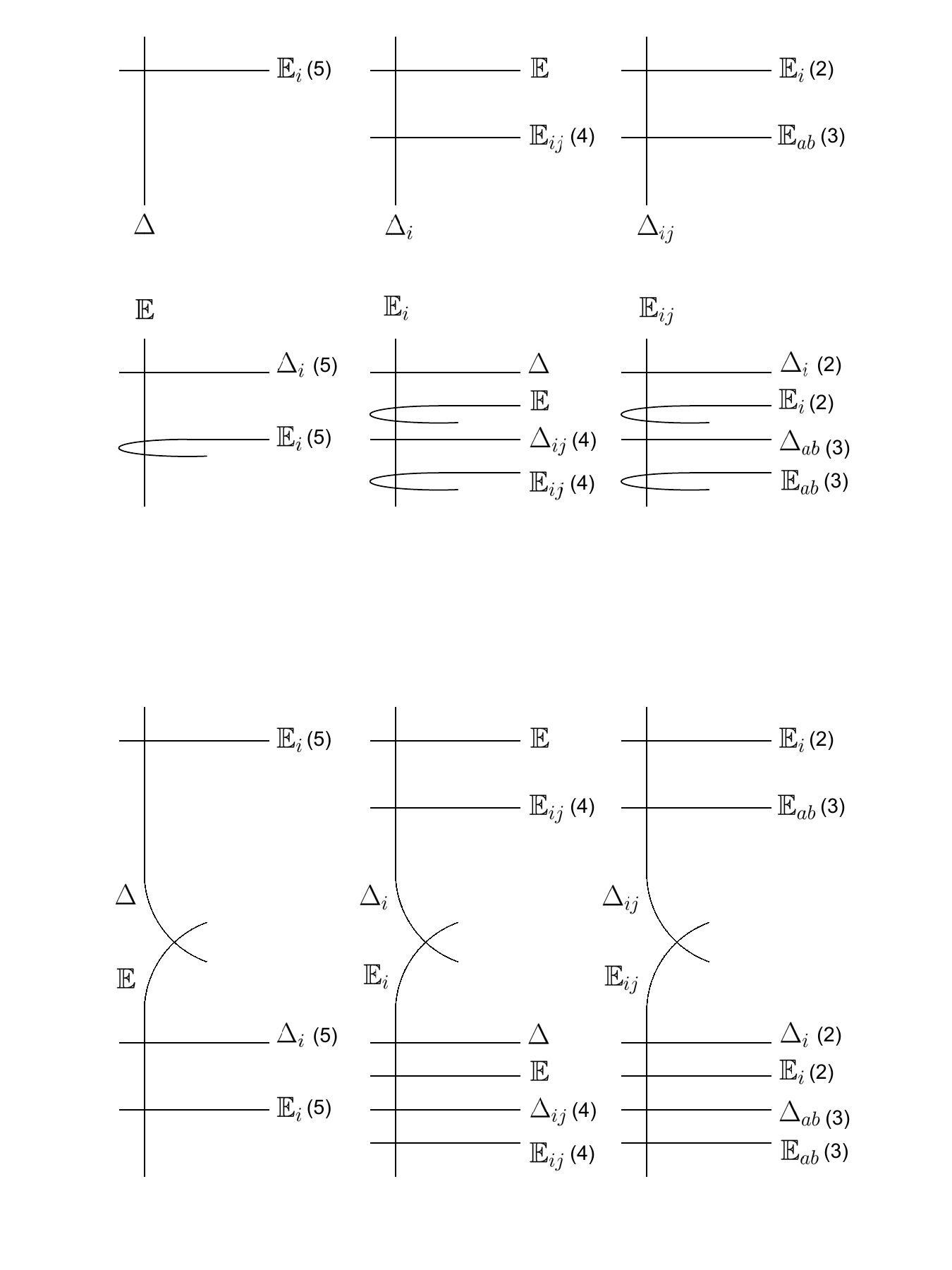}
\caption{``Double sixteen'' configuration on $X$ away from (top) and along (bottom) the corresponding $\cD$ divisors}\label{SFgFhFD}\label{kjashfg}
\end{figure}
We use that $\Theta^2=2$ and check how special curves pass through the special points.
In Figure~\ref{SFgFhFD} we draw the special curves when the configuration $(p_1,\ldots,p_5)$ is away from (top) and along (bottom) the corresponding 
$\cD$ divisors. 
 In the latter case, the inverse image of an ``$E$'' divisor in $\Pic^3C$ is the union ``$\Delta\cup\bE$''
and the ramification divisor on $X$  is the union of the proper transform of $R$ and the ``$\Delta$'' divisor.
Note that the configuration $(p_1,\ldots,p_5)$ can be on some $\cD$ divisors and not on the others,
so the actual picture could be a mixture of the top and the bottom of Figure~\ref{SFgFhFD}.

Next we study the map $\Sym^2C\to\bP^2$, $(x,y)\mapsto \ell_{xy}\cap\bP^2$ of Proposition~\ref{sfgFhRh}.
It is well-defined unless $\ell_{xy}\subset\bP^2$, which happens  at points $\delta_{ij}$.
So the map $X\to\bP^2$ induced by $\bbLa$ is regular away from $\Delta_{ij}$,
in particular it is regular along $\Delta\cup\bE$. 

What is the preimage of $p_i$? In $\Sym^2C$, there are two possibilities, one is the curve $(p_i,C)$, i.e.~$E_i$.
Another is a point $(x,y)$ such that $P-(p_i+x+y)$ is a degree $2$ pencil, which should be~$K$.
So another option is the point $P-p_i-K=\delta_i$. Thus on $X$ away from $\Delta_{ij}$'s the preimage of $p_i$ is
$\bE_i\cup\Delta_i$.

What is the preimage of the conic $Q$ through $5$ points? 
This conic is the intersection of $\bP^2$ with the unique quadric $\bQ$ that contains~$C$.
Thus a secant $\ell_{xy}$ must be contained in $\bQ$. If $\bQ$ is smooth, $C$ has bi-degree $(2,3)$.
One ruling is cut by points in the hyperelliptic involution, which is the curve $\Delta$.
Another ruling is cut by triples of points contained in the line, i.e.~such that $P-x-y-K$ is effective, which is the curve $\bE$.
If $\bQ$ is singular then these two rulings are the same, given by the vertex of the quadratic cone $\sim P-2K$
and pairs of points in the hyperelliptic involution. This shows that the rational map 
$X\dashrightarrow \dP$ is regular along $\Delta\cup\bE$ and maps this divisor to the proper transform
of the conic passing through five points. 

Analysis of other ``$\Delta\cup\bE$'' pairs can be done  similarly.
However, this is not necessary. By Proposition~\ref{sadvfb}, the map $X\dashrightarrow \dP$ 
can be extended to a $W(D_5)$ equivariant map
$$\coprod_{I\subset\{1,\ldots,5\}\atop |I|\equiv 0\!\!\!\!\mod 2}X_I\dashrightarrow \dP,$$
where $X_I$ 
corresponds to a $5$-tuple \eqref{sfvsfevwefvw}.
Since the action permutes ``$\Delta\cup\bE''$ pairs, if the map is regular  along one pair (for all $X_I$), it is regular along all pairs.
\end{proof}

When does  $(p_1,\ldots,p_5)$ belong to all configuration divisors $\cD,\cD_i,\cD_{ij}$?

\begin{Theorem}\label{sDGSG}
We have $(p_1,\ldots,p_5)\in\cD\cap\bigcap_i\cD_i\cap\bigcap_{ij}\cD_{ij}$ if and only if 
$p_1,\ldots,p_5$ are Weierstrass points of~$C$, i.e.~ramification points of $\phi_K$.
Suppose that this is the case.
\begin{enumerate}
\item Let $p_6$ be the remaining Weierstrass point.
The $16$ special points are
$$\delta=K+p_6, \quad \delta_i=K+p_i,\quad \delta_{ij}=p_i+p_j+p_6.$$
Equivalently, this is the set of $2$-torsion points in $\Pic^0C$ shifted by $\delta$.
\item
Let $\tau$ be the involution of $\Pic^3C$ induced by the hyperelliptic involution of $C$.
Let~$\Ku=\Pic^3/\tau$ be the Kummer surface with $16$ double points and let $\Kt$ be its minimal resolution.
Then $\tau$ lifts to $X=\Bl_{\delta,\delta_i,\delta_{ij}}\Pic^3$ and $\Kt=X/\tau$. 
\item
The scattering amplitude $4:1$ cover $\bbLa:\,X\to\dP$ factors as
$$X\arrow^{2:1}\Kt\arrow^{2:1}\dP.$$
\item The images of ``$\Delta$'' and ``$\bE$'' type divisors give two configuration of $16$ rational curves in $\Kt$.
Their images in $\Ku$ are the $16$~nodes and $16$ conics called ``tropes''.
The image of $R$ is a genus $5$ ``Humbert'' curve. 
\end{enumerate}
\end{Theorem}
 
\begin{proof}
Suppose that $p_1,\ldots,p_5$ are Weierstrass points.
Assumption~\ref{Sfbzfdhadfh} is satisfied. 
We have $P\sim 3K-p_6$ by the Riemann--Hurwitz formula. Thus 
$$P-2K\sim P-K-2p_i\sim P-2p_i-2p_j\sim p_6$$
is effective and the configuration belongs to all divisors.

In the opposite direction, let $(p_1,\ldots,p_5)\in\cD\cap\bigcap_i\cD_i$.
Define 
$$P-2K\sim z\in C,\quad P-K-2p_i\sim z_i\in C.$$
Then $z-z_i\sim 2p_i-K$, i.e.~$z+\tau(p_i)\sim p_i+z_i$ for every $i$, where $\tau$
denotes the hyperelliptic involution. There are two cases, either $p_i$ is a Weierstrass point 
or $z=p_i$ and $z_i=\tau(p_i)$. Thus either all five points are Weierstrass points, which is what we are trying to prove,
or $4$ of them are, let's say $p_2,p_3,p_4,p_5$, and $P-2K\sim p_1$. 
But then $p_2+p_3+p_4+p_5\sim 2K$, i.e.~$p_2+p_3\sim p_4+p_5$, which is impossible.

The rest is also easy: (1) is immediate to verify, 
(2) and (4) are classical and well-known~\cite{Dol,Sk}.
To prove (3) it suffices to notice that $\bbLa$ is $\tau$-invariant if all marked points are Weierstrass points.
\end{proof}

\section{Scattering measures of M-curves}\label{sdcqwv}

\begin{Review}
We refer to \cite{GH} for the basic theory of  real algebraic curves and  Jacobians.
Let $C$ be a smooth projective complex algebraic curve of genus $g$ . 
We view the set of complex points $C(\bC)$ as a Riemann surface. 
The curve $C$ endowed with a real structure (equivalently, an anti-holomorphic involution) 
is called an {\em  M-curve}\footnote{``M'' stands for ``maximal''.} if the set of real points $C(\bR)\subset C(\bC)$ has $g+1$ (the maximal possible number)
connected components $C_1,\ldots, C_{g+1}$. 
These ovals 
separate $C(\bC)$ into two connected subsets interchanged by the complex conjugation $p\mapsto\bar p$.
Recall from Theorem~\ref{efvwb} that all smooth curves with $n=g+3$ marked points are MHV curves.
In this section we will study the scattering amplitude  of M-curves.
\end{Review}

\begin{Definition}\label{SfbXfbSDb}
A smooth {\em MHV M-curve} is a pointed M-curve such that 
its divisor of marked points $p_1+\ldots+p_n$, $n=g+3$, is preserved by the complex 
conjugation. Moreover, we assume that we have one of the following two cases:
\begin{enumerate}
\item[(A)] All marked points are real and all components of $C(\bR)$ contain one marked point each except for one component, which contains three marked points.
\item[(B)] All but two marked points are real, one on each component of $C(\bR)$. 
The two remaining marked points are complex-conjugate.
\end{enumerate}
\end{Definition}

\begin{Review}\label{awrgawrg}
It is well-known \cite{GH} that the set of real points  
$$\Pic^{d}(\bR)\subset \Pic^{d}(\bC)$$ 
of every component of the Picard group of an M-curve~$C$
is a union of 
$2^g$ connected components denoted by $\Pic_I^{d}(\bR)$ and indexed by subsets $I\subset \{1,\ldots,g+1\}$ such that $|I|\equiv d\mod 2$.
Namely, if a divisor $D=\bar D$ is preserved by the complex conjugation  then the following conditions are equivalent:
\begin{enumerate}
\item $\cO(D)\in \Pic_I^{d}(\bR)$;
\item the divisor $D\cap C_i$ has odd degree if and only if $i\in I$.
\end{enumerate}
Each component $\Pic_I^{d}(\bR)\subset\Pic^{d}(\bR)$ is a torsor of the real Lie group 
$$\Pic^0_\emptyset(\bR)\simeq \U(1)^g\simeq(\bR/\bZ)^g.$$
In what follows we mostly focus on the MHV component $\Pic^dC$, so that $d=g+1$.
\end{Review}

\begin{Definition}
Let $M^{(k,l)}_{0,n}(\bR)$ be the moduli space of $n$-tuples of distinct points in $\bP^1(\bC)$ such that there are $k$ real points 
and $l$ pairs of complex conjugate points. 
We will need  the cases $(k,l) = (n,0)$ and $(n-2,1)$:
\begin{enumerate}
\item[(A)] $M^A_{0,n}(\bR)$,  or simply $M_{0,n}(\bR)$, the moduli space of $n$ real points in $\bP^1(\bC)$;
\item[(B)] 
$M^B_{0,n}(\bR)$, the moduli space of $n-2$ real and two complex-conjugate points.
\end{enumerate}
See~\cite{Ceyhan} for compactifications of  real forms of $M_{0,n}$.
\end{Definition}

\begin{Review}
By Amplification~\ref{zxbzfn},
 the scattering amplitude map $\bLa:\,\Pic^{g+1}C\dashrightarrow M_{0,n}$ is well-defined
away from the divisors $E$, $E_{ij}$ and is unramified away from the divisor $R$.
Since all components $\Pic_I^{g+1}(\bR)$ of an MHV M-curve are Zariski dense in $\Pic^{g+1}(\bC)$, 
$\bLa$ is defined and unramified generically along them.
Moreover, 
$$\bLa(\Pic^{g+1}(\bR))\subset M^?_{0,n}(\bR),$$
where $?=A$ or $B$ depending on the type of the curve $C$. 
Indeed, if $L\in\Pic^{g+1}(\bR)$ then we can choose $\phi_L\in\bR(C)$ a real  
meromorphic function
and the marked points will be mapped to points of the Riemann sphere according of type $A$ or~$B$.
\end{Review}

\begin{Theorem}\label{dfvwevev} 
Let $?=A$ or $B$ depending on the type of the MHV M-curve $C$. Then
$$\bLa^{-1}(M^{?}_{0,n}(\bR))\subset \Pic^{g+1}(\bR).$$
Moreover, restriction of $\bLa$ to any of the $2^g$ connected components 
$\Pic_I^{g+1}(\bR)$ is injective.
\end{Theorem}

\begin{proof}
Let $(q_1,\ldots,q_n)\in M^?_{0,n}(\bR)$.
Let $f\in\bC(C)$ be a rational function of degree $g+1$ such that $f(p_i)=q_i$ for every $i=1,\ldots,n$.
First we claim that $f\in\bR(C)$. We~argue by contradiction and suppose that this is not the case.
Then 
$$g(z)=if(z)-i\overline{f(\bar z)}\in\bR(C)$$
is a non-zero real rational function.
Thus $g(z)$ has finitely many zeros on $C(\bR)$, and so $f(C(\bR))\cap\bP^1(\bR)$ is a finite union of points. By~applying a transformation from $\PGL_2(\bR)$, we can assume that
$\infty\not\in f(C(\bR))$. Thus $g$ also doesn't have any poles on $C(\bR)$. On the other hand, $g(p_i)=0$ for every $i$
and $\div(g)\cap C_i$ is even for every~$i$, which is true for any real rational function, see \cite[Lemma~4.1]{GH}. Thus $g$ has at least one additional zero 
on each $C_i$ and therefore has degree at least $2g+4$.
This is a contradiction: zeros of $g$ are intersection points of the diagonal in $\bP^1\times\bP^1$ with the image of $C$ under the map $(f(z),\overline{f(\bar z)})$
which has homology class $(g+1,g+1)$. Thus we have at most $2g+2$ intersection points.

Next we show injectivity of the restriction of $\bLa$ to any connected component $\Pic^{g+1}_I(\bR)$.
We argue by contradiction and suppose that $\bLa(L)=\bLa(L')$ for two different line bundles in the same component.
We can find rational functions
$f,f'\in\bR(C)$ in the corresponding linear systems such that $f(p_i)=f'(p_i)=q_i$ for every $i=1,\ldots,n$ 
(because $\bLa(L)=\bLa(L')$)
and such that $f^{-1}(\infty)\cap C_i$ and $f'^{-1}(\infty)\cap C_i$ have the same parity for every $i$ (because they are in the same component $\Pic_I$).
We can also apply a projective transformation from $\PGL_2(\bR)$ so that $f$ and $f'$ have disjoint poles.
Then $f-f'$ has an even number of poles on every~$C_i$. Therefore $f-f'$ has an even number of zeros on every $C_i$.
One of these zeros is $p_i$, so there must be at least one additional zero on each $C_i$.
Thus $f-f'$ has at least $2g+4$ zeros total and we finish as in the first part.
\end{proof}

\begin{Corollary}\label{dfvwefv}
The scattering amplitude map of a generic MHV curve  has degree $2^g$.
\end{Corollary}

\begin{proof}
By Theorem~\ref{dfvwevev}, every point in $M^?_{0,n}(\bR)$ has at most $2^g$ preimages by $\bLa$.
Since $M^?_{0,n}(\bR)$ is a real form of $M_{0,n}$ of real dimension {$n-3$}, it is Zariski dense in~it.
Thus the scattering amplitude map $\bLa$ of any MHV M-curve has degree at most~$2^g$
in both types (A) and (B).
Let~$D\le 2^g$~be the maximum
of these degrees.

Since the locus of MHV $M$-curves $(C;p_1,\ldots,p_{n})$ of either of these types is a connected component of the real form of the moduli space $M_{g,n}$
of real dimension $3g-3+n$  (see \cite{SS}), it is Zariski dense in $M_{g,n}$. It follows that the degree of the scattering amplitude is bounded above by $D$
for a Zariski dense subset in $M_{g,n}$. 

We claim that the degree of $\bLa$ is bounded above by $D$ for all smooth  curves.
Indeed, let $\cPic^{g+1}\to M_{g,n}$ be the universal Jacobian and let 
$$\bLa:\,\cPic^{g+1}\dashrightarrow M_{g,n}\times M_{0,n}$$
be the universal scattering amplitude map
(the reader uncomfortable with stacks can use a $1$-parameter family of curves connecting two  curves instead of $M_{g,n}$). 
Since $\bLa$ is generically finite 
and $M_{g,n}\times M_{0,n}$ is normal, the Stein factorization of $\bLa$ implies that
cardinality of finite fibers can't be more than $D$. 

Finally, by Theorem~\ref{asfgbsfgasfh}, 
the scattering amplitude map of a hyperelliptic curve has degree $2^g$. 
Therefore, $D=2^g$, which completes the proof.
\end{proof}

\begin{Definition}\label{jhsbdchjs,bdv}
There exists a distinguished connected component of $\Pic^{g+1}(\bR)$, 
$$\Pic^{g+1}_H(\bR):=\Pic_{\{1,\ldots,g+1\}}^{g+1}(\bR),$$ 
which was studied in~\cite{Huisman}.
We call it the {\em Huisman component}.
Every effective divisor from $\Pic^{g+1}_H(\bR)$ is a union of $g+1$ points, one in each connected component 
$C_1,\ldots,C_{g+1}\subset C(\bR)$.
Here are some  nice properties of $\Pic^{g+1}_H(\bR)$:
\end{Definition}

\begin{Proposition}[\cite{Huisman}]\label{adfhadrhadrha}$\ $
\begin{enumerate}
\item Every $L\in\Pic^{g+1}_H(\bR)$ is non-special and globally generated. 
\item $\phi_L$ is unramified along $C(\bR)=C_1\cup\ldots\cup C_{g+1}$.
\item $\phi_L|_{C_i}:\,C_i\to\bP^1(\bR)$ is a real-analytic isomorphism for any~$i=1,\ldots,g+1$.
\item Fix $z_{g+1}\in C_{g+1}$. The map $C_1\times\ldots\times C_g\to\Pic^{g+1}_H(\bR)$ that sends $z_1,\ldots,z_g$ to
$\cO(z_1+\ldots+z_g+z_{g+1})$ is 
a real-analytic isomorphism.
\end{enumerate}
\end{Proposition}

Here's the second main result of this section.

\begin{Theorem}\label{DSvSDb}
In the notation of Amplification~\ref{zxbzfn},
$\Pic^{g+1}(\bR)$ is disjoint from the ramification divisor $R$
in both types (A) and (B).
In particular, the scattering amplitude map 
$$\bLa_I:\,\Pic_I^{g+1}(\bR)\setminus(E\cup\bigcup_{i<j} E_{ij})\to M^?_{0,n}(\bR)$$
is a real analytic isomorphism onto its image for every component $\Pic_I^{g+1}(\bR)$.
\end{Theorem}

\begin{proof}
Suppose $D\in \Pic^{g+1}(\bR)\cap R$. By definition of $R$, this means that 
$$2D\sim x_1+\ldots+x_{g-1} +p_1 +\ldots +p_{g+3}$$
for some $x_i\in C$. Since $C$ has type (A) or (B), 
we can assume without loss of generality that either $p_{g+2}$ and $p_{g+3}$ 
are complex-conjugate points (type B) or additional points on one of the connected components of $C(\bR)$ (type A). 
Let 
$$G=x_1 +...+x_{g-1} +p_{g+2} +p_{g+3}\sim 2D-p_1-\ldots-p_{g+1}.$$
Since $O(2D)\in \Pic_{\emptyset}^{2g+2}(\bR)$ and $O(p_1 + ... + p_{g+1})\in \Pic^{g+1}_H(\bR)$,
it follows that 
$\cO(G)\in\Pic^{g+1}_H(\bR)$, and therefore it is base-point-free and $h^0(G) = 2$ by~Proposition~\ref{adfhadrhadrha}.
On~the other hand, it has a complex-conjugate section $\bar x_1+\ldots+\bar x_{g-1}+p_{g+2}+p_{g+3}$, which also
contains $p_{g+2}$ and $p_{g+3}$. So this must be the same section, otherwise $p_{g+2}$ and $p_{g+3}$
are in the base locus of $\cO(G)$. It follows that $G$ is invariant under complex conjugation.
But then it can't be in $\Pic^{g+1}_H(\bR)$ since either $p_{g+2}$ and $p_{g+3}$ are complex-conjugate
or belong to the same connected component of $C(\bR)$, in either case the degree of $G\cap C_i$ can't be odd for all $i$.
Since $\bLa_I$ is injective by Theorem~\ref{dfvwevev} 
and its domain is disjoint from $R$,
it is an analytic immersion where it is defined.
\end{proof}

\begin{Amplification}
By Corollary~\ref{dfvwefv}, for smooth {\em complex} curves, the scattering amplitude map $\bLa:\,\Pic^{g+1}C\to M_{0,n}$
 has large degree $2^g$. However, Theorems~\ref{dfvwevev} and \ref{DSvSDb} show that  
 for smooth {\em real} M-curves line bundles in every fiber ``localize''
 into different connected components $\Pic_I^{g+1}(\bR)$ of $\Pic^{g+1}(\bR)$.
In other words, $\bLa$
 is  determined by open immersions (in the domain of $\bLa$)
$$\bLa_I:\,\Pic_I^{g+1}(\bR)\dashrightarrow M^?_{0,n}(\bR).$$
Each component $\Pic_I^{g+1}(\bR)$ carries a uniform probability measure $|A_I|$ given 
by the real-valued volume form $A_I$ translation-invariant under the real $g$-dimensional
torus $\Pic^0_\emptyset(\bR)$. Viewed on $M^?_{0,n}(\bR)$, this gives a differential form $A_I$
and a probability measure $|A_I|$ that we call a {\em scattering amplitude probability measure}.
We will typically compactify $M^?_{0,n}(\bR)$ in some way and 
extend $\Lambda$ to its domain of definition
(so that in particular it will be determined at least generically along $E$ and $E_{ij}$).
Since not every reasonable compactification of $M^?_{0,n}(\bR)$ is 
an orientable manifold, one has to remember (from calculus) that in this case the integral of a
form is not well-defined, however one can always integrate probability measures.
The scattering  measure of the Huisman's component is especially well-behaved.
\end{Amplification}

\begin{Theorem}\label{sGARSGARHA}
Let $C$ be a smooth MHV M-curve with ovals $C(\bR)=C_1\cup\ldots\cup C_{g+1}$. Furthermore, we assume that 
$p_i\in C_i$ for $i=1,\ldots,g+1$ and that
$p_{g+2},p_{g+3}\in C_{g+1}$ (type A) or $p_{g+2}=\bar p_{g+3}$ (type B).
We compactify $M_{0,n}$ by $(\bP^1)^g\simeq(\oM_{0,4})^g$
using the product of forgetful maps $$\pi_{i,g+1,g+2,g+3}:\,M_{0,n}\to M_{0,4}$$
for $i=1,\ldots,g$. The scattering amplitude map induces a real-analytic isomorphism
$$\bR^g/\bZ^g\simeq\Pic^{g+1}_H(\bR)\arrow^\bLa(\bR\bP^1)^g$$
and gives a positive real-analytic scattering amplitude probability measure on $(\bR\bP^1)^g$.
\end{Theorem}

\begin{proof}
By Theorem~\ref{DSvSDb}, $\Pic^{g+1}_H(\bR)$ does not intersect the ramification divisor $R$.
As an easy reformulation of  \ref{jhsbdchjs,bdv}(1), $\Pic^{g+1}_H(\bR)$ does not intersect~$E$ either.
Indeed, let $D=z_1+\ldots+z_{g+1}$ with $z_i\in C_i$.
Arguing by contradiction, 
suppose $D$ is linearly equivalent 
to a divisor in $E$, i.e.
$D\sim K+x_0-x_1-\ldots-x_{g-2}$ for some $x_i\in C$. Then $h^0(D)=2$ by \ref{jhsbdchjs,bdv}
but since $h^0(x_1+\ldots+x_{g-2})\ge 1$,
$$h^0(D-x_0)=h^0(K-x_1-\ldots-x_{g-2})=h^1(x_1+\ldots+x_{g-2})\ge2$$
by Riemann--Roch formula, 
which contradicts global generation of $D$.

For every $L\in \Pic^{g+1}_H(\bR)$, $\phi_L(p_{g+1}),\phi_L(p_{g+3}),\phi_L(p_{g+3})$  are different.
In type (A) this follows from  \ref{jhsbdchjs,bdv}(3). In type (B), this follows because points 
$\phi_L(p_{g+2})$ and $\phi_L(p_{g+3})$ are not real (the preimage of every real point under $\phi_L$ is real).
By Lemma~\ref{adrhadrhsth}, the map $\Pic^{g+1}_H(\bR)\dashrightarrow(\bR\bP^1)^g$
is a regular real-analytic immersion at every point. On the other hand, it is injective by Theorem~\ref{DSvSDb}.
\end{proof}

\begin{Example}[genus $1$]\label{sdcwefvw}
Let $C$ be an elliptic M-curve with two ovals, $X$ and~$Y$. 
\begin{figure}[htbp]
\includegraphics[height=1.5in]{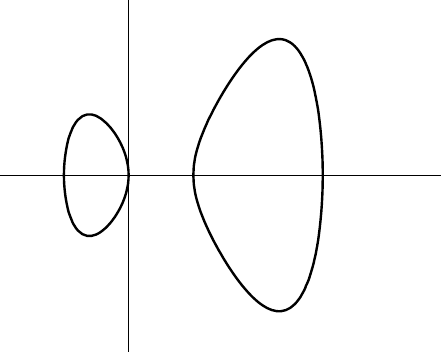}\qquad
\includegraphics[height=1.5in]{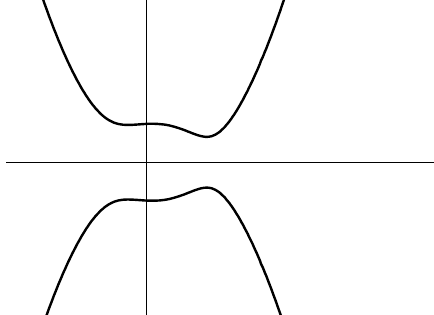}
\caption{\small Two types of double covers $C(\bR)\to\bP^1(\bR)$ of an elliptic $M$-curve $C$. A double cover from the Huisman component is on the right.}\label{wfegb}
\end{figure}
We can use a point $x\in X$
to identify $C(\bR)$ with $\Pic^2(\bR)$  by tensoring with~$\cO(x)$. We  denote the corresponding ovals of $\Pic^2(\bR)$ by 
$X^{(2)}$ and~$Y^{(2)}$.
It is clear that $X+X=X^{(2)}$,
where the left hand side denotes the locus of line bundles $\cO(x+x')$ for $x,x'\in X$.
It is also clear that $X+Y=Y^{(2)}$.
We claim that $Y+Y=X^{(2)}$. Indeed,  
we have a continuous map $C\to\Pic^2(\bR)$ which sends $z$ to $\cO(z+\bar z)$.
By~continuity, the image of this map has to be equal to $X^{(2)}$. In the notation of \ref{awrgawrg},
$$X^{(2)}=\Pic^2_\emptyset(\bR)\quad\hbox{\rm and}\quad Y^{(2)}=\Pic^2_{\{1,2\}}(\bR)=\Pic^2_{H}(\bR)$$
is the Huisman component of $\Pic^2(\bR)$.
If $L\in X^{(2)}$ then the map
$$\phi_L:\,C(\bR)\to\bP^1(\bR)$$ 
represents $C(\bR)$ in the  form $y^2=f_4(x)$,
where $f_4(x)$ is a real polynomial with four real roots  (the left  side of Figure~\ref{wfegb}).
By tradition, one of these roots is usually moved to infinity.
Here $\phi_L$ is projection onto the $x$-axis.
Note that restriction of $\phi_L$ to either  $X$ or $Y$ is neither surjective nor injective.
By contrast, if~$L\in Y^{(2)}$, the map
$\phi_L$ represents the same curve in the form  $y^2=f_4(x)$,
where $f_4(x)$ is a real polynomial with two pairs of  complex conjugate roots (the right side of Figure~\ref{wfegb}).
In accordance with \ref{jhsbdchjs,bdv}, the restriction of $\phi_L$ to both $X$ and $Y$  is a real analytic isomorphism
with inverses given by 
$$x\mapsto \left(x,\sqrt{f_4(x)}\right),\quad x\mapsto \left(x,-\sqrt{f_4(x)}\right).$$
Next we suppose that $C$ is an MHV elliptic M-curve of type A or B.
\begin{figure}[htbp]
\includegraphics[width=\textwidth]{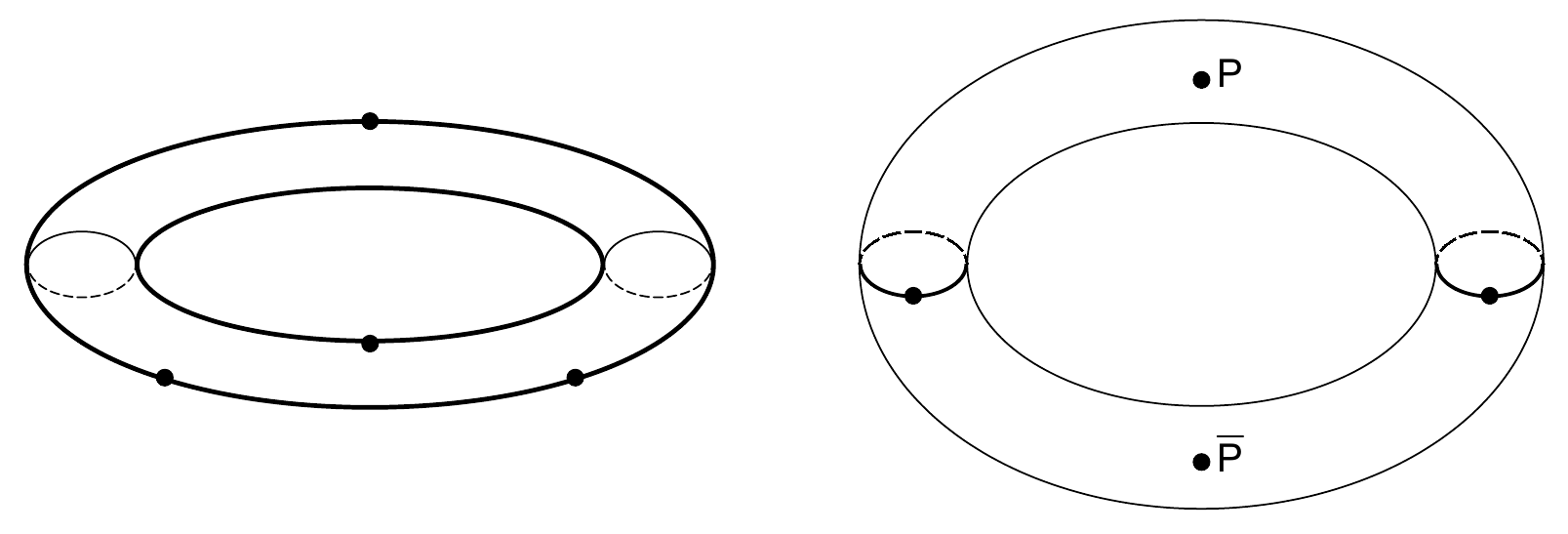}
\caption{\small MHV elliptic M curves of Type A (left) and type B (right).}
\end{figure}
In other words, we add $4$ marked points so that, depending on the type,
\begin{enumerate}
\item[(A)] $p_1,p_2,p_3\in X$ and $p_4\in Y$
\item[(B)] $p_1\in X$, $p_2\in Y$ and $p_3=\bar p_4$ are complex-conjugate points.
\end{enumerate}
We have six special points $p_{ij}=\cO(p_i+p_j)\in\Pic^2 C$. Depending on the type,
\begin{enumerate}
\item[(A)] $p_{12},p_{13},p_{23}\in X^{(2)}$ and $p_{14},p_{24},p_{34}\in Y^{(2)}$.
\item[(B)] $p_{34}\in X^{(2)}$, $p_{12}\in Y^{(2)}$ and $p_{13}=\bar p_{14}$, $p_{23}=\bar p_{24}$
are complex-conjugate.
\end{enumerate}
The scattering amplitude map $\bLa:\,\Pic^2C\to\bP^1$
is given by the line bundle
$$\cL=\cO(p_{12}+p_{34})\simeq\cO(p_{13}+p_{24})\simeq\cO(p_{14}+p_{23}).$$
In both cases (A) and (B), $\cL\in Y^{(4)}$, the Huisman component of $\Pic^2(\Pic^2 C)$.
The~restriction of $\bLa$ to both connected components $X^{(2)}$ and $Y^{(2)}$ of $\Pic^2(\bR)$
is an analytic isomorphism with $\bP^1(\bR)$. In other words, the scattering amplitude map $\bLa$ 
of an elliptic M-curve of type (A) or (B) looks like 
the projection onto the $x$-axis as on the right  side of Figure~\ref{wfegb}.
We get a $4$-parameter family of smooth scattering amplitude probability measures
given (up to normalization) 
by 
$$A={dx\over\sqrt{f_4(x)}},$$
where $f_4$ is a real polynomial with two pairs of  complex conjugate roots.
A familiar $2$-parameter subfamily is given by
\begin{center}
\begin{tabular}{ c c }
${M(a,b)\over\pi}{dx\over \sqrt{(x^2+a^2)(x^2+b^2)}}$
&\raisebox{-40pt}{{\includegraphics[height=1.2in]{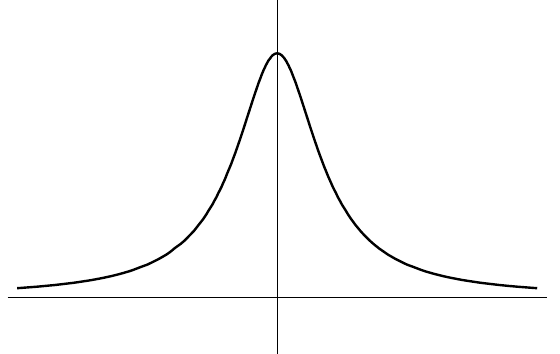}}}\cr
\end{tabular}
\end{center}
where $a\ne b$ and $M(a,b)$ is the arithmetic-geometric mean of $a$ and $b$. 
\end{Example}

\begin{Remark}\label{argewrgrg}
Notice that the scattering probability measure $\rho(x)\,|dx|$ is a smooth measure on $\oM_{0,4}(\bR)\simeq\bR\bP^1$
including the point at $\infty$, i.e.~ $\rho(1/x)\over x^2$ is smooth at $0$.
In~fact presenting it as a probability measure on $\bR$
depends on a specific choice of the identification $\oM_{0,4}(\bR)\simeq\bR\bP^1$ and it is well-known that
there are $|S_3|=6$ possibilities, which can result in  differently-looking density graphs:
\begin{center}
\begin{tabular}{c }
\includegraphics[width=\textwidth]{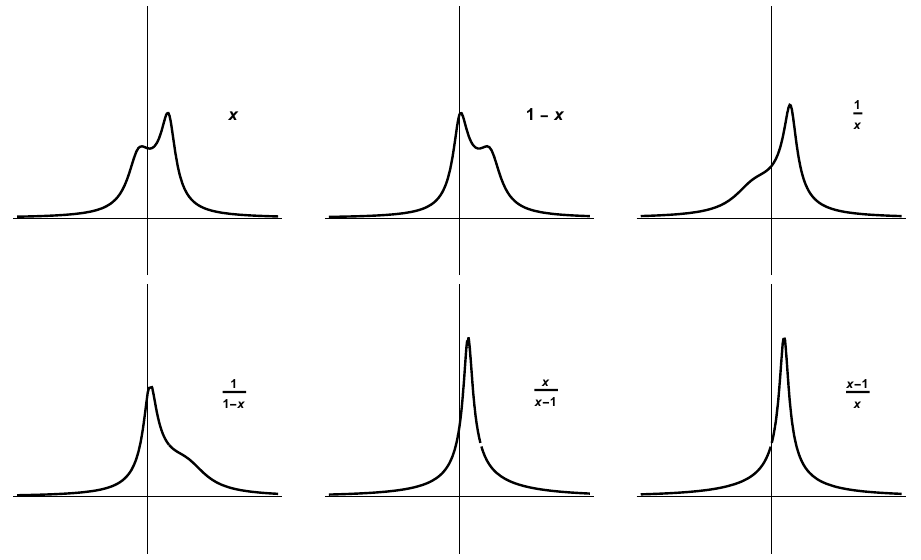}\cr
\end{tabular}
\end{center}
\end{Remark}

\begin{Review}[genus $1$ -- continued]
In type (A), the real form of $M_{0,4}$ is $M^A_{0,4}(\bR)=M_{0,4}(\bR)$, the configuration space of $4$ points in $\bP^1(\bR)$.
Viewing $M_{0,4}(\bC)$ as a Riemann sphere punctured at $0,1,\infty$ via the cross-ratio map,
$M_{0,4}(\bR)$ is identified with the real axis without $0,1,\infty$ and $\oM_{0,4}(\bR)$
with $\bP^1(\bR)$, the equator of the Riemann sphere.
The~scattering amplitude map $\bLa:\,\Pic^2(\bR)\to \oM_{0,4}(\bR)$ has the property that
\begin{equation}\label{sgshrHR}
\bLa(p_{14})=\bLa(p_{23})=0,\quad \bLa(p_{13})=\bLa(p_{24})=1,\quad \bLa(p_{12})=\bLa(p_{34})=\infty.
\cooltag\end{equation}
In type (B), the real form is $M^B_{0,4}(\bR)$, the configuration space of two real and two complex-conjugate points in $\bP^1(\bC)$.
If $\lambda=(p_1,p_2;p_3,p_4)$ is their cross-ratio then
$$\overline\lambda=(p_1,p_2;p_4,p_3)=1-\lambda.$$
Therefore,
$$M^{2,1}_{0,4}(\bR)=\left\{\Re(z)={1\over2}\right\}\subset\bC$$
and the compactification $\oM^{2,1}_{0,4}(\bR)$ is obtained by adding a single point (at $\infty$).
The~scattering amplitude map $\bLa:\,\Pic^2(\bR)\to \oM^{2,1}_{0,4}(\bR)$ still has property \eqref{sgshrHR}
but now only $\infty$ is part of the real locus $\oM^{2,1}_{0,4}(\bR)$.
This corresponds to the fact that points $p_{13}=\bar p_{14}, p_{23}=\bar p_{24}$ are not in $\Pic^2(\bR)$.
\end{Review}

\begin{Review}[Real planar locus W]
Recall from \ref{evwevwev} that the planar locus 
$W\subset \Pic^{g+1}C$ has codimension $3$ and (for a general curve)
parametrizes  realizations of $C$ as a degree $g+1$ 
plane curve with ${g(g-3)\over 2}$ nodes away from the marked points $p_1,\ldots,p_n$.
Generically along $W$, the scattering amplitude map is resolved by the blow-up \eqref{agargasr} of $W$
with exceptional divisor $\hat W$ and
the scattering amplitude form vanishes
along $\hat W$ with multiplicity $2$. 
Suppose that $C$ is a generic real MHV M-curve. Then line bundles in $W(\bR)$
give realizations of $C$ as {\em Harnack curves}, more precisely real plane curves of degree $g+1$ and genus $g$ with $g+1$ real components.
These curves have ${g(g-3)\over 2}$ acnodes, i.e.~isolated real points (where two complex-conjugate branches 
of the complex plane curve intersect transversally).
If $g$ is odd then all $g+1$ components are ovals (separate $\bR\bP^2$ into a disk and a M\"obius strip).
If $g$ is even, there are $g$ ovals and one pseudo-line (a generator of $\pi_1(\bR\bP^2)$), see~\cite{GH}.
It follows that a general real line in  $\bR\bP^2$ has even degree on all components of $C(\bR)$ if $g$ is odd
and on all but one if $g$ is even. So  for a general real MHV M-curve
$$W(\bR)\subset\Pic_\emptyset^{g+1}(\bR)$$ if $g$ is odd. If $g$ is even,
$W(\bR)$ is contained in the union of $\Pic_I^{g+1}(\bR)$ for $|I|=1$. 
It~follows that the scattering amplitude probability measures 
$|A_\emptyset|$ ($g$ odd) and $|A_I|$ for $|I|=1$ ($g$ even) are the
only  ones 
that vanish to the order $2$ along $\hat W(\bR)$.
\end{Review}

In the rest of the section  we 
study scattering  measures of genus~$2$ curves. 

\begin{Review}\label{qfvqev}
Let $C$ be a smooth MHV M-curve of genus $2$ given by the equation $y^2=f(x)$,
where $f$ is a real polynomial of degree $5$ with $5$ distinct real roots.
We denote the ovals of $C(\bR)$ by $C_1$, $C_2$ and $C_3$. Like in Theorem~\ref{sGARSGARHA},
we assume that 
$$p_i\in C_i\quad  \hbox{\rm for}\quad i=1,2,3$$
and that
$$p_4,p_5\in C_3\quad\hbox{\rm (type A)\quad or}\quad p_4=\bar p_5\quad\hbox{\rm    (type B)}.$$
See Figure~\ref{wEGsg}. We~follow notation of \ref{sfgasgsRH} for special points and divisors. In particular, 
$$K\in\Pic_{\emptyset}^2(\bR)\quad  \hbox{\rm and}\quad P=p_1+\ldots+p_5\in\Pic_{\{1,2,3\}}^5(\bR).$$
The special points $\delta=P-K$,\quad  $\delta_{i}=K+p_i$ and $\delta_{ij}=P-p_i-p_j$ are distributed among four 
components of 
$\Pic^3(\bR)$ as follows:
\medskip
\begin{center}
\begin{tabular}{ |c||c|c|c|c|} 
 \hline
 Type  & $\Pic_H^3(\bR)$ & $\Pic_{\{1\}}^3(\bR)$ & $\Pic_{\{2\}}^3(\bR)$ & $\Pic_{\{3\}}^3(\bR)$\\ 
 \hline
 A & $\delta$, $\delta_{45}$, $\delta_{34}$, $\delta_{35}$  & $\delta_1$, $\delta_{23}$, $\delta_{24}$, $\delta_{25}$  & $\delta_2$, $\delta_{13}$, $\delta_{14}$, $\delta_{15}$  & $\delta_3$, $\delta_4$, $\delta_5$, $\delta_{12}$ \\ 
 B & $\delta$, $\delta_{45}$ & $\delta_1$, $\delta_{23}$ & $\delta_2$, $\delta_{13}$ & $\delta_3$, $\delta_{12}$\\
 \hline
\end{tabular}
\end{center}
\medskip
The map $\phi_P$ embeds $C(\bR)$ into $\bP^3(\bR)$.  
Let $\bP^2\subset\bP^3$ be the plane passing  through $p_1,\ldots,p_5$.
Let $\dP=\Bl_{p_1,\ldots,p_5}\bP^2$ be the quartic  del Pezzo surface.
Depending on the type of the MHV M-curve $C$, there are two possibilities for its real form:
$$\dP^A(\bR)=\Bl_{p_1,\ldots,p_5}\bP^2(\bR)\quad \hbox{\rm or}\quad \dP^B(\bR)=\Bl_{p_1,p_2,p_3}\bP^2(\bR).$$
\end{Review}

\begin{Review}
By Theorem~\ref{DSvSDb}, 
$\Pic^{3}(\bR)$ is disjoint from the ramification divisor $R$.\footnote{
Since $\delta\in\Pic^{3}(\bR)$, this has a curious geometric consequence:
the unique quadric surface in $\bP^3$ containing $C$ is smooth (by Lemma~\ref{SFgSFhF}).}
This~implies by Theorem~\ref{vasdgsgG} that the
degree $4$ morphism
$$\bbLa:\,\Bl_{16}\Pic^3C(\bR)\to\Bl_{p_1,\ldots,p_5}\bP^2(\bR)\quad \hbox{\rm in type (A)},$$
$$\bbLa:\,\Bl_8\Pic^3C(\bR)\to\Bl_{p_1,p_2,p_3}\bP^2(\bR)\quad \hbox{\rm in type (B)}$$
induced by the scattering amplitude map 
is a real-analytic isomorphism on each connected component $\Pic^3_IC(\bR)$
blown up at $4$ (in type A) or $2$ (in type B) points.
Note that topologically each of these components 
is a connected sum of a torus $T^2$
with $4$ (in type A) or $2$ (in type B) $\bR\bP^2$'s while
$\dP(\bR)$ is a connected sum of $6$ (in type A) or $4$ (in type B) $\bR\bP^2$'s.
The scattering amplitude map $\bbLa$ provides a real-analytic isomorphism
between these (non-orientable) surfaces.
\end{Review}

\begin{Review}\label{asgasrghar}
By Theorem~\ref{vasdgsgG}, the preimage of each $(-1)$-curve in $\dP$
is the union of the ``$\Delta$'' and ``$E$''-type divisors, where the former is a $\bP^1$
and the latter is a copy of~$C$.  
In the following table we give the preimages under $\bbLa$ of real $(-1)$-curves in $\dP(\bR)$.
Recall that the ``$E$''-type divisors are 
$$E=K+C,\quad E_i=P-p_i-C,\quad E_{ij}=C+p_i+p_j.$$
The range of indices $k, l$ in the table is $3,4,5$.
We use notation 
$$E(C_i)=K+C_i,\quad E_i(C_j)=P-p_i-C_j,\quad E_{ij}(C_k)=C_k+p_i+p_j.$$

\medskip
\begin{center}
\begin{tabular}{ |c|c||c|c|c|c|} 
 \hline
  &  curve in $\dP(\bR)$  & $\Pic_{H}^3(\bR)$ & $\Pic_{\{1\}}^3(\bR)$ & $\Pic_{\{2\}}^3(\bR)$ & $\Pic_{\{3\}}^3(\bR)$\\ 
 \hline
 A & conic                        & $\Delta$       &$E(C_1)$      &$E(C_2)$      &$E(C_3)$       \\
    & exceptional               &$E_1(C_1)$    &$\Delta_1$    &$E_1(C_3)$  &$E_1(C_2)$    \\
    &                                  &$E_2(C_2)$    &$E_2(C_3)$  &$\Delta_2$    & $E_2(C_1)$   \\
    &                                  &$E_k(C_3)$    &$E_k(C_2)$   &$E_k(C_1)$  & $\Delta_k$     \\
    & line                           &$\Delta_{kl}$  &$E_{kl}(C_1)$&$E_{kl}(C_2)$&$E_{kl}(C_3)$\\
    &                                  &$E_{2k}(C_1)$&$\Delta_{2k}$ &$E_{2k}(C_3)$ & $E_{2k}(C_2)$\\
    &                                  &$E_{1k}(C_2)$&$E_{1k}(C_3)$&$\Delta_{1k}$&$E_{1k}(C_1)$\\
    &                                  &$E_{12}(C_3)$&$E_{12}(C_2)$&$E_{12}(C_1)$&$\Delta_{12}$\\
  \hline
 B & conic            & $\Delta$       &$E(C_1)$      &$E(C_2)$      &$E(C_3)$       \\
    & exceptional               &$E_1(C_1)$    &$\Delta_1$    &$E_1(C_3)$  &$E_1(C_2)$    \\
    &                                  &$E_2(C_2)$    &$E_2(C_3)$  &$\Delta_2$    & $E_2(C_1)$   \\
    &                                  &$E_3(C_3)$    &$E_3(C_2)$   &$E_3(C_1)$  & $\Delta_3$     \\
    & line                           &$\Delta_{45}$  &$E_{45}(C_1)$&$E_{45}(C_2)$&$E_{45}(C_3)$\\
    &                                  &$E_{23}(C_1)$&$\Delta_{23}$ &$E_{23}(C_3)$ & $E_{23}(C_2)$\\
    &                                  &$E_{13}(C_2)$&$E_{13}(C_3)$&$\Delta_{13}$&$E_{13}(C_1)$\\
    &                                  &$E_{12}(C_3)$&$E_{12}(C_2)$&$E_{12}(C_1)$&$\Delta_{12}$\\

 \hline
\end{tabular}
\end{center}
\end{Review}

\medskip

\begin{Remark}\label{argarsga}
The scattering amplitude  measure on each connected component $\Pic^3_I(\bR)$ is a probability measure
on a real algebraic surface that in any real-analytic chart $U$ has form $\rho(x,y)|dxdy|$, where $\rho(x,y)>0$ is a smooth, in fact real-analytic function.
The behavior of probability measure under the blow-up is as follows. Locally, we blow-up the origin $(0,0)$ of the chart,
$$\pi:\,\Bl_{(0,0)}U\to U.$$
The surface $\Bl_{(0,0)}U$ is non-orientable and contains an exceptional  curve
$\bR\bP^1$, the~preimage of $(0,0)$
(see the cover artwork of \cite{Sh}). A tubular neighborhood of the curve $\bR\bP^1$ is covered by two charts that correspond
to standard charts of $\bR\bP^1$. In each chart, $\pi$ has the form $(u,v)\mapsto (u,uv)$
and thus the probability measure is
\begin{equation}\label{sGsgav}
\rho(u,uv)|u||dudv|.
\cooltag\end{equation}
Note that $u=0$ is the local equation of the exceptional curve.
The probability  measure \eqref{sGsgav} vanishes along it to the order of $|u|$.
More generally, if $|A|$ is a probability measure on a real-analytic manifold $X$ smooth and non-vanishing generically along
a  submanifold $Y$ of codimension $c$ then the same measure on the blow-up $\Bl_YX$
vanishes along the exceptional divisor generically to the order of $|u|^{c-1}$.
\end{Remark}

\begin{Review}
Back in our situation,
the scattering amplitude measure on $\dP(\bR)$ is smooth and positive away from
four (in type A) or two (in type B)  projective lines that correspond to ``$\Delta$'' divisors in the Table~\ref{asgasrghar}.
Along these lines, the scattering amplitude form vanishes like in \eqref{sGsgav}.
These disjoint (-1)-curves can always be contracted to points giving morphisms $\dP(\bR)\to (\bR\bP^1)^2$.
Combining everything together gives a commutative diagram
$$\begin{CD}
\Bl \Pic^3_I(\bR) @>\bbLa>> \dP(\bR)\\
@VVV                                                  @VVV \\
\Pic^3_I(\bR) @>\simeq>>      (\bR\bP^1)^2              
\end{CD}$$
where vertical arrows are blow-downs of four (type A) or two (type B) projective lines and the bottom arrow is a real-analytic isomorphism
(note that algebraically these varieties can't be more different!).
It follows that we can view the scattering amplitude measure on $\dP(\bR)$
as a smooth positive probability measure on $(\bR\bP^1)^2$ (subject to M\"obius transformations of coordinates  as in Remark~\ref{argewrgrg}).
\end{Review}

\begin{Review}
So far we have treated the Huisman component and the other connected components of $\Pic^3(\bR)$ 
on the equal footing, but there is an important difference.
\begin{figure}[htbp]
\begin{center}
\begin{tabular}{ l c r}
\raisebox{10pt}{\includegraphics[height=0.9in]{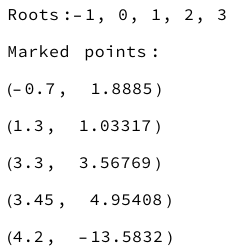}}\quad$\ $&\includegraphics[height=1.1in]{results1B.pdf}&\includegraphics[width=1.4in]{results1C.pdf}\cr
\raisebox{10pt}{\includegraphics[height=0.9in]{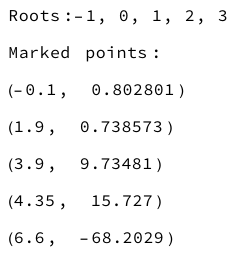}}\quad$\ $&\includegraphics[height=1.2in]{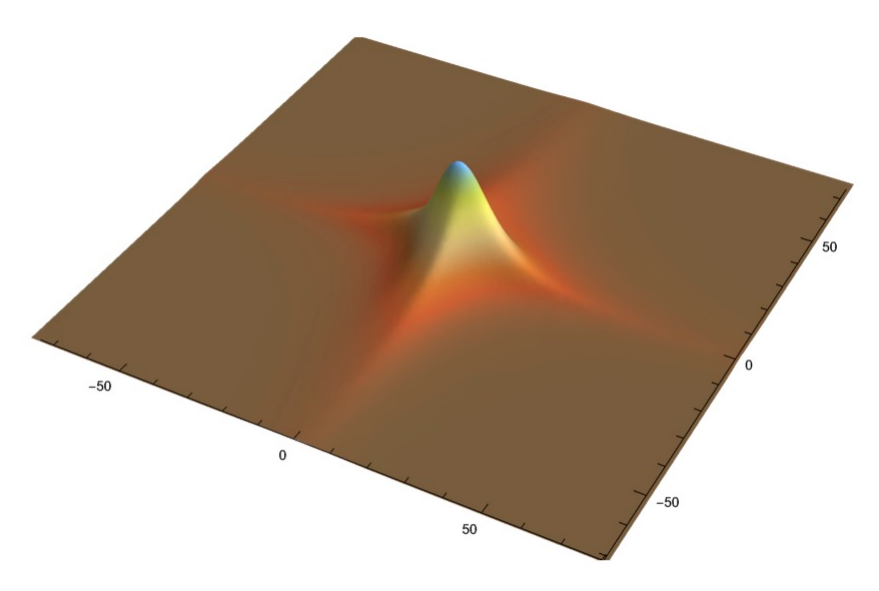}&\includegraphics[width=1.4in]{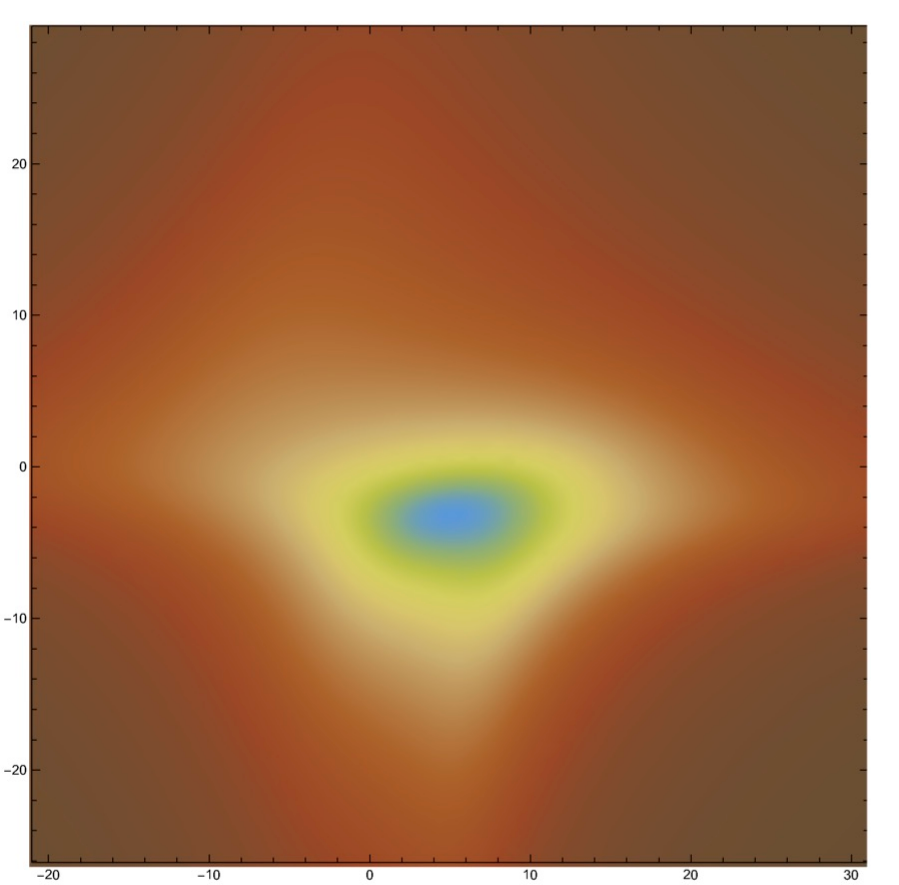}\cr
\raisebox{10pt}{\includegraphics[height=0.9in]{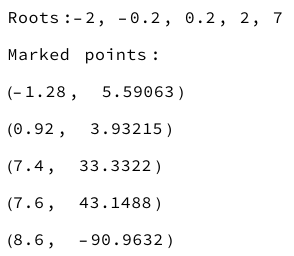}}\quad$\ $&\includegraphics[height=1.1in]{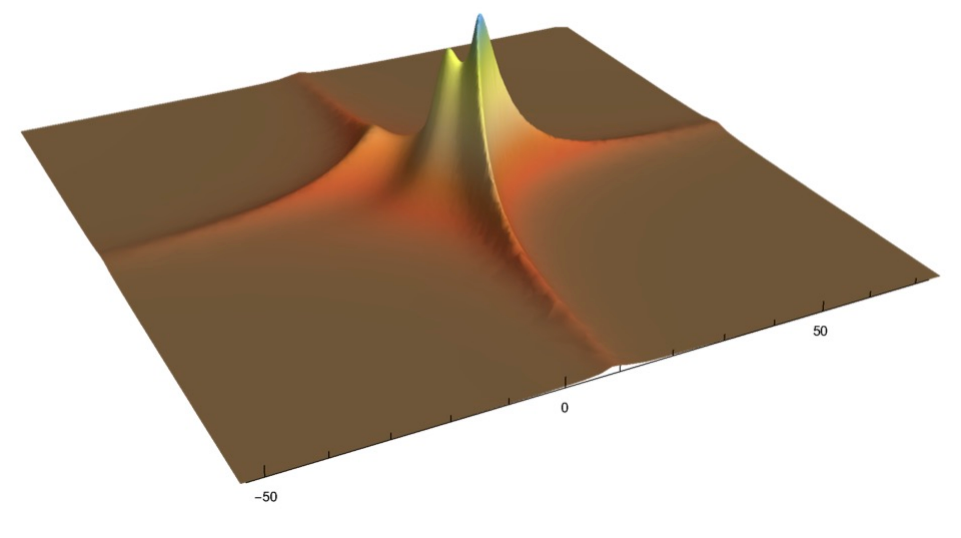}&\includegraphics[width=1.4in]{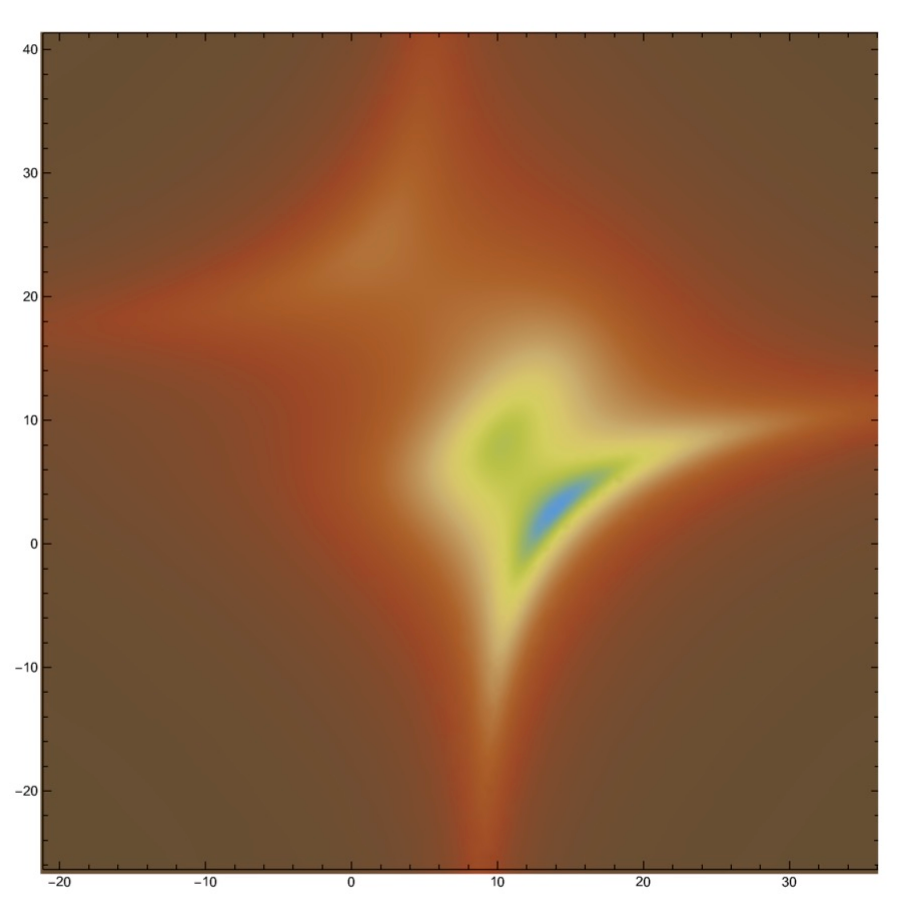}\cr
\end{tabular}
\end{center}
\caption{Scattering amplitudes  in genus $2$, Huisman's component}\label{argerhqer}
\end{figure}
For simplicity, we consider Type A only.
Recall that the space $\dP$ is an artifact of the hyperelliptic case, the actual observable is not this surface but $\dPf=\oM_{0,5}$.
The map $\dP\to \dPf$ is a contraction of the conic that corresponds to the exceptional divisor $\Delta$
only for the Huisman's component (see \ref{asgasrghar}), in which case
the scattering amplitude
map gives a real-analytic isomorphism
$$\bLa:\,\Bl_{\delta_{45}, \delta_{34}, \delta_{35}}\Pic^3_H(\bR)\to \dPf=\oM_{0,5}(\bR).$$
The $10$ exceptional divisors of $\dPf$ (the ``Petersen graph'') are images of the $10$ proper transforms $l_{ij}$ in $\dP$ of lines in $\bP^2$ 
connecting points 
$p_i$ and $p_j$ of the conic pairwise. The lines $l_{34}$, $l_{35}$, $l_{45}$ are
images of the exceptional divisors $\Delta_{34}$, $\Delta_{35}$, $\Delta_{45}$ on $\Bl_{\delta_{45}, \delta_{34}, \delta_{35}}\Pic^3_H$.
These three lines on $\dPf$ are contracted by the map 
$$\oM_{0,5}\to(\bP^1)^2=(\oM_{0,4})^2$$ 
given by cross-ratios $1345$ and $2345$.
This gives a  diagram of Theorem~\ref{sGARSGARHA}
$$\begin{CD}
\Bl_{\delta_{45}, \delta_{34}, \delta_{35}}\Pic^3_H(\bR) @>\bLa>\simeq> \oM_{0,5}(\bR)\\
@VVV                                                  @VV{1345\atop 2345}V \\
\Pic^3_H(\bR) @>\simeq>>      (\oM_{0,4}(\bR))^2              
\end{CD}$$
where horizontal arrows are real-analytic isomorphisms.
To summarize, for the Huisman component we have a preferred choice of cross-ratios
that gives a smooth positive scattering amplitude probability measure on $ (\bR\bP^1)^2$.
For the reader's amusement, we include a few pretty examples of these probability density functions 
in Figure~\ref{argerhqer}, rendered using Algorithm~\ref{asrgarh}.
\end{Review}

\begin{Review}\label{qergarh}
For a different choice of two cross-ratios, the scattering amplitude measure of the Huisman's component, smooth on $\oM_{0,5}(\bR)$,
acquires a singularity in $(\bR\bP^1)^2$ of the same type as the scattering amplitude of a non-Huisman component
has already in the interior of $\oM_{0,5}(\bR)$. 
Consider the local picture as in Remark~\ref{argarsga},
the real blow-up $\pi:\,\Bl_{(0,0)}U\to U$, where $U$ is a real-analytic chart.
Then $\pi$ has the form $(u,v)\mapsto (u,uv)$.
Suppose the probability measure $\rho(u,v)|dudv|$ on the blow-up
is continuos and positive generically along the exceptional curve $u=0$,
with the limit function $\bar\rho(v)=\lim\limits_{u\to 0}\rho(u,v)$.
In coordinates $(x,y)$ of $U$, 
$$\lim_{x\to 0}{\rho(x,vx) |x|}=\bar\rho(v)$$
for every  slope $v$. So the restriction of $\rho(x,y)$ to every line $y=vx$ grows as $\bar\rho(v)\over |x|$
when $|x|\to 0$.
In Figure~\ref{adfadfbadf}, we show how the probability density function
of the smooth scattering amplitude from the top of Figure~\ref{argerhqer}
looks like in $(\bR\bP^1)^2$ for two different choices of cross-ratios.
In both pictures notice the line of zero probability density running through singularities.
\begin{figure}[htbp]
\begin{center}
\begin{tabular}{c c c c}
\includegraphics[height=1in]{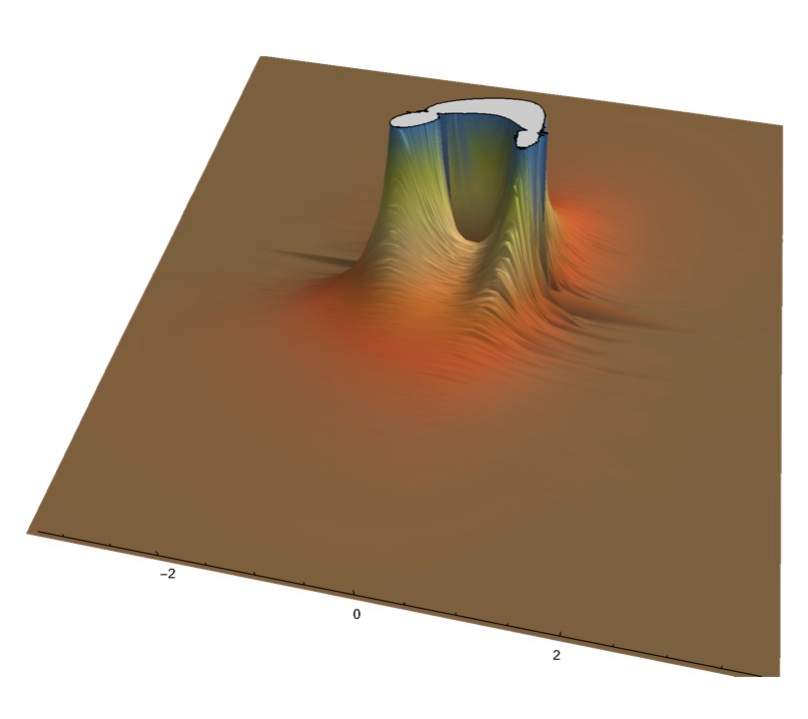}&\includegraphics[width=1in]{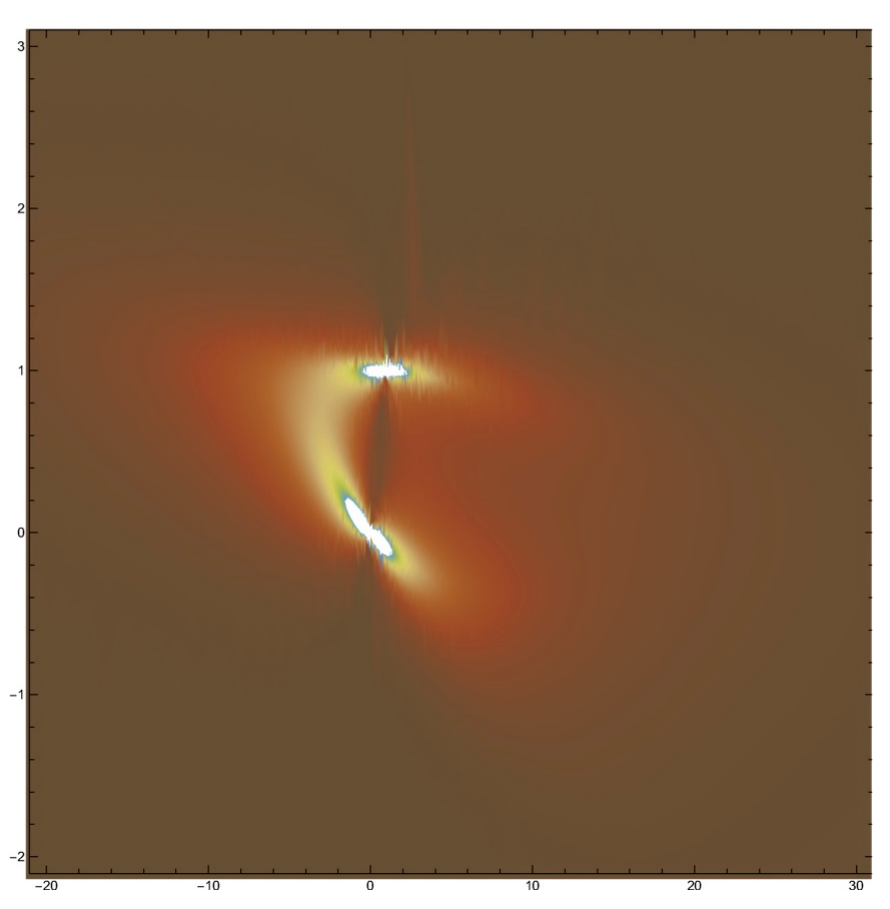}&
\includegraphics[height=1in]{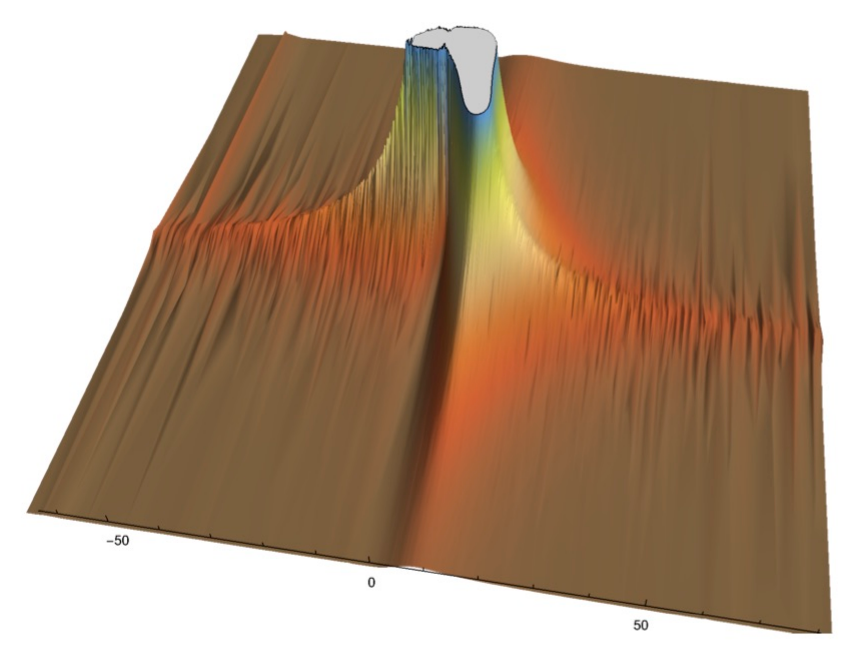}&\includegraphics[width=1in]{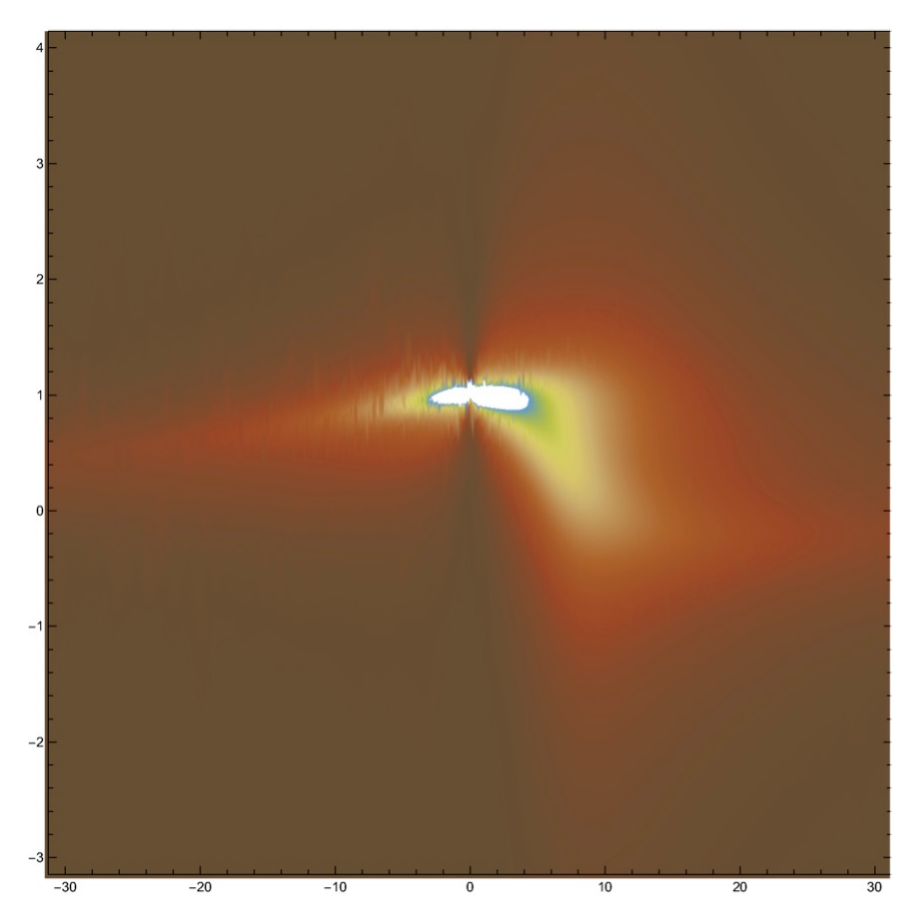}\cr
\end{tabular}
\end{center}
\caption{Rendering of singularities. Left:  $4123,5123$;\quad right: $1345,5123$.}\label{adfadfbadf}
\end{figure}
\end{Review}

\section{Forest behind the hypertrees}\label{SDgSfh}

\begin{Review}
Let  $C$ be a  {\em maximally degenerate} stable  curve:
every irreducible component is a $\bP^1$ with three special points (either nodes or marked points).
Equivalently, an~on-shell diagram is trivalent.
We will characterize MHV curves of this type.
\end{Review}

\begin{Lemma}
If $C$ is an MHV curve then $L$ has degree $0$ or $1$ on every component of $C$.
The on-shell diagram
is a trivalent graph with $d$ black circles (where $L$ has degree~$1$)
and the remaining white circles (where $L$ has degree~$0$). 
\end{Lemma}

\begin{proof}
We argue by contradiction and suppose that the degree of $L$ on some irreducible component $B$ is at least $2$.
By Lemma~\ref{qwefv2e2}, the remaining part of the curve $A$ is connected.
Let $g_A$ be its genus and let $d_A$ be the degree of $L$ on $A$.
Then $d_A\le d-2=g-1$. On the other hand, $g_A\ge g-2$.
If $g_A>0$ then, by Lemma~\ref{wefwetb}~(1), $d_A=g-1$ and $d_A=g_A+1$. 
If $g_A=0$ then $g=2$ and so again $d_A=g_A+1$.
In~both cases, by Lemma~\ref{wefwetb} (2), $A$ contains at most $g_A+3=n-2$
marked points. Thus $B$ contains at least $2$ marked points. This contradicts Lemma~\ref{qwefv2e2}.
\end{proof}

\begin{Assumption}\label{zxfbzxfzf}
By Lemma~\ref{SFbSFh} and Remark~\ref{AKHJDbfHSDB},
if $C$ is an MHV curve then every connected component of a subgraph of white circles of the on-shell diagram
is a tree with at most one marked point. In addition, 
by Lemma~\ref{wefwetb}, a connected subcurve $A\subset C$ of arithmetic genus~$p$ should satisfy the following conditions:
\begin{enumerate}
\item If $p>0$ then $\deg L|_A\ge p+1$.
\item If $\deg L|_A=p+1$ then $A$ contains at most $p+3$ marked points.
\end{enumerate}
From now on we will assume that all these  conditions hold.
\end{Assumption} 

The precise shape of these subtrees is not important for the calculation of the scattering amplitude map and form
and there are two ways to ignore them.

\begin{Definition} Let $C$ be a maximally degenerate curve satisfying Assumption~\ref{zxfbzxfzf}.
\begin{enumerate}
\item 
For each maximal connected subtree of white circles of the on-shell diagram, contract all interior edges of the tree so that it shrinks to a point. We call this point a white {\em megacircle}.
It is possible for it to have  more than three outgoing edges. A similar operation appears in \cite{MHV}.
Geometrically, this corresponds to a curve $C_s$ obtained by smoothing some nodes of $C$.
\item
Contract  each component of $C$ where $L$ has degree~$0$ to a singular point
of a {\em hypertree curve} $\Sigma$ \cite[Def.~1.8]{CT_Crelle}.
$\Sigma$ is not a stable curve: it has singularities worse than nodes
and marked points at  singularities. But it is  convenient: every morphism to $\bP^1$
given by a line bundle in $\Pic^{\vec d}C$ factors through~$\Sigma$.
\end{enumerate}
\end{Definition}

The following lemma was essentially proved in \cite{MHV}.

\begin{Lemma}\label{wefvev4tbg2} Let $C$ be a maximally degenerate curve satisfying Assumption~\ref{zxfbzxfzf}. Then
its on-shell diagram has the following properties:
\begin{enumerate}
\item
Black circles 
are only connected to white megacircles and marked points.
\item
Each white megacircle is connected to exactly one marked point.
\item
The data of the on-shell diagram is equivalent to the data of triples in $\{1,\ldots,n\}$:
$$\Gamma=\{\Gamma_1,\ldots,\Gamma_d\}.$$
Each $\Gamma_i$ is associated to one of the $d$ black circles of the on-shell diagram. For such a circle, the associated $\Gamma_i$ 
consists of the markings that are either attached directly at that black circle, or sit on a white megacircle connected to it.
\end{enumerate}
\end{Lemma}

\begin{proof}
The first step is to place a dummy white megacircle with two edges
between every marked point and  a black circle directly connected to it.
After this operation, no black circle is connected directly to a marked point.
Let $r$ be the total number of white megacircles.
The on-shell diagram can be built step-by-step as follows: start with $r$ white megacircles
connected to marked points. The Euler characteristic of this graph is $r$. 
Now add black circles one-by-one. Every time
we add a new black (trivalent!) circle, the Euler characteristic goes down by at most $2$ and exactly by $2$ only
if the black circle is connected to an already constructed graph at three points.
It follows that at the end of the construction the Euler characteristic will be at least $r-2(g+1)$, on the other hand it should be equal to $1-g$.
Thus $r\le g+3=n$. On~the other hand, $n\le r$ because each white megacircle is connected to  at most one marked point.
Thus we have all the conclusions:  $r=n$, all white megacircles are connected to marked points and black vertices
are not connected to each other. It remains to note that every black vertex is connected to three different marked points.
Indeed, otherwise a black vertex is connected to a white vertex by two paths, producing a subcurve of arithmetic genus $1$ and degree $1$. which contradicts  Assumption~\ref{zxfbzxfzf}.
\end{proof}

Not every  maximally degenerate curve 
satisfying satisfying Assumption~\ref{zxfbzxfzf} is an MHV curve.
The~precise condition was found in \cite[Theorems 2.4 and ~3.2]{CT_Crelle}:

\begin{Theorem}[\cite{CT_Crelle}]\label{hjgcjghcj}
A maximally degenerate stable curve $C$  is an MHV curve
if and only if it is give by triples $\Gamma$ 
as in Lemma~\ref{wefvev4tbg2} that form a CT hypertree\footnote{See \cite{CT_Crelle} for a motivation 
and a more general definition for subsets other than triples.
There are  other notions of hypertrees in literature, so to distinguish our case we  use terminology  ``CT hypertree".}, 
i.e.~satisfy~\eqref{CondS}:
$$\bigl|\bigcup_{j\in S}\Gamma_j\bigr|\ge |S|+2
\quad\hbox{\rm for every $S\subset\{1,\ldots,d\}$}.$$
The scattering amplitude map is birational with the inverse map (called $v$ in \cite{CT_Crelle})
\begin{equation}\label{nbcxjhfx}
\bLa^{-1}:\,M_{0,n}\to\Pic^{\vec d}C,
\cooltag\end{equation}
$$\bLa^{-1}(q_1,\ldots,q_n)=v^*\cO_{\bP^1}(1),$$
where $v:\,C\to\bP^1$
is a unique morphism such that $v(p_i)=q_i$ and $v^*\cO_{\bP^1}(1)\in \Pic^{\vec d}C$.
\end{Theorem}

\begin{Definition}[\cite{CT_Crelle}]
A CT hypertree $\Gamma$ is called irreducible if \eqref{CondS} is a strict inequality for every $S$ such that $1<|S|<d$.
\end{Definition}

\begin{Example}
In genus $0$ and $1$, irreducible CT hypertrees are Examples~\ref{asgsrgr} and~\ref{sdfvwfv}, respectively.
In genus 2, there are none. See Figure~\ref{SGDFHADFH} for irreducible hypertrees in genus $3,4,5,6$.
\begin{figure}[htbp]
\centerline{\includegraphics[width=\textwidth]{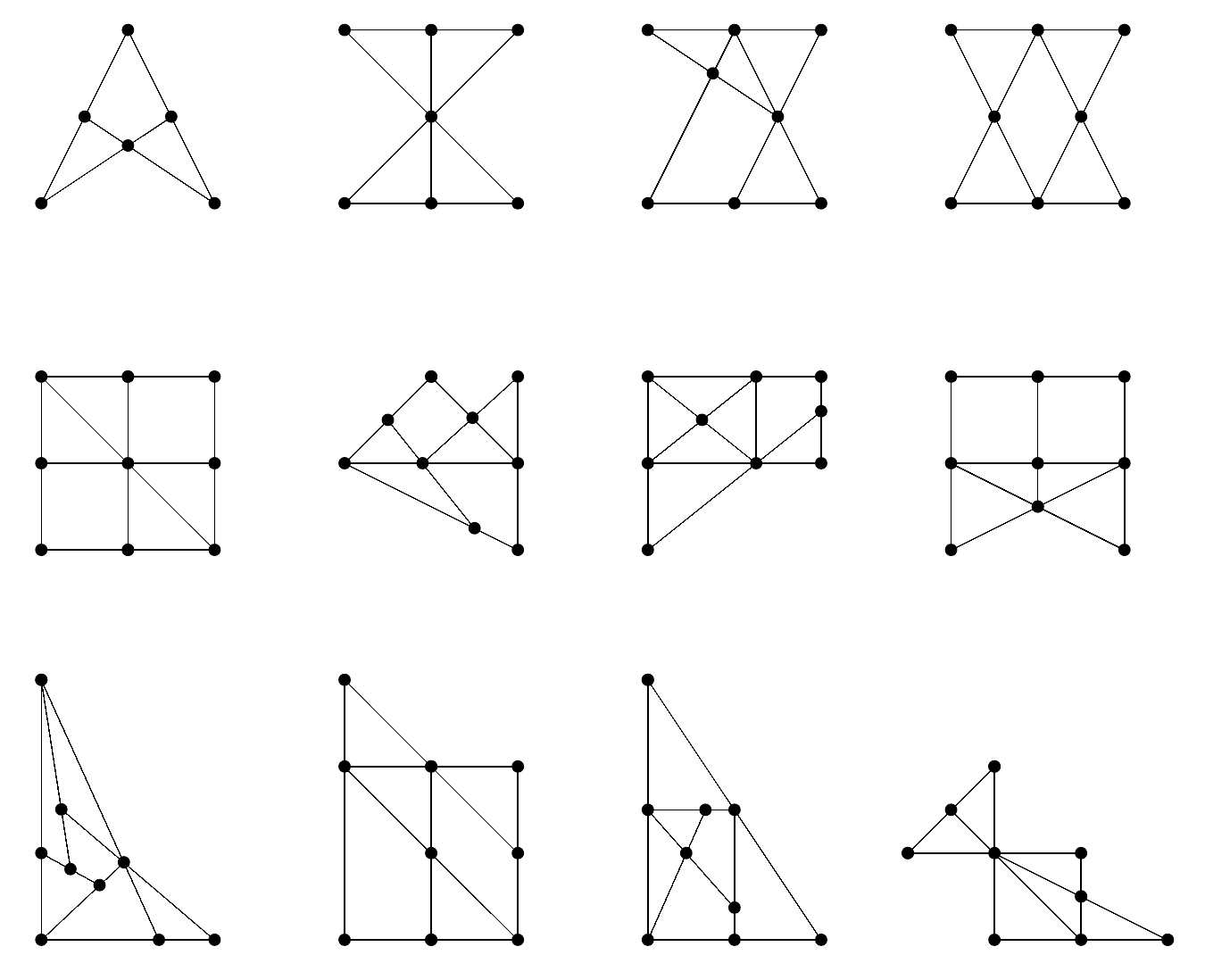}}
\caption{}\label{SGDFHADFH}
\end{figure}
The database of irreducible CT hypertrees in genus at most $8$ up to symmetries
along with equations and classes of the corresponding hypertree divisors
was created by Opie and Scheidwasser \cite{OS}.
\end{Example}

\begin{Review}
The main goal of \cite{CT_Crelle} was 
to construct  {\em hypertree divisors} $D_\Gamma\subset M_{0,n}$ with  good properties.
Suppose that $\Gamma$ is an irreducible CT hypertree and $g\ge3$.
Then $\bLa^{-1}$ contracts a unique divisor $D_\Gamma\subset M_{0,n}$ 
given~by
\begin{itemize}
\item choosing a configuration of different points $p_1,\ldots,p_n\in\bP^2$ such that different points $p_i,p_j,p_k$
are collinear if and only if $\{i,j,k\}\in\Gamma$,
\item projecting points $p_1,\ldots,p_n$ from a point 
$p\in\bP^2$ to points $q_1,\ldots,q_n\in\bP^1$,
\item and representing 
the datum $(\bP^1;q_1,\ldots,q_n)$ by a point of $M_{0,n}$.
\end{itemize}
\centerline{\includegraphics[width=0.8\textwidth]{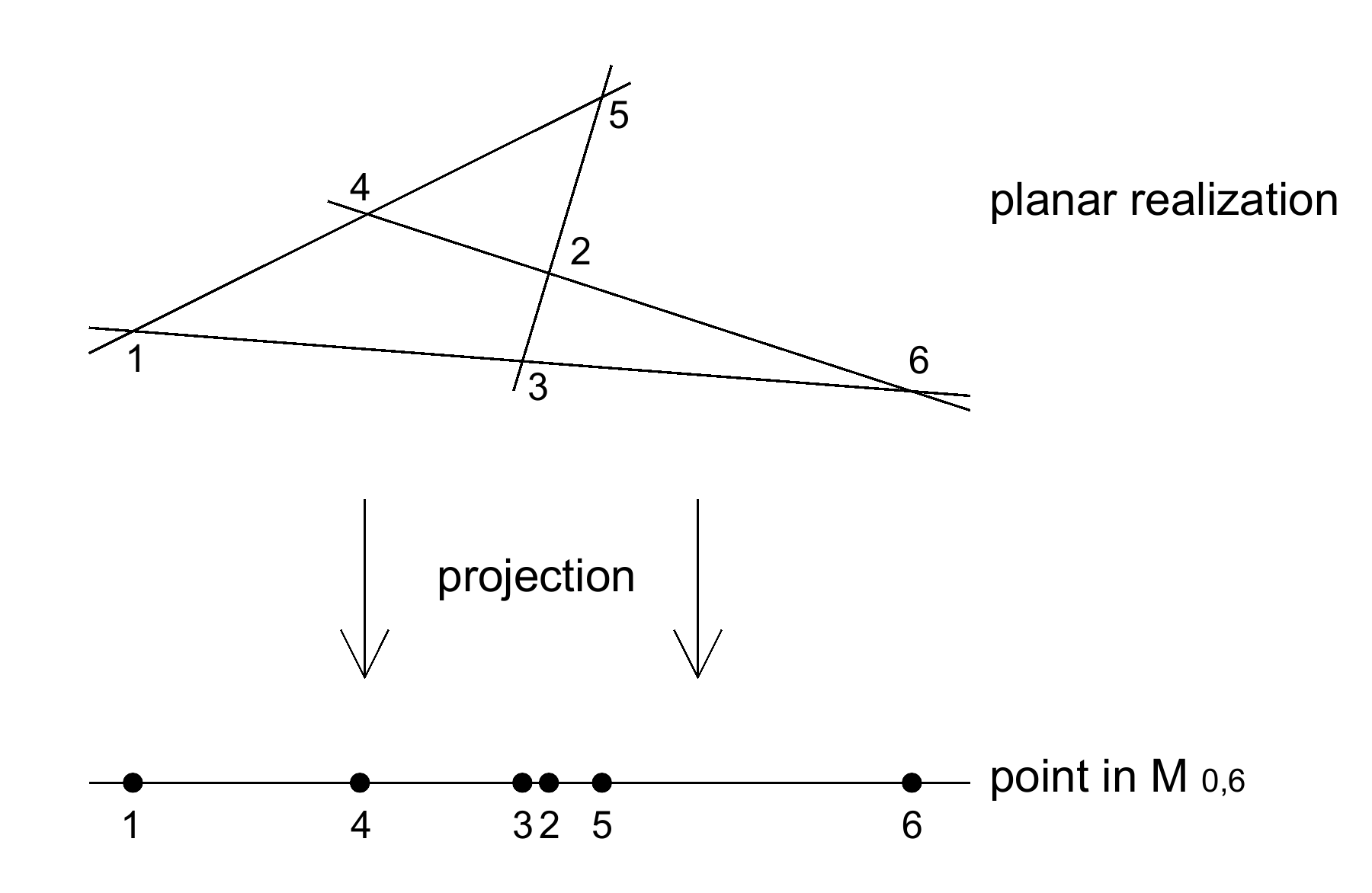}}
The reader will  notice a parallel with \ref{evwevwev}: the image $W_\Gamma:=\bLa^{-1}(D_\Gamma)\subset\Pic^{\vec d}C$
is  the analogue of the planar locus $W\subset\Pic^{g+1}C$ of a smooth MHV curve $C$ that parametrizes its presentations
as a plane curve of degree $g+1$. Like in the smooth case, the scattering amplitude form on $\Pic^{\vec d}C$
(when pulled back to $M_{0,n}$) vanishes along $D_\Gamma$ with multiplicity $2$ because
$W_\Gamma$ has codimension~$3$. 
This was noticed in  \cite{CT_Crelle} and \cite{MHV},
which both contain (the same) determinantal equation for $D_\Gamma$.
\end{Review}

\begin{Remark}
The same stable curve can give many hypertrees by changing a vector $\vec d$ of multidegrees.
It is a non-trivial question to understand 
all MHV curves with the same underlying stable maximally degenerate curve $C$.
\end{Remark}

We finish this section by studying geometry of spherical CT hypertrees.
Recall that  every checkerboard triangulation of a $2$-sphere as in Figure~\ref{triangpic}
gives a CT hypertree:
vertices of the triangulation give the indexing set $\{1,\ldots,n\}$ 
and black triangles give triples (another CT hypertree is given by  white triangles.)

\begin{Theorem}\label{sVSGsrg}
Let $\Gamma$ be a CT hypertree and let $C$ be the corresponding maximally degenerate stable MHV curve.
Then $\Gamma$ is spherical if and only if
$C$ does not have a $2$-channel factorization (see Definition~\ref{FbfhFHF}) and
admits a real algebraic  structure such that
\begin{enumerate}
\item $C$ is a stable limit of smooth pointed M-curves $C_t$ of type A (see Definition~\ref{SfbXfbSDb}).
Let ~$C(\bR)=W_1\cup\ldots\cup W_{g+1}$ be the union of images of $g+1$ ovals of $C_t(\bR)$.
\item No $W_i$ is an acnode, i.e.~a real node which is a  connected component of $C(\bR)$.
\item No $W_i$ is  contained in a degree $0$ component of $C$ (white circle of the on-shell diagram). Equivalently,
$W_i$ is not contracted to a point by the morphism $C\to\Sigma$.
\end{enumerate}
\end{Theorem}

\begin{figure}[htbp]
\includegraphics[width=\textwidth]{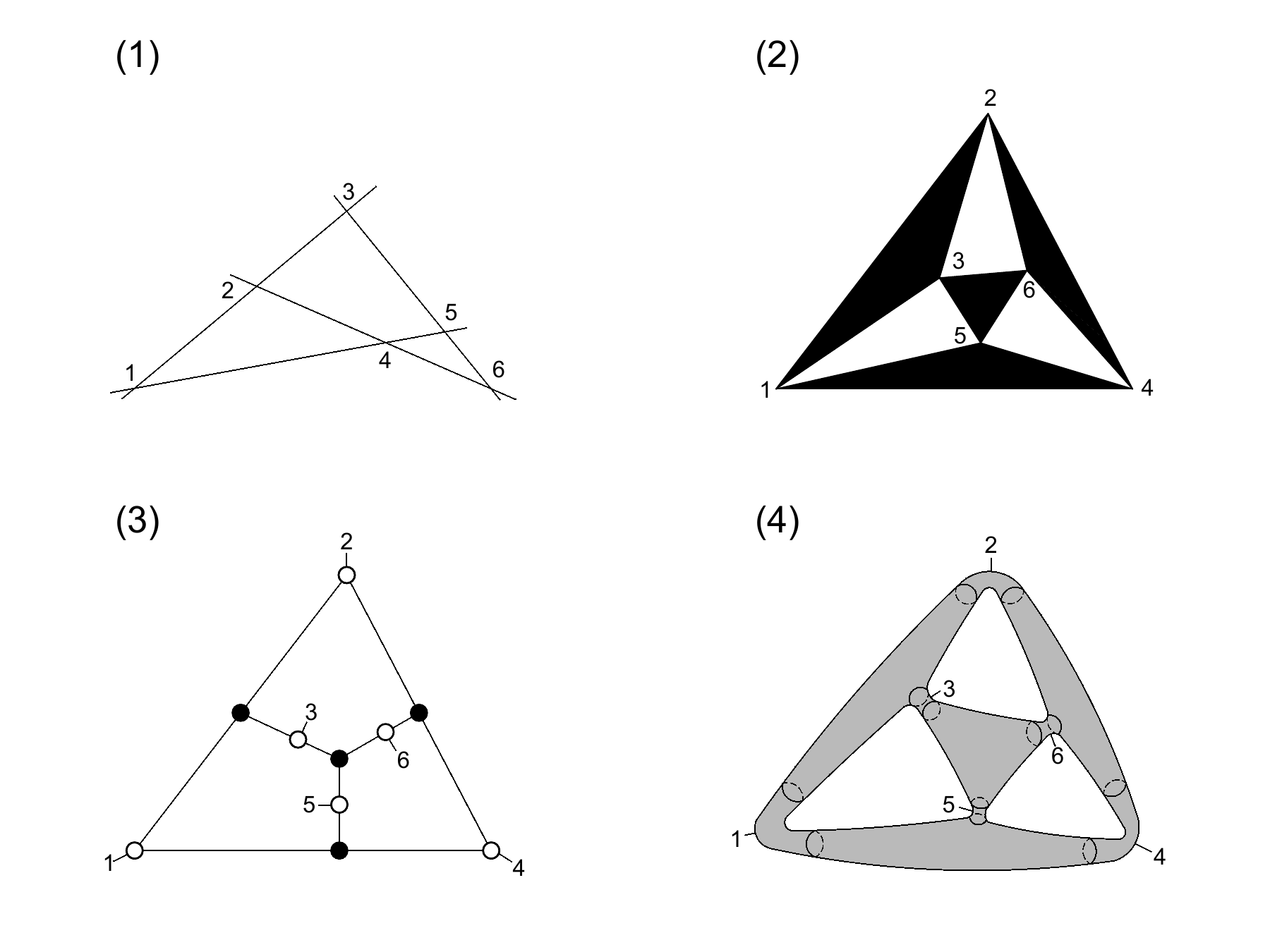}
\caption{\small A hypertree curve (1) of a spherical CT hypertree (2) with an on-shell diagram (3)  is a degeneration of a real algebraic M-curve (4).}\label{wertv23f}
\end{figure}

\begin{proof}
Let $C$ be a maximally degenerate stable MHV curve without $2$-channel factorization
that satisfies (1)--(3). By Theorem~\ref{hjgcjghcj}, it corresponds to a CT hypertree $\Gamma$.
We need to show that $\Gamma$ is spherical.
According to~\cite{Seppala}, we can obtain $C$ as follows\footnote{
More precisely, \cite{Seppala} is restricted to curves without marked points. But there is a  trick, a  morphism
from $\oM_{g,n}$ to the ``flag stratum'' of $\oM_{g+n}$, which attaches a nodal genus $1$ curve to every marking.
All~results that we need about the stable reduction follow from the corresponding results for $\oM_{g+n}$.}.
Start with a smooth pointed M-curve $C_t$ of type~$A$. Topologically, $C_t$ is obtained by gluing two identical discs $D$ and $\bar D$
with $g$ holes removed. In $C_t$, these disks are  interchanged by the complex conjugation. They
are glued along
the real locus $C(\bR)=D\cap\bar D$, which is the union of the outside oval of $D$ and $g$ inside ovals. 
Of course we can turn any oval into an ``outside'' oval by turning $D$ inside out. But there is a good choice:
we~choose the oval with $3$ marked points as an outside one and the remaining ovals (with
one marked point each) as the inside ovals. 

Next we fix a hyperbolic metric on $C_t\setminus\{p_1,\ldots,p_n\}$, 
remove small conjugation-invariant 
discs around marked points with geodesic boundaries and choose a conjugation-invariant pair of pants decomposition 
of the remaining surface with geodesic boundaries.
The stable curve $C$ is obtained by shrinking the boundaries of pairs of pants to nodes  except for the boundaries around marked points.

We investigate the structure of this pair of pants decomposition.
Since $C$ does not contain acnodes by condition (2), no real oval of $C_t$ is the boundary of a pair of pants.
We claim that no boundary $B$ is entirely contained in $D$ (or $\bar D$) either.
Indeed, if this were the case then by Jordan's lemma $B$ separates $D$ into two connected components.
Shrinking $B$ (and $\bar B$) creates a pair of complex-conjugate nodes of $C$.
Neither of these nodes is a separating node because $C$ does not admit a one-channel factorization by Lemma~\ref{qwefv2e2}.
Thus these nodes separate $C$ into two connected components.
But this contradicts the fact that $C$ does not admit a two-channel factorization\footnote{
Interestingly, if this two-channel factorization by complex-conjugate nodes does occur,
the resulting curve is of the easy type (I) completely described in Theorem~\ref{safbsfhsfn}.} and satisfies condition (3).

It follows that every pair of pants, a Riemann sphere with three  discs removed, is in fact conjugation-invariant,
moreover isomorphic to $\bP^1(\bC)$ with the usual complex conjugation and three removed discs are centered at three points of $\bP^1(\bR)$.
The quotient of $C_t$ with the pair of pants decomposition by the complex conjugation is the original disc $D$ with $g$ holes
and the quotient of each pair of pants is a curved hexagon with alternating sides that are either arcs of real ovals of $C(\bR)$ or arcs connecting the ovals.
The pairs of pants arising from the marked points look as follows: an arc around the marked point connecting two points of the same real oval $Z$,
then two arcs of $Z$, then two sides of the hexagon connecting $Z$ to an oval~$Z'$, then an arc of $Z'$.
Other pairs of pants connect three real ovals.
We claim that in fact in the first case $Z\ne Z'$ and in the second case all three real ovals are different.
Indeed, if one of the sides of the hexagon (not around the marked point) connects a real oval to itself then
the corresponding border of the pair of pants is shrunk to a node of $C$ that separates $C$ into two connected components,
which is impossible by Lemma~\ref{qwefv2e2}.
We illustrate the decomposition of $D$ into hexagons in Figure~\ref{asgwrg}, where we shade in gray components of $C$ of degree $0$.
\begin{figure}[htbp]
\includegraphics[width=0.45\textwidth]{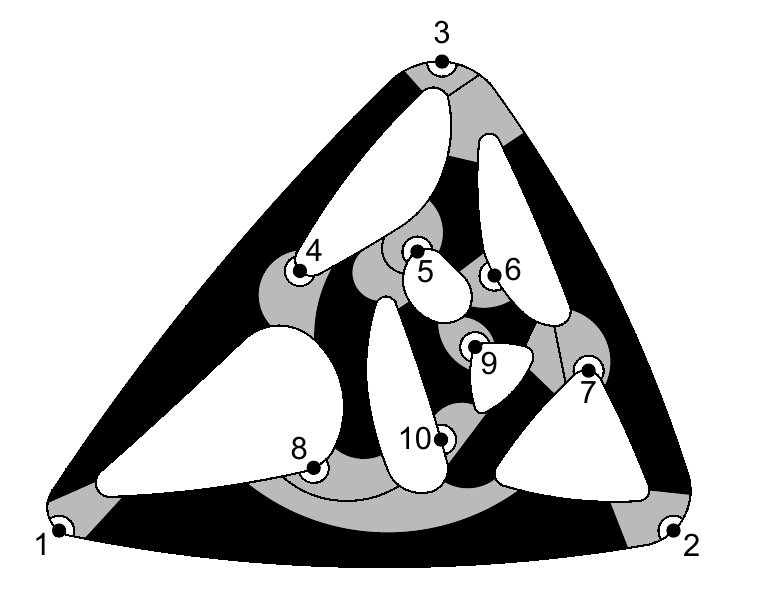}
\includegraphics[width=0.45\textwidth]{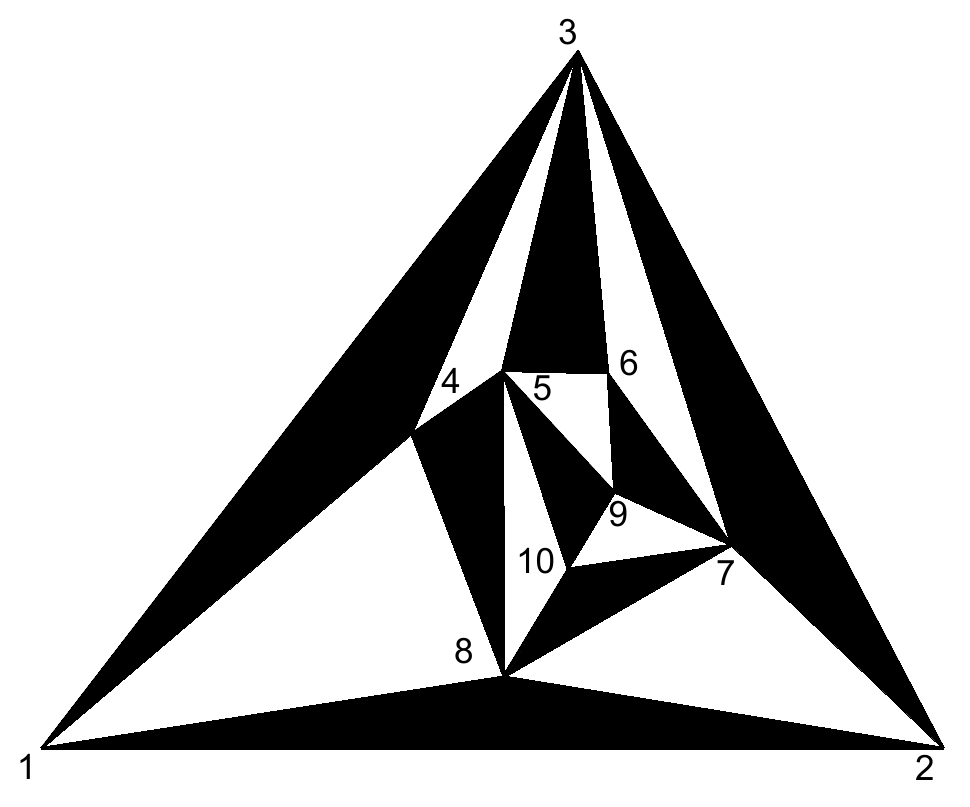}
\caption{\small Decomposition of the disk $D$ and the corresponding triangulation.}\label{asgwrg}
\end{figure}

Finally, we produce a checkerboard decomposition of the sphere, as follows.
We shrink pairs of pants arising from the marked points to points. Furthermore, we shrink hexagons that correspond to other degree $0$ components of $C$ to points.
Furthermore, we shrink arcs connecting real ovals of degree $1$ components of~$C$ to points. This gives $g+1$ black triangles.
We glue $g+1$ white polygons to the real ovals of $C$ along the arcs of black triangles. 
This gives a checkerboard decomposition of the sphere into $g+1$ black triangles and $g+1$ white polygons.
Apriori, some of these white polygons may have two sides, however if that's the case then two adjacent black triangles
will produce a two-channel factorization of~$C$, which is impossible. Thus all white polygons have at least three edges, and therefore
they are all triangles since the total number of edges of black and white triangles is the same.
Therefore, the CT hypertree is spherical. We illustrate the pair of pants decomposition 
and the corresponding triangulation in Figure~\ref{asgwrg}.

It remains to prove that every spherical CT hypertree gives a maximally degenerate stable MHV curve
that can be obtained as a limit of smooth pointed M-curves $C_t$ of type A.
The construction of the conjugation-invariant pair of pants decomposition should be obvious at this point.
The only non-trivial remaining part of the argument is to explain how to distribute points among the real ovals of $C_t(\bR)$.
Concretely, we start with a checker-board triangulation of the sphere, attach three marked 
points to the ``outside'' white triangle and now have to distribute the remaining $g$ marked points in $g$ inside white triangles
so that every triangle contains exactly one marked point as in Figure~\ref{asgwrg}. In other words, we have to construct a perfect matching between
non-exterior vertices of the triangulation and interior white triangles. This  non-trivial matching has been constructed 
by Tutte in his famous Trinity Theorem \cite{Tutte}.
The~algorithm goes as follows (see Figure~\ref{sfvwfeb}).

\begin{figure}[htbp]
\includegraphics[width=0.45\textwidth]{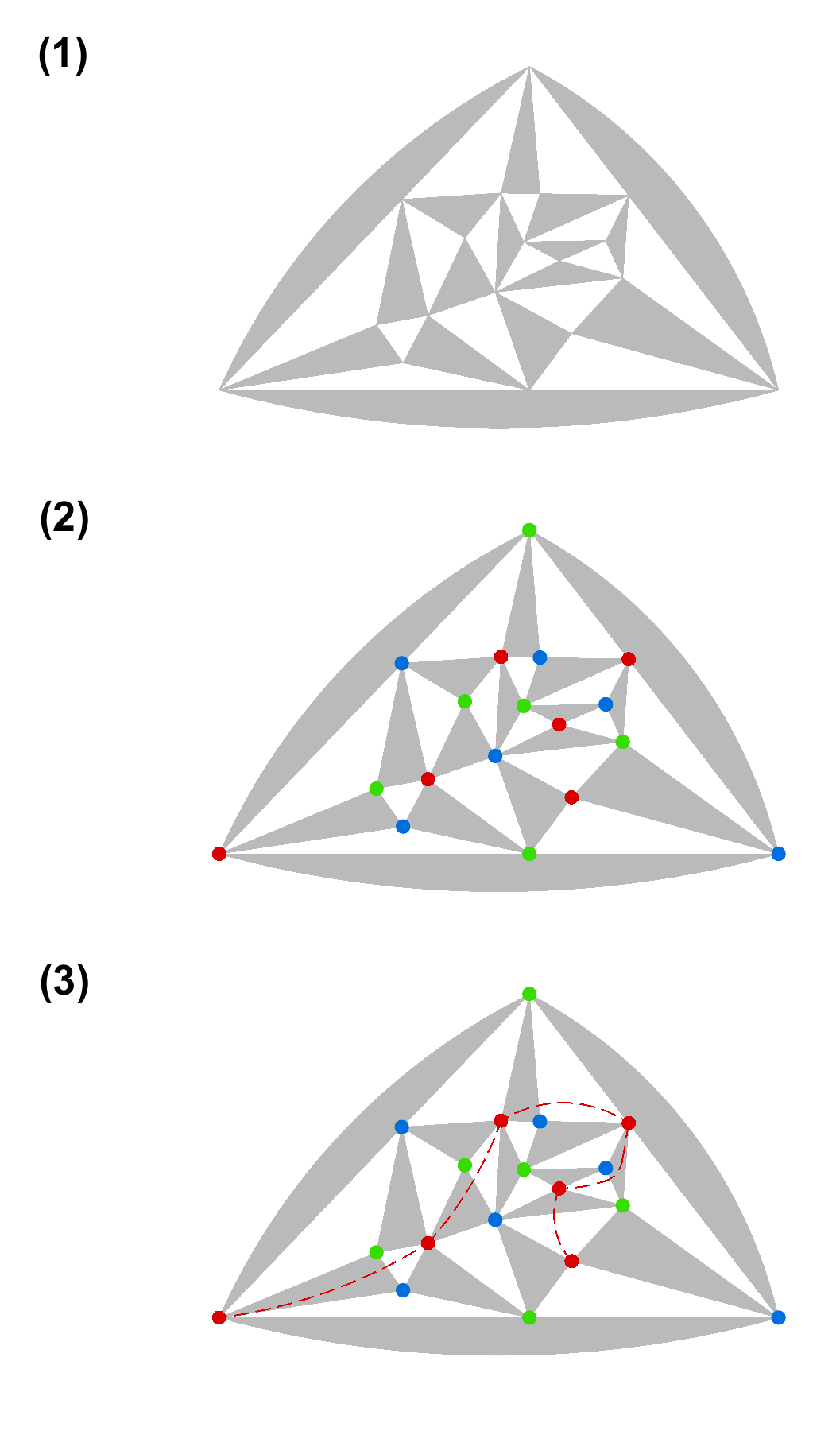}
\includegraphics[width=0.45\textwidth]{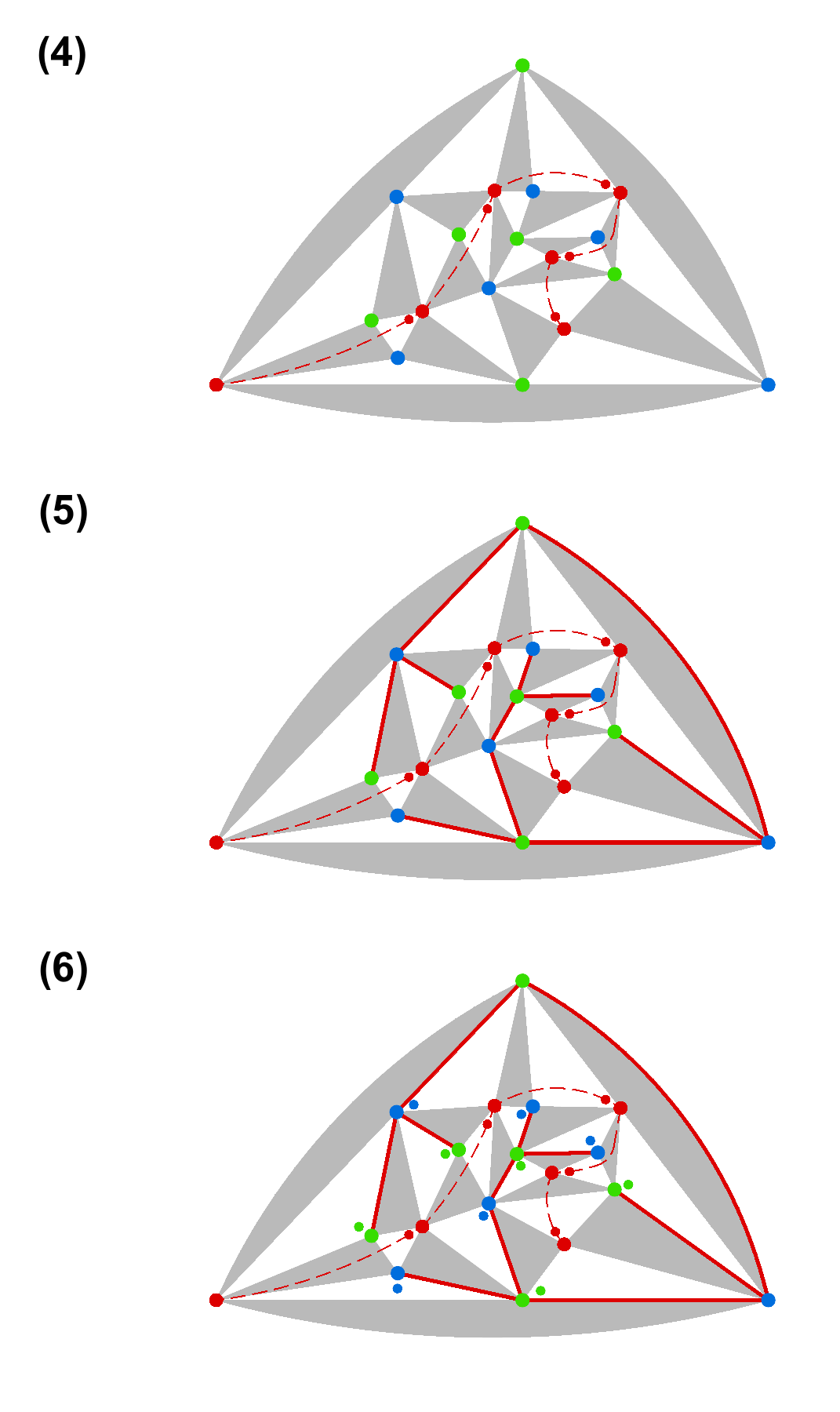}
\caption{Tutte matching algorithm \cite{Tutte}}\label{sfvwfeb}
\end{figure}

\begin{enumerate}
\item Start with any checker-board triangulation.
\item It is easy to see that the vertices of the triangulation can be colored red, green and blue so that every triangle
has vertices of all colors.
\item There always exists at least one red {\em arborescence}: a directed connected tree connecting an outside red vertex to all inside red vertices
with one arrow pointing into each inside red vertex. Every arrow first crosses a black and then a white triangle.
In our illustration the red arborescence is just a path but it can of course be a more complicated tree.
\item At this point we can already construct a partial Tutte matching: match interior red vertices to white triangles right before
them in the arborescence. We indicate this by small red dots in these white triangles.
\item Construct another connected red tree by connecting blue and green vertices by edges which are not intersected by the red arborescence.
\item Shrink the red tree leaf-by-leaf and on every step match a vertex from which the leaf is removed to a white triangle adjacent to the leaf.
We indicate the matching by small green and blue dots  in these white triangles. 
\end{enumerate}
This completes construction of the  matching and the proof of the Theorem.
\end{proof}

\section{Further remarks}\label{lastsection}

Returning to the general case of factorization $\omega_C(p_1+\ldots+p_n)=L\otimes\tilde L$,
 we write 
$d=\deg L=g+k-1$,
where $k$ is an expected dimension of $H^0(C,L)$, known in physical context as helicity.
When $k>2$, the line bundle $L$ itself does not determine a map $C\to\bP^1$,
one also needs to choose a two dimensional subspace in $H^0(C,L)$, i.e.~a pencil of divisors. There are several ways to turn this set-up into a problem
about geometry of linear systems.
One approach is to consider a diagram
$$
\begin{CD}
\cG  @>\bLa>>  M_{0,n} \\
@VVV    \\
\Pic^{d}C                     
\end{CD}
$$
where all arrows are rational maps, the vertical map is generically a Grassmannian fibration with fiber $G(2,H^0(C,L))$,
and the horizontal map sends a pencil into images of marked points in $\bP^1$.
While there is no natural probability distribution on $\cG$, one can still ask about properties of the map $\bLa$, for example about
its degree if it is generically finite, which happens if 
$g+2(k-2)=n-3$, i.e.~$n=g+2k-1$. 

Another approach is to consider maps $\phi_L:\,C\to\bP^{k-1}$
given by complete linear systems instead of maps to $\bP^1$ given by pencils.
This gives a rational map 
$$\Pic^{d}C\mathop{\dashrightarrow}^\Lambda X(k,n),$$
where $X(k,n)$ is the moduli space of $n$ points in $\bP^{k-1}$ modulo $\PGL_k$, see papers \cite{KT,ST,Tev} for the background on this space and its compactifications. With this definition,
we can again view the volume form on $\Pic^{d}C$ as a multi-valued meromorphic scattering amplitude form on $X(k,n)$, at least if the map $\Lambda$ is generically finite, which happens if $g=n(k-1)-k^2+1$, i.e.~if $n={g\over k-1}+k+1$.
Another advantage is that the duality between bundles $L$ and $\tilde L$ is manifest: we have 
$$\tilde d=\deg \tilde L=(2g-2+n)-(g+k-1)=g+(n-k)-1=g+\tilde k-1,$$
where $\tilde k=n-k$. In fact we have a commutative diagram 
$$
\begin{CD}
\Pic^{d}C  @>\simeq>>  \Pic^{\tilde d}C \\
@V\Lambda VV  @VV{\tilde\Lambda}V   \\
X(k,n) @>>\simeq> X(n-k,n)                     
\end{CD}
$$
with rational vertical arrows.
Indeed, the Gale duality  $X(k,n)\simeq X(n-k,n)$
corresponds to the pairing $(U\subset V)\longleftrightarrow (U^\perp\subset V^*)$ between subspaces of $\bC^n$ and global sections of $L$ and $\tilde L$ are indeed perpendicular due to the residue theorem.

\begin{Example}
The first non-MHV case is the following ``666 Puzzle''. Let $C$ be a genus $4$ curve with $6$ marked points $p_1,\ldots p_6$. 
A choice of a degree $6$ line bundle $L\in\Pic^6C$ realizes $C$ as a $6$-nodal sextic in $\Bbb P^2$ with $6$ marked points. 
This gives a generically finite rational map $\Lambda$ from $\Pic^6C$ to $X(3,6)$, which has dimension~$4$. 
What is the degree of $\Lambda$? An easier case for $k=3$ would be to consider nodal genus~$2$ quartics in $\Bbb P^2$ with $5$ marked points but by the Gale duality this case is equivalent to the MHV case $g=k=2$, so the degree of $\Lambda$ will be $4$.
\end{Example}

In the rest of this section we  show that, in the maximally degenerate case, our scattering amplitude form on $X(k,n)$ 
(which is isomorphic to $M_{0,n}$ in the MHV case $k=2$)
is equivalent to the leading singularity form of the scattering amplitude
on the Grassmannian $G(k,n)$ as described in \cite{Grass}. It~is convenient to pass to the $(\bC^*)^{n-1}$-torsor $\widehat{\Pic}\ C\to\Pic C$
which parametrizes line bundles $L$ along with trivializations 
$t_i:\,L|_{p_i}\simeq\bC$ for $i=1,\ldots,n$. This automatically gives trivializations $\tilde t_i:\,\tilde L|_{p_i}\simeq\bC$
such that $t_i\otimes\tilde t_i$ is the canonical trivialization  $\omega_C(p_1+\ldots+p_n)|_{p_i}\simeq\bC$ given by the residue.
The choice of global sections $s_\alpha\in H^0(C,L)$, $\tilde s_{\tilde\alpha}\in H^0(C,\tilde L)$ for $\alpha,\tilde\alpha=1,2$
gives spinor 2-vectors  $\lambda_i=t_i(s_\alpha|_{p_i})$ and $\tilde\lambda_i=\tilde t_i(\tilde s_\alpha|_{p_i})$
such that each (external) momentum can be written as  $\bold p_i=\lambda_i\tilde\lambda_i^T$.
The change of trivializations corresponds to the action of the little torus $(\bC^*)^n$ by 
$\lambda_i\mapsto t_i\lambda_i$ and $\tilde\lambda_i\mapsto t_i^{-1}\tilde\lambda_i$.

We have a commutative diagram  
$$
\begin{CD}
\widehat{\Pic}^{\vec d}C && \mathop{\dashrightarrow}\limits^{\hat\Lambda} && G(k,n)  \\             
@VVV    && @VV\pi V\\
\Pic^{\vec d}C && \mathop{\dashrightarrow}\limits^\Lambda && X(k,n)    \\           
\end{CD}
$$
Here $\hat\Lambda(L,t_1,\ldots,t_n)$ is the row space of the matrix $\left(t_i(s_j|_{p_i})\right)_{i=1,\ldots,n\atop j=1,\ldots,k}$,
where $s_1,\ldots,s_k$ is a basic of $H^0(C,L)$, and 
$\pi$ is a (rational) quotient map by the little torus acting on the Grassmannian. 
Being a torsor over $\Pic^{\vec d}C$, $\widehat{\Pic}^{\vec d}C$
also carries a canonical volume form and writing it down as a multi-valued form on $G(k,n)$ in Pl\"ucker coordinates is equivalent to writing a scattering amplitude form on $X(k,n)$. 

We specialize to the maximally degenerate case. Let $v$ be the number of irreducible components of $C$, i.e.~the number of vertices
of the on-shell diagram, and let $i$ be the number of internal edges. Then
$\dim\widehat{\Pic}^{\vec d}C=g+n-1=2v-i$. Furthermore, 
 $\widehat{\Pic}^{\vec d}C$ has a presentation as a quotient torus
$$\widehat{\Pic}^{\vec d}C\simeq\widehat{\Pic}^{\vec d}C^\nu/(\bC^*)^{i}\simeq(\bC^*)^{2v}/(\bC^*)^{i}$$ 
defined as follows.
The restriction of $L$ to every irreducible component of $C$ (equivalently, every connected component of the normalization $C^\nu$)
is either $\cO_{\bP^1}(1)$ (for black vertices) or $\cO_{\bP^1}$ (for white vertices). In either case,
choose its trivialization at three special points of that component 
and use these trivializations to glue a line bundle on $C$ from the line bundle on the normalization.
The action of the torus $(\bC^*)^{i}$ comes from simultaneously rescaling  trivializations at two points of $C^\nu$
that glue to the same node of $C$.  The standard $\dlog$  form on $(\bC^*)^{2v}/(\bC^*)^{i}$
then gives the formula \cite[(4.41)]{Grass} for the leading singularity of the scattering amplitude.
Various delta-functions in this formula 
simply encode the fact that choosing the spinor variables $\lambda_i$ and $\tilde\lambda_i$
corresponds to choosing a $2$-dimensional subspace in $H^0(C,L)$ (resp.~in the Gale-dual $H^0(C,\tilde L)$ -- orthogonal of $H^0(C,L)$).

\end{document}